\documentclass[11pt]{amsart}

\usepackage{eucal}
\usepackage{graphicx}
\usepackage{color}
\usepackage{amsopn,amssymb, relsize, bm, mathrsfs,tikz}
\usetikzlibrary{fit,matrix, patterns}
\usepackage{paralist}

\usepackage[T1]{fontenc}

\usepackage[margin=1.25in]{geometry} 

\usepackage{enumitem}
\usepackage{csquotes}

\definecolor{verydarkblue}{rgb}{0,0,0.4}
\usepackage{hyperref}

\usepackage{caption}
\graphicspath{{figures/}}

\makeatletter%
\def\@commafont{\check@mathfonts
    \fontsize\sf@size\z@\selectfont}
\DeclareRobustCommand{\cb}[1]{{%
\setbox\z@\hbox{#1}%
\ifdim\dp\z@<.1\ht\z@\ooalign{\unhbox\z@\crcr\hidewidth\lower.3ex\hbox{\@commafont,}\hidewidth}%
\else\ooalign{\unhbox\z@\crcr\hidewidth\raise.5ex\hbox{\@commafont`}\hidewidth}%
\fi}}%
\makeatother

\newcommand{\Z}{\mathbb{Z}}
\newcommand{\R}{\mathbb{R}}
\newcommand{\C}{\mathbb{C}}
\newcommand{\N}{\mathbb{N}}

\DeclareMathOperator{\Hopf}{Hopf}

\renewcommand{\Im}{\im}

\newcommand{\CP}{\mathbb{CP}}
\renewcommand{\H}{\mathbb{H}}

\renewcommand{\tilde}[1]{\widetilde{#1}}

\newcommand{\g}{\mathfrak{g}}

\newcommand{\D}{\mathcal{D}}

\newcommand{\tec}{Teichm\"uller }

\newcommand{\SF}{\mathcal{F}}

\newcommand{\sJ}{\mathcal{J}}
\newcommand{\sG}{\mathcal{G}}

\newcommand{\Om}{\Omega}

\newcommand{\End}{\mathrm{End}}
\newcommand{\SL}{\mathrm{SL}}
\newcommand{\PSL}{\mathrm{PSL}}
\newcommand{\SO}{\mathrm{SO}}
\newcommand{\QD}{\mathcal{QD}}
\newcommand{\SQD}{\mathcal{SQD}}

\newcommand{\app}{\mathrm{app}}
\newcommand{\Xhat}{\widehat{X}}

\newcommand{\Ahat}{\widehat{A}}
\newcommand{\khat}{\widehat{k}}

\renewcommand{\leq}{\leqslant}
\renewcommand{\geq}{\geqslant}

\renewcommand{\hat}{\widehat}

\renewcommand{\Im}{\mathrm{Im}}

\newcommand{\Stab}{\operatorname{Stab}}

\newcommand{\Prym}{\mathrm{Prym}}
\newcommand{\SW}{\mathrm{SW}}

\newcommand{\map}[3]{#1\colon#2\rightarrow#3}
\newcommand{\deq}{\mathrel{\mathop:}=}
\newcommand{\Xtilde}{\widetilde{X}}

\newcommand{\alphabold}{{\bm\alpha}}
\newcommand{\betabold}{{\bm\beta}}

\newcommand{\sigmabold}{{\bm\sigma}}

\newcommand{\CBbb}{\mathbb C}
\newcommand{\DBbb}{\mathbb D}

\newcommand{\HBbb}{\mathbb H}

\newcommand{\NBbb}{\mathbb N}

\newcommand{\RBbb}{\mathbb R}

\newcommand{\ZBbb}{\mathbb Z}

\newcommand{\Acal}{\mathcal A}

\newcommand{\Ccal}{\mathcal C}

\newcommand{\Ecal}{\mathcal E}
\newcommand{\Fcal}{\mathcal F}

\newcommand{\Hcal}{\mathcal H}
\newcommand{\Ical}{\mathcal I}

\newcommand{\Lcal}{\mathcal L}

\newcommand{\Ocal}{\mathcal O}
\newcommand{\Pcal}{\mathcal P}
\newcommand{\Qcal}{\mathcal Q}

\newcommand{\Scal}{\mathcal S}

\newcommand{\Ucal}{\mathcal U}
\newcommand{\Vcal}{\mathcal V}

\newcommand{\Xcal}{\mathcal X}

\newcommand{\gfrak}{\mathfrak g}

\newcommand{\SU}{\mathsf{SU}}

\newcommand{\Psf}{\mathsf{P}}

\newcommand{\Hscr}{\mathscr{H}}
\newcommand{\Sscr}{\mathscr{S}}
\newcommand{\Dscr}{\mathscr{D}}
\newcommand{\Pscr}{\mathscr{P}}

\DeclareMathOperator{\Hom}{Hom}

\DeclareMathOperator{\id}{id}
\DeclareMathOperator{\imag}{Im}
\DeclareMathOperator{\real}{Re}

\DeclareMathOperator{\tr}{tr}

\DeclareMathOperator{\gr}{gr}

\DeclareMathOperator{\Iso}{Iso}

\newcommand{\dbar}{\bar\partial}

\newcommand{\lra}{\longrightarrow}
\newcommand{\doubleslash}{\bigr/ \negthinspace\negthinspace \bigr/}

\newcommand{\Pic}{{\rm Pic}}

\newcommand{\Op}{{\rm Op}}

\DeclareMathOperator{\NAH}{\mathsf{NAH}}
\DeclareMathOperator{\RH}{\mathsf{RH}}

\DeclareMathOperator{\B}{\mathsf{B}}
\DeclareMathOperator{\sfP}{\mathsf{P}}
\DeclareMathOperator{\DR}{\mathsf{DR}}
\DeclareMathOperator{\Higgs}{\mathsf{H}}

\DeclareMathOperator{\LC}{\mathsf{LC}}
\DeclareMathOperator{\HF}{\mathsf{HF}}

\DeclareMathOperator{\odd}{\mathsf{odd}}
\DeclareMathOperator{\ev}{\mathsf{ev}}
\newcommand{\sbt}{\,\begin{picture}(-1,1)(-1,-1)\circle*{2}\end{picture}\ }
\newcommand{\dt}[1]{\overset{\sbt}{#1}}

\newcommand{\isorightarrow}{\xrightarrow{
   \,\smash{\raisebox{-0.5ex}{\ensuremath{\sim}}}\,}}

\newcommand{\param}{{\mathchoice{\mkern1mu\mbox{\raise2.2pt\hbox{$\centerdot$}}\mkern1mu}{\mkern1mu\mbox{\raise2.2pt\hbox{$\centerdot$}}\mkern1mu}{\mkern1.5mu\centerdot\mkern1.5mu}{\mkern1.5mu\centerdot\mkern1.5mu}}}

\numberwithin{equation}{section}

\theoremstyle{plain}
\newtheorem{thm}{Theorem}[section]
\newtheorem{cor}[thm]{Corollary}

\newtheorem{lem}[thm]{Lemma}
\newtheorem{prop}[thm]{Proposition}

\newtheorem{remlabel}[thm]{Remark}
\newtheorem{definition}[thm]{Definition}
\newtheorem{exlabel}[thm]{Example}

\newtheorem*{mainthm}{Main Theorem}

\theoremstyle{definition}

\theoremstyle{definition}

\newtheorem*{remarksenv}{Remarks}

\begin{document}

\title[Higgs Bundles, Harmonic Maps,  and Pleated Surfaces]{Higgs Bundles, Harmonic Maps,  and Pleated Surfaces}

\author[A.\ Ott]{Andreas Ott}
\address{Mathematisches Institut,
    Ruprecht-Karls-Universit\"at Heidelberg,
Im Neuenheimer Feld 205, 
D-69120 Heidelberg,
  Germany}
  \email{aott@mathi.uni-heidelberg.de}

\author[J.\ Swoboda]{Jan Swoboda}
\address{Mathematisches Institut,
    Ruprecht-Karls-Universit\"at Heidelberg,
Im Neuenheimer Feld 205, 
D-69120 Heidelberg,
  Germany}
  \email{swoboda@mathi.uni-heidelberg.de}

\author[R.\ Wentworth]{Richard Wentworth}
\address{Department of Mathematics,
   University of Maryland,
   College Park, MD 20742, USA}
   \email{raw@umd.edu}

\author[M.\ Wolf]{Michael Wolf}
\address{Department of Mathematics, Rice University,
Houston, TX 77251, USA}
\email{mwolf@rice.edu}

\begin{abstract} 
This paper unites the gauge-theoretic and hyperbolic-geometric perspectives on the asymptotic geometry of the character variety of $\SL(2,\CBbb)$ representations of a surface group. Specifically, we find an asymptotic correspondence between the analytically defined limiting configuration of a sequence of solutions to the $\SU(2)$ self-duality equations on a closed Riemann surface constructed by Mazzeo-Swoboda-Wei{\ss}-Witt, and the geometric topological shear-bend parameters of equivariant pleated surfaces in hyperbolic three-space due to Bonahon and Thurston. The geometric link comes from the nonabelian Hodge correspondence and a study of high energy degenerations of harmonic maps. Our result has several applications. We prove: (1) the local invariance of the partial compactification of the moduli space of solutions to the self-duality equations by limiting configurations; (2) a refinement of the harmonic maps characterization of the Morgan-Shalen compactification of the character variety; and (3) a comparison between the family of complex projective structures defined by a quadratic differential and the realizations of the corresponding flat connections as Higgs bundles, as well as a determination of the asymptotic shear-bend cocycle of Thurston’s pleated surface.
\end{abstract}

\date{\today}
\maketitle

\setcounter{tocdepth}{2}

\tableofcontents

\section{Introduction}

The purpose of this paper is to bring together two perspectives on the asymptotic structure of the $\SL(2,\CBbb)$ character variety of a surface group---the complex analytic perspective from algebraic geometry and nonlinear analysis, and the synthetic perspective from hyperbolic geometry and low dimensional topology.  The former finds its incarnation in a gauge-theoretic partial compactification of the moduli space of rank two Higgs bundles on a closed Riemann surface, via decouplings of the self-duality equations.  The latter is understood in terms of equivariant  pleated surfaces in hyperbolic three-space.

As a means of studying the asymptotics of the character variety, the analytic and the synthetic perspectives each have advantages and disadvantages.  The complex analytic perspective presents the character variety as a fibration over a $6g-6$ dimensional vector space of holomorphic differentials, and so presents the compactification as a fibration over a real projective space of just one less dimension.  Now, the emphasis here is holomorphic invariants, and so there is a built-in reference to a fixed Riemann surface.  This dependence of the compactification on an arbitrarily chosen Riemann surface renders the compactification unnatural from the point of view of the mapping class group action on the character variety.  
On the other hand, the synthetic perspective relies on a choice of lamination and so avoids issues of naturality, but the most celebrated compactification from this setting, the Morgan-Shalen compactification, has the far larger codimension $6g-5$, so entails a substantial loss of information. 

By uniting the perspectives, we provide a partial compactification that retains attractive features from both perspectives: it is topological, so that any dependence on an original choice of base surface has vanished by the frontier of the character variety, but it remains codimension one and so captures some of the nuance of the fibration.  Along the way, we explain the hyperbolic geometry of the gauge theory perspective, at least asymptotically. This relationship further allows us to locate the family of projective structures on a Riemann surface, again asymptotically, as a collection of nearly linear flows on the fibers at infinity.

The duality between the analytic and synthetic perspective on the $\SL(2,\CBbb)$ character variety has its roots in the identification of $\SL(2,\CBbb)$ (or rather its adjoint group $\PSL(2,\CBbb)=\SL(2,\CBbb)/\pm 1$) as the oriented isometry group of hyperbolic three-space.  The nonabelian Hodge correspondence gives a homeomorphism between the associated moduli spaces, and an important and long-standing direction of research is the study of how different geometric properties on both sides are related under this identification.  In this sense, the present work describes the asymptotics of the nonabelian Hodge correspondence from a geometric point of view.

Let us now explain some of the structures of the moduli space in more detail. 
The first is that of an algebraically completely integrable 
system, where a Higgs bundle is determined by a point in 
the Prym variety of the spectral curve associated to a holomorphic 
quadratic differential on the surface.  The second 
is in terms of solutions to the self-duality equations. Recent work proves 
asymptotic convergence, as the norm of the Higgs field diverges, to 
solutions of a decoupled  equation called a limiting configuration.  The 
data describing a limiting configuration is also a quadratic differential 
and a choice of Prym differential. The third aspect of the moduli space is 
 the link between solutions to the self-duality equations and 
flat $\SL(2,\CBbb)$ connections that is obtained from the existence 
of equivariant harmonic maps to $\HBbb^3$. 
Here, the quadratic differential appears as the Hopf differential 
of the map. Asymptotics of harmonic maps are well understood.
  While the relationship between spectral data and the Prym 
differential of the limiting configuration is transparent, 
it is perhaps not so clear how to recover this information 
from the asymptotic behavior of the harmonic map. 
This is where the fourth perspective intervenes; that 
of hyperbolic geometry and methods of Thurston.
The new ingredient linking  harmonic maps to limiting configurations
   is the notion of an (equivariant) pleated surface.
Pleated surfaces, and their monodromy representations, 
can be parametrized by shear-bend cocycles with respect to a 
maximal geodesic lamination on the surface. 
This shear-bend dichotomy is parallel to the asymptotic
decoupling of the self-duality equations which gives rise to
limiting configurations, and a major goal of this
paper is to make this similarity more concrete.

Roughly speaking, our main result proves an asymptotic correspondence 
between shear-bend coordinates and limiting 
configurations via periods of Prym differentials. 
More precisely, we prove the following.  First,
the energy of equivariant harmonic maps diverges for
a sequence $\rho_n$ of $\PSL(2,\CBbb)$-representations 
that leaves all compact sets in the character variety. 
We shall show that  for sufficiently large $n$, there exist 
$\rho_n$-equivariant pleated surfaces in $\HBbb^3$ where the
bending lamination is asymptotically close to the lamination associated to
the horizontal foliation of the Hopf differentials of the
harmonic maps.  The image of the harmonic map is itself close to this 
equivariant pleated surface, in an appropriate sense.  
Moreover, the shearing cocycle of the pleated surface 
is approximated projectively by the intersection number 
with the vertical foliation of the Hopf differential, and  
the asymptotic limit of the bending cocycle of the equivariant pleated surface is determined by
the periods of the Prym differential associated to the limiting
configuration, as described above.  This analysis has several applications.
First,  we find that a small change in the base Riemann surface 
used to define the moduli space of Higgs bundles
changes the data of the limiting configurations in the boundary 
associated to a sequence of representations by parallel
transport via the Gauss-Manin connection. In other words,
the analytically defined limiting configurations are topological. 
Next,  we obtain a partial refinement of 
the Morgan-Shalen compactification of the character variety.
Ideal points of this compactification are defined by the projective 
length functions of isometric actions of the surface group 
on $\RBbb$-trees.  The refinement decorates a tree in a portion of the compactification with a
bending cocycle which provides more precise information on the
relationship between the  limiting length functions and dual trees
of measured foliations. 
Finally,  we  determine the asymptotic shear-bend cocycles of the equivariant pleated
surfaces that Thurston associates to  complex projective structures
 (or \emph{opers}). 
The result states that these cocycles are  asymptotic to the
ones defined  by the ``Seiberg-Witten
differential'' on the spectral curve, which itself is  an important device
that figures prominently in the WKB analysis of the
differential equation defining the projective structure. 

Before formulating the precise  statements of these results in
\textsection{\ref{sec:intro-results} below, 
we first provide some notation and important terminology.

\newpage 

\subsection{Preliminaries}

\subsubsection{Notation} \label{sec:notation}
Throughout this paper,  $\Sigma$ will be a fixed closed, oriented surface of genus $g\geq 2$ with fundamental group $\pi_1=\pi_1(\Sigma, p_0)$,  where $p_0\in \Sigma$ is a fixed base point.  We will typically denote  a marked Riemann surface structure on $\Sigma$ by $X$, and a marked hyperbolic structure on $\Sigma$ by $S$. The
 almost complex structure on $X$ appears as $J$. 
 Universal covers are denoted by $\widetilde\Sigma$, etc., and  $\tilde p_0$ will indicate a fixed lift of $p_0$ to $\widetilde\Sigma$. The $2$ and $3$ dimensional real
 hyperbolic  spaces are written $\HBbb^2$ and $\HBbb^3$, respectively.
 Notice that $\HBbb^2$ (\emph{resp}.\ $\HBbb^3$) carries a left $\SL(2,\RBbb)$ (\emph{resp}.\ $\SL(2,\CBbb)$) action by isometries
 that factors through $\PSL(2,\RBbb)=\SL(2,\RBbb)/\{\pm 1\}$
  (\emph{resp}.\ $\PSL(2,\CBbb)=\SL(2,\CBbb)/\{\pm 1\})$.
The canonical bundle of $X$ is indicated by $K_X$. 
We denote by $\QD(X)$  the space of holomorphic quadratic differentials
 on $X$, and by $\QD^{\ast}(X)\subset\QD(X)$ the cone of differentials with only simple zeroes. We let $\SQD(X)\subset\QD(X)$ denote the unit differentials with respect to some norm, and $\SQD^\ast(X)=\SQD(X)\cap\QD^\ast(X)$.
The Teichm\"uller space $T(X)$ of $X$ will  sometimes be labeled by
 $T(\Sigma)$ when identifying it
  with the Fricke space of discrete $\PSL(2,\RBbb)$ representations.

\subsubsection{Moduli spaces} \label{sec:intro-moduli} 
Define the {\bf Betti moduli space}
\begin{equation} \label{eqn:betti}
M_{\B}(\Sigma):= \Hom(\pi_1,\SL(2,\CBbb))\doubleslash \SL(2,\CBbb)
\end{equation}
parametrizing conjugacy classes of semisimple representations of the fundamental group of $\Sigma$. The choice of $\SL(2,\CBbb)$ rather than $\PSL(2,\CBbb)$ is really a matter of convenience. 
For simplicity we work with $\SL(2,\CBbb)$, but
 we caution that it will be important at several points to track the distinction between the special and projective groups. Hence, let us also define the character variety 
 \begin{equation} \label{eqn:rep-variety}
R(\Sigma):= \Hom(\pi_1,\PSL(2,\CBbb))\doubleslash \PSL(2,\CBbb).
\end{equation}
We shall be interested in the connected component of the
trivial representation,  $R^o(\Sigma)$, consisting of
representations that  lift to $\SL(2,\CBbb)$. This can be
realized as a quotient: $R^o(\Sigma)=M_{\B}(\Sigma)/J_2(\Sigma)$, where 
$
J_2(\Sigma):=\Hom(\pi_1, \{\pm 1\})
$.
The results of this paper  apply with little change to the other component of $R(\Sigma)$ consisting of representations that do not lift to $\SL(2,\CBbb)$. 

Let $M_{\DR}(\Sigma)$ denote the {\bf de Rham moduli space} of completely reducible flat $\SL(2,\CBbb)$ connections on $\Sigma$. Since $\dim_\CBbb X=1$, a holomorphic connection on a vector bundle on $X$ is automatically flat, and so 
we have a canonical identification of $M_{\DR}(\Sigma)$ with $M_{\DR}(X)$, the moduli space of rank $2$ holomorphic connections. We will sometimes confuse the two when the Riemann surface structure is understood. 
The Riemann-Hilbert correspondence gives a homeomorphism 
\begin{equation} \label{eqn:riemann-hilbert}
\RH : M_{\DR}(\Sigma)\isorightarrow M_{\B}(\Sigma)\ ,
\end{equation}
obtained by associating to a flat connection $\nabla$ its monodromy representation $\rho$. Gauge equivalent connections give conjugate representations.

Let $M_{\Higgs}(X)$ denote the {\bf Hitchin moduli space} of rank $2$ {Higgs bundles} consisting of isomorphism classes of pairs $(\Ecal, \Phi)$, where $\Ecal\to X$ is a rank $2$ holomorphic vector bundle on $X$ with fixed trivial determinant, and $\Phi$ is a holomorphic Higgs field. 
The {\bf nonabelian Hodge correspondence} gives a homeomorphism 
\begin{equation} \label{eqn:nah}
\NAH :  M_{\DR}(X)\isorightarrow M_{\Higgs}(X)\ .
\end{equation}
This map is described in more detail in \textsection{\ref{subsect:hitchineqharmmaps}}.

 There is a proper holomorphic map $\Hscr: M_{\Higgs}(X)\to \QD(X)$.  The fiber $\Hscr^{-1}(q)$, $q\in \QD^{\ast}(X)$ may be identified with the {\bf Prym variety} $\Prym(\widehat X_q, X)$, where $\widehat X_q\to X$ is a double cover branched at the zeroes of $q$ called the {\bf spectral curve}.
  This realizes $M_{\Higgs}(X)$ as a smooth torus fibration over this locus (see \textsection{\ref{sec:spectralcurve}}). 
A choice of spin structure $K_X^{1/2}$ gives a section of $\Hscr$ called the {\bf Hitchin section}, and its image is called a {\bf Hitchin component}. This realizes $T(X)\subset M_{\Higgs}(X)$. 
 One direction in the identification \eqref{eqn:nah} goes as follows:
Given a flat connection with monodromy  $[\rho]\in M_{\B}(\Sigma)$ there is a $\rho$-equivariant harmonic map $u:\widetilde X\to \HBbb^3$. 
The Hopf differential of $u$ descends to $X$ as a holomorphic quadratic differential. 
Restricted to a lift of the Fricke space to $\SL(2,\RBbb)$, this gives a diffeomorphism $\QD(X)\simeq T(\Sigma)\subset M_{\B}(\Sigma)$ (see \cite{Hitchin:87} and \cite{Wolf:thesis}).
Under this identification the harmonic maps parametrization of $T(\Sigma)$ and Hitchin section $T(X)$ agree.

\subsubsection{Self-duality equations and limiting configurations}
\label{intro_section:limiting_configurations}
The gauge theoretical perspective on $M_{\Higgs}(X)$ is in terms of 
solutions of  Hitchin's {\bf  self-duality equations}
 for a pair $(A,\Psi)$ consisting of a $\SU(2)$-connection $A$ and
 a hermitian Higgs field $\Psi$. The proper setup for these will be introduced in detail in \textsection{\ref{subsubsect:selfdualityequation}}.  The Kobayashi-Hitchin correspondence yields a bijection between $M_{\Higgs}(X)$ and the space of unitary gauge equivalence classes of solutions of the self-duality equations. We will not distinguish the notation between $M_{\Higgs}(X)$ and the latter space, and therefore we write $[(A,\Psi)]\in M_{\Higgs}(X)$ if $(A,\Psi)$ is an irreducible solution of eqns.\ \eqref{eq:sde}.  There is a  partial compactification of $M_{\Higgs}(X)$ in terms of  {\bf limiting configurations} which we shall present in \textsection{\ref{subsec:mswwComp}}. Briefly, a limiting configuration $[(A_\infty,\Psi_\infty)]$ associated to a differential $q\in \SQD^\ast(X)$
is a solution of the decoupled self-duality equations \eqref{eq:dsde}
on the punctured surface $X^{\times}=X\setminus q^{-1}(0)$ that has a singularity of a specific type in each zero of $q$. Furthermore, $2q=\tr(\Psi_\infty\otimes\Psi_\infty)^{2,0}$. Such limiting configurations have a natural interpretation in terms of parabolic Higgs bundles. For our purposes it will be important that the set of unitary gauge equivalence classes of  limiting configurations associated with  any $q$ as above is a real torus of dimension $6g-6$, and that the Hitchin map $\Hscr$ extends continuously to a map $\Hscr_\infty$ from the space of limiting configurations to $\SQD^\ast(X)$. The fiber $\Hscr^{-1}_\infty$ may be
identified with a torus of  {\bf Prym differentials} on $\widehat X_q$ (see Proposition \ref{prop:prym-limiting}). The Liouville form restricts to a natural Prym differential $\lambda_{\SW}$ called the {\bf Seiberg-Witten differential}, and this will play an important role in the paper.

\subsubsection{Pleated surfaces} \label{sec:intro-pleated}
By a {\bf pleated surface} we mean\footnote{What is defined
here might be called an \emph{equivariant} or  \emph{abstract} pleated 
surface 
to distinguish it from the more standard, nonequivariant situation
(\emph{cf}.\ \cite{CEG}).
 Following \cite{Bonahon:shearing}, we will simply use the term
 \emph{pleated surface} in the equivariant case as well.}
  a $4$-tuple $\mathsf{P}= (S, f,\Lambda, \rho)$
where $S$ is a hyperbolic structure on $\Sigma$; $\Lambda$ is a maximal geodesic lamination on $S$ called the {\bf pleating locus}; $f: \widetilde S\to \HBbb^3$ is a continuous map from the lift $\widetilde S$ to $\HBbb^3$ that is totally geodesic on the components of $\widetilde S\setminus\widetilde \Lambda$,  maps leaves of $\widetilde\Lambda$ to geodesics, and  is equivariant with respect to a representation $\rho:\pi_1\to \PSL(2,\CBbb)$. 
We sometimes abbreviate the $4$-tuple $\mathsf{P}$ as $f: \widetilde S\to \HBbb^3$ when context provides the other data.

Let $\Hcal(\Lambda, \RBbb)$ and $\Hcal(\Lambda, S^1)$ denote the spaces of {\bf shearing} and {\bf bending} cocycles, respectively (see \textsection{\ref{subsec: transverse cocycles}}). We further set $\Hcal^o(\Lambda, S^1)\subset \Hcal(\Lambda, S^1)$ to be the connected component of the identity. 
Then $\Hcal(\Lambda, \RBbb)$ (\emph{resp}.\ $\Hcal^o(\Lambda, S^1)$) is a $(6g-6)$-dimensional vector space (\emph{resp}.\ torus).
Bonahon proves that there is an injective map
\begin{equation}\label{eqn:bonahon}
B_{\Lambda} : \Ccal(\Lambda)\times\sqrt{-1}\,\Hcal^o(\Lambda, S^1)\lra R^o(\Sigma)
\end{equation}
that is a biholomorphism onto  its image \cite[Thm.\ D]{Bonahon:shearing}. Here, $\Ccal(\Lambda)\subset \Hcal(\Lambda, \RBbb)$ is an open convex polyhedral cone  that is naturally identified with $T(\Sigma)$. 
In fact, $[\rho]=B_{\Lambda}(\sigmabold,\betabold)$ 
is constructed via 
a pleated surface $f: \widetilde S\to \HBbb^3$ that is $\rho$-equivariant and has pleating locus $\Lambda$. The complex cocycle $\sigmabold+i\betabold\in\Hcal(\Lambda,\CBbb/2\pi i\ZBbb)$ is called the {\bf shear-bend} cocycle of the pleated surface.

In this paper we will be interested in the special case where $\Lambda$ is determined by the geodesic lamination $\Lambda^h_q$ associated to the horizontal measured foliation $\Fcal_q^h$ of a quadratic differential $q$ that is holomorphic for a marked Riemann surface structure $X$ on $\Sigma$.
We shall always demand that $q$ have simple zeroes.
 If in addition $q$ has no (horizontal) {\bf saddle connections}, then $\Lambda=\Lambda_q^h$.  When $q$ does have saddle connections the situation is  more complicated to formulate, but the fundamental picture described below is unchanged (see \textsection{\ref{sec:maximalization}}). 
In any case, there is a canonical transverse cocycle 
$\sigmabold^{can}_{q}\in\Hcal(\Lambda, \RBbb)$
for $\Lambda$  related to  the transverse measure to the vertical foliation defined by $q$ (see Example \ref{ex:SW-cocycle}).

\subsection{Results} \label{sec:intro-results}

\subsubsection{Statement of the main theorem}
The result below gives an asymptotic comparison between Bonahon's parametrization of the character variety in terms of shear-bend cocycles \eqref{eqn:bonahon}, and the limiting configurations of solutions to the self-duality equations. Consider the following set-up.

Let $[\rho_n]$ be an unbounded sequence in $R^o(\Sigma)$, by which we mean it leaves every compact subset. Assume that the Hopf differentials of the $\rho_n$-equivariant harmonic maps $u_n:\widetilde X\to \HBbb^3$ are of the form $4t_n^2 q_n$, with $t_n\to +\infty$ and $q_n\to q\in \SQD^\ast(X)$. 
We will assume that for some fixed hyperbolic structure on $\Sigma$ we have chosen  maximal laminations $\Lambda_n$, $\Lambda$ containing $\Lambda_{q_n}^h$ and $\Lambda_q^h$, respectively, and such that $\Lambda_n\to \Lambda$ in the Hausdorff sense. In this case, there is a notion of convergence of cocycles in $\Hcal(\Lambda_n,\RBbb)$ and $\Hcal(\Lambda_n, S^1)$ (see Definition \ref{def:convergence-cocycles}). 

Lift $[\rho_n]\in R^o(\Sigma)$ to $[\widetilde \rho_n]\in M_{\B}(\Sigma)$,
and let
  \begin{equation} \label{eqn:nonabelian-hodge}
  [(A_n,\Psi_n)]=\NAH\circ \RH^{-1}([\widetilde\rho_n])
  \end{equation}
be the associated solutions to the self-duality equations.
Let 
$[(A_{\infty},\Psi_{\infty})]$ be any subsequential limiting configuration of the sequence
 $[(A_n,\Psi_n)]$.

To normalize bending cocycles, we adopt the convention that a pleated surface for a Fuchsian representation has a bending cocycle equal to  zero (\emph{cf}.\ \cite[Prop.\ 27]{Bonahon:shearing}). This will allow us to compare the bending cocycles of pleated surfaces with the same underlying hyperbolic metric and pleating lamination. 
In \textsection{\ref{subsec: transverse cocycles and Prym differentials}} we shall describe an explicit realization of elements of $\Hcal(\Lambda, \RBbb)$ and $\Hcal^o(\Lambda, S^1)$ in terms of periods of Prym differentials.
Combining this with the characterization of limiting configurations mentioned at the end of \textsection{\ref{intro_section:limiting_configurations}}, we show that
 there is a $2^{2g}$-sheeted covering homomorphism
\begin{equation} \label{eqn:limit-config-bending-cocycle}
\Hscr^{-1}_\infty(q)\lra \Hcal^o(\Lambda, S^1)\ .
\end{equation}

Given the above we can now make the following statement.

\begin{mainthm} \label{thm:main}
After passing to a subsequence, there is $N\geq 1$ such that the following hold:
\begin{compactenum}[(i)]
\item for all $n\geq N$, $[\rho_n]=B_{\Lambda_n}(\sigmabold_n, \betabold_n)$ for some shearing and bending cocycles $\sigmabold_n$ and $\betabold_n$;
\item the $\rho_n$-equivariant pleated surfaces $f_n: \widetilde S_n\to \HBbb^3$ from (i) are asymptotic to the $\rho_n$-equivariant harmonic maps $u_n: \widetilde X\to \HBbb^3$ in the sense of Definition \ref{def:asymptotic-surfaces};
\item as $n\to \infty$, the shearing cocycles satisfy:
$
(2t_n)^{-1}\sigmabold_n\lra \sigmabold_q^{can}
$;
\item as $n\to \infty$, the bending cocycles satisfy: $\betabold_n\to \betabold$,
where $\betabold\in \Hcal^o(\Lambda, S^1)$ is the image of the limiting configuration $(A_\infty,\Psi_\infty)$ under the map \eqref{eqn:limit-config-bending-cocycle}.
\end{compactenum}
\end{mainthm}

\subsubsection{Invariance with respect to the base Riemann surface}
A key assumption above and in the work of \cite{msww15} is that quadratic differentials have simple zeroes. For this reason,  limiting configurations give only a partial compactification of $M_{\Higgs}(X)$, and we are unable to make a uniform statement about the topological invariance of these limit points.  We therefore content ourselves here with proving the \emph{local} invariance with respect to the Riemann surface structure $X$.

To make this precise, let $q_0\in \SQD^\ast(X_0)$, and let $\Fcal_{q_0}^v$ denote the associated vertical measured foliation.  Let $\widetilde U\subset T(\Sigma)$ be the set of all equivalence classes of Riemann surfaces $X$ such that the Hubbard-Masur differential $q_X$ of the pair $(X, \Fcal_{q_0}^v)$ has simple zeroes. Then for a contractible open subset $ U_0\subset\widetilde U$ containing $X_0$, and $X\in U_0$, the Gauss-Manin connection gives an identification of the Prym varieties of $X$ and $X_0$, and this in turn induces an identification of bending cocycles for the horizontal laminations associated to $q_X$ and $q_0$ through \eqref{eqn:limit-config-bending-cocycle}.   As mentioned above, a complication, described in more detail in  \textsection{\ref{sec:maximalization}}, is that the  laminations $\Lambda_{q_0}^h$ and $\Lambda_{q_X}^h$ may not be maximal.

 \begin{cor} \label{cor:invariance-of-basepoint}
 The partial compactification by limiting configurations is locally independent of the base Riemann surface in the following sense. 
 Suppose $[\rho_n]$ is a divergent sequence in $R^o(\Sigma)$, and lift $[\rho_n]$ to $[\widetilde \rho_n]\in M_{\B}(\Sigma)$.
For $X\in U_0$, define $[(A_n,\Psi_n)_X]$ and $[(A_n,\Psi_n)_{X_0}]$ as in \eqref{eqn:nonabelian-hodge}. We assume $[(A_n,\Psi_n)_{X_0}]$ has a well defined limiting configuration, which we suppose lies in the fiber $(\Hscr_\infty^{X_0})^{-1}(q_0)$. 
Let $[\widehat\eta_{X_0}]$  be the associated Prym differential (as mentioned above; see Proposition \ref{prop:prym-limiting} for the precise statement), and $q_X \in \SQD^\ast(X)$ chosen as above to share the projective class of vertical measured foliations with $q_{X_0}$.

Then $[(A_n,\Psi_n)_X]$ has a well defined limiting configuration 
in the fiber $(\Hscr_\infty^{X})^{-1}(q_X)$. Moreover,
if  $[\widehat\eta_X]$ is the associated Prym differential for the bending cocycle of this limiting configuration, then 
 $[\widehat\eta_X]$ and $[\widehat\eta_{X_0}]$ are identified by parallel translation  by the Gauss-Manin connection. 
 \end{cor}

Here is an interpretation of this result. For $q\in
\SQD^\ast(X)$ there is a natural identification of the fibers
$\Hscr^{-1}(tq)$ for all $t>0$. This gives a partial
compactification of $M_H(X)$, and hence via $\NAH$ and $\RH$,
also of $M_{\B}(\Sigma)$. A priori, this depends on the
choice of base Riemann surface structure $X$. Corollary 
\ref{cor:invariance-of-basepoint} states that (locally) this
partial compactification is independent of $X$.

\subsubsection{Relation to the Morgan-Shalen compactification} \label{sec:MS}
There is a mapping class group invariant compactification of $R(\Sigma)$ due to Morgan and Shalen \cite{MorganShalen:84}. The ideal points of this compactification are generalized length functions on $\pi_1$, which turn out to be the translation length functions for an isometric action of $\pi_1$ on an $\RBbb$-tree (see \textsection{\ref{sec:morgan-shalen}}).
A harmonic maps description of this compactification was partially described in \cite{DaDoWen}, which was an attempt to mirror the result in \cite{Wolf:thesis} for the Thurston compactification of $T(\Sigma)$. 
A consequence of \cite{WolfT} is that the $\RBbb$-trees appearing in the limit of a sequence of Fuchsian representations are obtained as the leaf space of the vertical foliation $\Fcal_q^v$ of the rescaled Hopf differential $q$ on $\widetilde X$. This is called the dual tree $T_{q}$ to $q$.  In the case of $R(\Sigma)$, sequences of representations that are not discrete embeddings may give rise to trees that are foldings of $T_{q}$. 
The harmonic maps point of view gives some information about this: A folding cannot occur if $q$ has simple zeroes and $\Fcal_q^v$ has no saddle connections.
It has been an open question how to describe this process completely in terms of harmonic maps.  Using Theorem \ref{thm:main}, we can obtain a criterion ruling out folding in the case of simple zeroes, as well as a partial refinement of the Morgan-Shalen compactification by classes of limiting configurations. 

\begin{cor} \label{cor:morgan-shalen}
Let $[\rho_n]\in R^o(\Sigma)$ be as in Theorem \ref{thm:main}. Suppose that the periods of the Prym differential associated to the limiting configuration $(A_\infty,\Psi_\infty)$
are bounded away from $\pi$
 on every cycle defined by a saddle connection of the vertical foliation of $q$. Then the $\RBbb$-tree defined by the Morgan-Shalen limit of any subsequence of $[\rho_n]$ is $\pi_1$-equivariantly isometric to the  dual tree $T_{q}$. 
\end{cor}

\subsubsection{Limits of complex projective structures}
The subset of $R^o(\Sigma)$ consisting of monodromies of complex projective structures on $\Sigma$ with underlying Riemann surface $X$ is naturally an affine space modeled on  $\QD(X)$. 
The corresponding local systems are called $\SL(2,\CBbb)$-{\bf opers}.
A base point is given by the Fuchsian projective structure $Q_X$.  More precisely, uniformization gives rise to an isomorphism $u: \widetilde X\to \HBbb^2$, equivariant with respect to $\pi_1$ and a Fuchsian representation of $\pi_1\to \Iso^+(\HBbb^2)$, and the Schwarzian derivative of $u$ gives a projective connection $Q_X$ on $X$.
Any other projective connection is of the form $Q(q)=Q_X-2q$, for $q\in \QD(X)$, and we obtain an  embedding $\Pscr: \QD(X)\hookrightarrow R^o(\Sigma)$ from the monodromy of the oper defined by  the following differential equation on $X$:
\begin{equation} \label{eqn:diff-eq}
y''+\frac{1}{2}Q(q)y=0\ ,
\end{equation}
where $y$ is a local section of $K_X^{-1/2}$. Thurston associates to every projective structure a pleated surface $f :\widetilde S(q)\to \HBbb^3$ that is equivariant with respect $\Pscr(q)$ and has pleating locus  along some measured
 lamination $\Lambda(q)$ (see \cite{KamTan})\footnote{Strictly speaking, $\Lambda(q)$ need not be maximal, but this possibility will not play any role in the result.}. 
 Choosing a lift of the Fuchsian representation to $M_{\B}(\Sigma)$ gives a lift $\widetilde\Pscr(q)\in M_{\B}(\Sigma)$.
Let 
\begin{equation} \label{eqn:op}
\Op(q)=\NAH\circ\RH^{-1}(\widetilde\Pscr(q))\in M_{\Higgs}(X)\ .
\end{equation}

The next application compares the limiting behavior of $\Op(q)$ for $q$ large and the geometry of Thurston's pleated surface.
By work of Dumas \cite{Dum2}, the measured  laminations $\Lambda(q)$ converge projectively to $\Lambda^h_q$. 
  This allows us to compare bending cocycles. 
 Combined with the Main Theorem, we prove the following.

\begin{cor} \label{cor:thurston's-pleated-surface}
Let $q\in \SQD^\ast(X)$. 
\begin{compactenum}[(i)]
\item $\displaystyle \lim_{t\to+\infty}t^{-2}\Hscr(\Op(t^2q))=q/4$;
\item under the correspondence between Prym differentials and points in the Prym variety, the spectral data $[\widehat\eta_t]$ of $\Op(t^2q)$
satisfies
$$
\lim_{t\to+\infty}[\widehat \eta_t-it\imag\lambda_{\SW}]= 0\ \text{in} \ 
\Prym(\widehat X_q,X)/J_2(X)\ ,
$$
where $\lambda_{\SW}$ is the Seiberg-Witten differential on  $\widehat X_q$;
\item $(${\sc Dumas}$)$
if $\Gamma_t=\sigmabold_t+i\betabold_t$ is the shear-bend cocycle of Thurston's pleated surface for the projective connection $Q(t^2q)$ in \eqref{eqn:diff-eq},
then
$$
\lim_{t\to+\infty}t^{-1} \Gamma_t = \Gamma_{\SW}\ ,
$$
where $\Gamma_{\SW}$ is the complex cocycle determined by the
 periods of $\lambda_{\SW}$ $($see Definition
\ref{def:gamma-SW}$)$.
\end{compactenum}
\end{cor}
In Part (ii) of the Corollary, $J_2(X)$ denotes the $2$-torsion points of the Jacobian variety of $X$. Its appearance in the statement of part (ii) is due to the ambiguity in the choice of the square root of $K_X$.
Part (iii) also follows from the results of \cite{Dum2} and \cite{Dumas06} (see also \cite{Dumas06Err}).

We refine Corollary~\ref{cor:thurston's-pleated-surface}(i) with an error estimate $t^{-2}\Hscr(\Op(t^2q))- q/4 = O(t^{-1})$ in Proposition~\ref{prop: refined error}.

\subsection{Further comments}

\subsubsection{Discussion of the main results}

The Hitchin parametrization of (an open set in) $M_{\Higgs}(X)$ gives it the structure of a torus fibration over $\QD^\ast(X)$ which, via the Hopf differential of the harmonic diffeomorphism from $X$ to a hyperbolic surface $S$, can be identified with Teichm\"uller space $T(\Sigma)$. 
 Similarly, in the presence of a maximal lamination, Bonahon also parametrizes (an open set in) $R^o(\Sigma)$ as a torus fibration over $T(\Sigma)$. The nonabelian Hodge correspondence is a transcendental homeomorphism between these two pictures. 
The asymptotic decoupling of the Hitchin equations reflects the conclusion of this paper that these two torus fibrations are essentially asymptotically equivalent. 

Previous work on the asymptotics of equivariant harmonic maps focused on the behavior of divergent length functions corresponding to shearing cocycles, and this is well understood. The novelty of the present work is to extract information on the \emph{complex} length, which involves bending. Perhaps not surprisingly,  through the nonabelian Hodge correspondence, bending turns out in the gauge theory picture  to involve the unitary part of the flat connection.

An important subtlety happens when the quadratic differentials have saddle connections. These may occur in either the horizontal or vertical foliations, or both, and they play different roles. 
Saddle connections in the vertical foliation give rise to the possibility of folding in the image of harmonic maps. This will be discussed more explicitly in \textsection{\ref{sec:MS}} below.
More relevant are the saddle connections that appear in the horizontal foliation. 
In this case, geodesic straightening of the 
leaves does not produce a maximal lamination, and 
so a choice of \emph{maximalization}  is required 
(see \textsection{\ref{sec:laminations}). Unlike the 
complicated wall-crossing phenomena that emerge from this 
situation in other contexts, here in the asymptotic limit the 
choice of maximalization is a technical tool that 
amounts to a change of coordinates in the identification \eqref{eqn:limit-config-bending-cocycle}. 
 
In terms of the consequences of the Main Theorem, let us
 elaborate
 on the discussion in the introductory paragraphs.  The 
work of \cite{MSWW:Ends} analytically describes the 
frontier of $M_{\Higgs}(X_0)$ as a torus fibration over 
$\SQD^{\ast}(X_0)$  
for a chosen Riemann surface $X_0$; 
the elements of a fiber are equivalence classes of 
Prym differentials
defined in terms of $q \in \SQD^{\ast}(X_0)$.
It is not apparent how 
this parametrization of limiting configurations
in terms of the Riemann surface $X_0$
relates to the one defined in terms of a nearby Riemann surface $X$.
 To clarify the question, imagine a pair of sequences of
 representations $\{\rho^{+}_n, \rho^{-}_n\} \subset R^o(\Sigma)$ 
where the  associated solutions to the self-duality equations have  
 distinct limiting configurations in a particular torus fiber over
 $q_0 \in \SQD^{\ast}(X_0)$. If we then change the
 original choice of Riemann surface from $X_0$ to $X$ and consider the partial bordification of $M_{\Higgs}(X)$, will the
solutions of \eqref{eq:dsde} for $\{\rho^{+}_n, \rho^{-}_n\}$ 
on the new surface $X$
still 
have limiting configurations
  in a  torus fiber over a single quadratic differential
 $q \in \SQD^{\ast}(X)$, or will they accumulate over 
different fibers? Corollary~\ref{cor:invariance-of-basepoint} 
asserts that an entire limiting torus, defined in terms of 
either Riemann surface, 
projects to a single point in 
the Morgan-Shalen compactification, which is defined only in terms of 
the topologically defined projective vertical measured foliation of $q$. 
 Thus the tori of limiting configurations defined by 
 $X$ or $X_0$ are either disjoint or coincide.  Moreover, and this
property  is more subtle,
the elements of each torus fiber may be identified by periods
of a differential corresponding to the limiting bending
cocycle of a sequence of pleated surfaces.
These  periods are therefore  
also topologically defined. 
 Turning this discussion around, we thus see that 
Corollaries~\ref{cor:invariance-of-basepoint}
 and \ref{cor:morgan-shalen} provide a partial refinement of the
 Morgan-Shalen compactification, an apparently topological result, 
via a construction that is geometric-analytical.

Corollary~\ref{cor:thurston's-pleated-surface} provides an appealing
 picture of the space of complex projective structures. 
 The classical Schwarzian view of the space of complex projective
 structures is as an affine bundle over \tec space: the fibers over a
 point $X\in T(\Sigma)$ is the space $\QD(X)$ of quadratic differentials
 on $X$, and we can focus our attention on a ray $\{tq, t>0\} \subset 
\QD(X)$ of Schwarzian derivatives on $X$. A basic question is to describe the image of this ray in the fibration $M_{\Higgs}(X)$.

 In Corollary~\ref{cor:thurston's-pleated-surface} we see the end of 
such a ray, when seen as a family in $M_{\Higgs}(X)$, shadows 
a linear flow on one of the torus boundary fibers.  
Different rays over a common point $X$ in \tec space shadow flows
 over distinct tori, depending on the vertical foliation
 of the common (projective) Schwarzian derivative.
 In short, the rays over a single Riemann surface have ends
 accumulating in each fiber of the partial compactification of
 $M_{\Higgs}(X)$.  On the other hand, rays over distinct Riemann
 surfaces, whose Schwarzians have a common projective 
vertical measured foliation, shadow flows over a
 common torus fiber in the partial compactification of $M_{\Higgs}(X)$. 
 Here the direction vectors of the flows reflect the 
underlying Riemann surface of the family of complex
 projective structures, through the horizontal 
foliation of the Schwarzians.

Finally, in the context of this last corollary, we provide a 
small bit of intuition for these claims,
 effectively due to Dumas in this setting. A family of
 complex projective structures over $X$ with proportional
 Schwarzians $tq, t\gg0$ may be seen as the \enquote{graftings} of
 a family of surfaces $X = \gr_{\lambda_t}(X_t)$.
\footnote{We provide a quick informal introduction to grafting.  A complex projective structure will develop as a domain over complex projective space $\CP^1$. We can regard $\CP^1$ as $\partial\H^3$. Then given a pleated surface $(S, f, \Lambda, \rho)$ in $\H^3$, we can imagine exponentiating in the normal direction from the image $f(S)$ of that pleated surface. The limiting image of a totally geodesic plaque under this flow will inject onto a domain in $\CP^1$ bounded by circular arcs. The image of the bending lamination $\Lambda$ is more complicated, reflecting the complicated nature of a geodesic lamination, but can be imagined as (limits of) thin crescents that connect the images of plaques: for example, if $\Lambda$ were only a single simple closed curve $\gamma$ with bending measure $\theta$, then each lift of $\gamma$ would force the inclusion of a \enquote{lune} of width $\theta$. Thurston observed that each complex projective structure admitted a unique description as a hyperbolic structure $S$ as above, together with the insertion of flat lunes corresponding to the bending lamination. We might call $S$ the \enquote{pruning} of the complex projective structure.}
 Here the lamination $\Lambda_t$ is the bending lamination for a pleated
 surface whose underlying hyperbolic surface is $X_t$, and 
obviously the bending measure grows with $t$. Of course, 
the result of the grafting is a fixed Riemann surface $X$, 
and so the pruned surface $X_t$ reflects the increased bending by 
growing thin and long in the direction of the bending lamination.  But 
the representation $[\rho_t]\in R^o(\Sigma)$  of the pleated surface $X_t$
 is what we focus on in this paper. We see then that 
translation lengths for this representation must be growing long,
 roughly parallel the stretched 
 lamination $\Lambda_t$.  Passing from these synthetic constructions
 to geometric analysis by considering the shape of the 
$\rho_t$-equivariant energy minimizing map  $u_t: \tilde{X}\to \H^3$,
 the Hopf differential $q_t \in \QD^{\ast}(X)$ of $u_t$
 will have  horizontal foliation in the direction of the maximal 
stretch of the map, which in this case will be forced to be
 along the very stretched lamination $\Lambda_t$.
 Indeed, $u_t$ will crowd much of its image near the maximally
 stretched lamination $\Lambda_t$.  The regions complementary to those
 mapping near $\Lambda_t$ will be stretched to efficiently
connect 
components of $\Lambda_t$,  and they will thus lie near portions of 
totally geodesic hyperbolic planes in $\H^3$. Taken together, 
these heuristics imply that
 the image of $u_t$
 will itself approximate the pleated surface
$f:\widetilde S(q_t)\to\HBbb^3$.  
 
 Now, on the one hand, the analysis of \cite{MSWW:Ends} tells us,
 as a starting point, that $\rho_t$ will
have a limiting configuration in the fiber 
 defined by the limit of  normalized  Hopf differentials. 
The Main Theorem 
works by recognizing the gauge theoretic endomorphism
 that represents such a limit point as an infinitesimal rotation in 
$\H^3$ 
about the geodesic tangent to the image of a horizontal leaf.
 
Ignoring for now the issue of how the harmonic map is bending near
 $\Lambda_t$, we note that near the preimages of $\Lambda_t$, 
the harmonic map is well-approximated by the very simple harmonic
 map $\C\to \H^3$ which takes horizontal lines in the plane to 
a geodesic with a parametrization proportional to arc length.  As
 that simple model map has Hopf differential $dz^2$, we see that we 
can expect the vertical foliation of the $\rho_t$ harmonic map to 
predict the length spectrum of the representation of $\rho_t$,
 at least up to its leading terms. As the length spectrum of a
 representation is independent of the choice $X$ of the background
 Riemann surface, we find evidence for
 Corollaries~\ref{cor:invariance-of-basepoint} and 
\ref{cor:morgan-shalen}. Finally, Dumas \cite{Dumas06} and 
\cite{Dum2} makes the deep observation that $\Lambda_t$ is 
well-approximated by the horizontal lamination of the Schwarzian, 
and thus the underlying geometric lamination for the bending 
lamination $\Lambda_t$ becomes increasingly fixed as $t$ increases,
 even as the amount of bending grows linearly with the measure of the 
horizontal foliation, \emph{i.e.} linearly with $t$.  That
 linear change in the complex translation lengths of the dominant 
elements of the holonomy for $\rho_t$, coupled with the just mentioned
 relationship of the gauge theory to hyperbolic geometry, 
suggests the linear flow in Corollary~\ref{cor:thurston's-pleated-surface}.

\subsubsection{Relation with other work}
The literature on Hitchin systems, solutions to differential equations on Riemann surfaces, and their asymptotic properties is vast, and the Main Theorem in this paper may be viewed in that context. 

Asymptotic decoupling of the self-duality equations has been studied in \cite{Taubes,MSWW:Ends,MSWW:HitchinMetric,Mochizuki:16,Fredrickson:18b}. 
 This idea is also central to the work of Gaiotto-Moore-Neitzke \cite{GMN:13} and the conjectural structure of the hyperk\"ahler metric on $M_{\Higgs}(X)$ (see \cite{DumasNeitzke:19,Fredrickson:18a}). The idea of
``nonabelianization'' also arises from this work and is related to Fock-Goncharov cluster coordinates and the Bonahon parametrization. This
 has been investigated in Hollands-Neitzke \cite{HollandsNeitzke:16,HollandsNeitzke:19} and Fenyes \cite{Fenyes:15}.  

Corollary \ref{cor:thurston's-pleated-surface} is a kind of zero-th order analog of the much more extensive results from the  \emph{exact WKB analysis} of Schr\"odinger equations \cite{KawaiTakei:05} (where $t=1/\hbar$). In particular, the
period map
\begin{equation} \label{eqn:central_charge}
Z(\gamma) = \int_\gamma\lambda_{\SW}\ ,
\end{equation}
for $\gamma$ representing an odd homology class on $\widehat X_q$, 
 plays a central role in  \cite{GMN:13}.  For some recent work, see
\cite{IwakiNakanishi:14,Allegretti:19a,Allegretti:19b}.

\subsubsection{Outline of the paper}
This paper is organized as follows.
 In \textsection{\ref{sec:background}} we have provided a rather large amount of background material in order to make the rest of the paper  accessible to a wide readership.
The main topics are the moduli space of solutions to the self-duality equations, limiting configurations, and their relation to  spectral data and Prym differentials.  We also provide details on equivariant harmonic maps and their high energy properties. The section concludes with background on laminations, measured foliations, train tracks, and $\RBbb$-trees, which will be useful in the sequel.

 These preliminaries are followed in \textsection{\ref{sec:bending}} by a discussion of ``bending''. We first introduce a naive geometric notion of how to measure the bending of an immersive map to hyperbolic space in terms of dihedral angles of intersecting tangent planes.
In the context of the equivariant maps that appear in the nonabelian Hodge correspondence, we compare this notion to an alternative definition of bending coming from parallel translation in bundles with connections. When Higgs pairs approach a limiting configuration,  the gauge theoretic bending  is shown to be determined by the periods of Prym differentials. 

In \textsection{\ref{sec:pleated-surfaces}} we review the notion of a transverse cocycle to a geodesic lamination, as well as Bonahon's parametrization of the character variety $R(\Sigma)$. In Lemma \ref{lem:finite-approximation-bending} we show that under certain assumptions on the pleating locus the bending cocycle can be related to the geometric notion of bending introduced in  \textsection{\ref{sec:bending}}. We use this property to derive the bending cocycle of a pleated surface from the gauge theoretic notion of bending, under the assumption that the pleated surface and the image of the equivariant harmonic map are appropriately close.

The existence of a pleated surface with the properties just mentioned is proven in \textsection{\ref{sec:realization}}.  The required asymptotic results for high energy harmonic maps are largely due to Minsky. The key idea is to compare an arrangement of geodesics in $\HBbb^3$ obtained from the image of horizontal leaves of the foliation by an equivariant harmonic map to the geodesic lamination on the hyperbolic surface corresponding to a harmonic diffeomorphism with the same Hopf differential. We show that  by perturbing this hyperbolic structure slightly, the geodesic configuration in $\HBbb^3$ extends to a pleated surface. 

Finally, in \textsection{\ref{sec:proofs}} we give the proofs of the Main Theorem and its corollaries.

\subsubsection{Acknowledgments}
The authors wish to express their gratitude to a great many people  for their interest in this work and for their helpful comments. Among them are Dylan Allegretti, Francis Bonahon, Marc Burger, Steve Bradlow, Aaron Fenyes, Vladimir Fock,  Fran\c{c}ois Labourie, Rafe Mazzeo, Yair Minsky, and Andy Neitzke.

{\bf A.O.} was supported by the European Research Council under ERC-Consolidator Grant 614733 ``Deformation Spaces of Geometric Structures'', and by the Priority Program 2026 ``Geometry at Infinity'' of the Deutsche Forschungsgemeinschaft (DFG, German Research Foundation) under DFG grant 340014991.  He further acknowledges funding by the Deutsche Forschungsgemeinschaft (DFG, German Research Foundation) -- 281869850 (RTG 2229). He was also supported by the Deutsche Forschungsgemeinschaft (DFG, German Research Foundation) under Germany's Excellence Strategy EXC 2181/1 - 390900948 (the Heidelberg STRUCTURES Excellence Cluster).
{\bf J.S.} is supported by   a Heisenberg grant  of the German Research Foundation,   the DFG Priority Program 2026 ``Geometry at Infinity'',   the  DFG RTG 2229 ``Asymptotic invariants and limits of groups and spaces'' and by   DFG under  Germany's Excellence Strategy EXC-2181/1 -- 390900948 (the Heidelberg  STRUCTURES Cluster of Excellence).
{\bf R.W.} is supported in part by NSF grants DMS-1564373 and DMS-1906403. 
{\bf M.W.} acknowledges support from NSF grants DMS-1564374 and DMS-2005551 and the Simons Foundation.
{\bf M.W.} and {\bf R.W.} also received funding from NSF grant  DMS-1440140 administered by the
Mathematical Sciences Research Institute while the authors were in residence at MSRI
during the ``Holomorphic Differentials in Mathematics and Physics'' program Aug 12-Dec 13, 2019.
Finally, the authors  acknowledge support from NSF grants DMS-1107452, -1107263, -1107367 ``RNMS: GEometric structures And Representation varieties'' (the GEAR Network).

\section{Background material} \label{sec:background}

\subsection{Higgs bundles}\label{subsect:Higgsbundles}

\subsubsection{The self-duality equations}\label{subsubsect:selfdualityequation}

We introduce the setup for Hitchin's self-duality equations for a topologically trivial rank $2$ complex vector bundle $E$   in a form that will be useful later on. 

Fix a choice of  spin structure $K_{X}^{1/2}$ as in \textsection{\ref{sec:intro-moduli}} and consider
\begin{equation}\label{eq:decompvbE}
	E= K_{X}^{-1/2}\oplus K_{X}^{1/2}\ .
\end{equation}
A choice of conformal metric $ds^2=m(z)|dz|^2$ on $X$ induces a hermitian metric $h=(m^{1/2}, m^{-1/2})$ on $E$ which will be fixed throughout. Notice that the determinant line bundle
 $\det E$ with its induced metric from $h$ is canonically trivial.
  Let $\g_E$ be the vector bundle of traceless skew-hermitian endomorphisms of $E$, and $\gfrak_E^\CBbb$ its complexification consisting of all traceless endomorphisms. 
   We will also often use $\sqrt{-1}\gfrak_E$, the bundle of traceless hermitian endomorphisms. The hermitian metric $h$ on $E$ induces a hermitian metric on the associated endomorphism bundle $\gfrak_E^\CBbb$ which is given by
   \begin{equation*}
   	\langle A,B\rangle =\frac{1}{2} \tr(AB^{\ast_h})
   \end{equation*}
   for $A,B\in\Gamma(\gfrak_E^\CBbb)$. On the subbundle of traceless hermitian endomorphisms this metric reads  $\langle A,B\rangle = (1/2)\tr(AB)$.

 Denote by $\mathcal A(E,h)$ the space of smooth connections on $E$ that are unitary with respect to $h$ and which induce the trivial connection on $\det E$. This is an affine space modeled on $\Omega^1(X,\gfrak_E)$. A $\dbar$-operator $\dbar_E$ on $E$ defines a holomorphic bundle $\Ecal$ which we will often denote by $\Ecal=(E,\dbar_E)$. There is a connection $A=(\dbar_E,h)\in\Acal(E,h)$ called the {\bf Chern connection} that is uniquely determined by the requirement
 $$\dbar_A:=(d_A)^{0,1}=\dbar_E\ .$$
  In this way, $\Acal(E,h)$ is identified with the space of $\dbar$-operators on $E$. 
Similarly,
there is a real linear isomorphism
\begin{equation} \label{eqn:one_form_iso}
\Omega^1(X, \sqrt{-1}\gfrak_E)\isorightarrow \Omega^{1,0}(X, \gfrak_E^\CBbb) : \Psi\mapsto \Phi=\Psi^{1,0}
\end{equation}
with inverse: 
\begin{equation} \label{eqn:psi_phi}
\Psi=\Phi+\Phi^{\ast_h}\ .
\end{equation}
 We shall often use this convention, $\Phi\leftrightarrow \Psi$, for the isomorphism \eqref{eqn:one_form_iso}.

A {\bf Higgs bundle} is a pair $(\dbar_E,\Phi)$, where $\dbar_E\Phi=0$. The {\bf Higgs field} $\Phi$
may either be regarded as a holomorphic $1$-form valued in the sheaf $\End_0\Ecal$ of traceless endomorphisms of $\Ecal$, or as a holomorphic section of $\Ecal nd_0\Ecal\otimes K_X$. The context throughout will make clear which interpretation applies.

  For a pair $(A,\Psi)\in \Acal(E,h)\times \Omega^1(X, \sqrt{-1}\gfrak_E)$,  the system of PDEs
\begin{equation}\label{eq:sde}
\begin{cases}
F_A+[\Psi\wedge\Psi]	=0 \\
d_A\Psi=0 \\
d_A(\ast\Psi)=0
\end{cases}
\end{equation}
is called the {\bf self-duality equations}.
 A solution $(A,\Psi)$  
 gives a Higgs bundle $(\dbar_A,\Phi)$. The holomorphicity of $\Phi$ follows from the last two equations in \eqref{eq:sde}. Conversely, Hitchin shows that for a polystable Higgs bundle $(\dbar_E,\Phi)$ there is a complex gauge transformation $g$ such that
 the Chern connection and $\Psi$ associated to $g\cdot(\dbar_E,\Phi)$ give a solution to \eqref{eq:sde}. Polystability will not play a role in this paper, so we omit its definition. 
We frequently refer to a solution $(A,\Psi)$ of  \eqref{eq:sde} as a {\bf Higgs pair}.

 Let $M_{\Higgs}(X)$ denote the moduli space of unitary gauge equivalence classes of solutions of  \eqref{eq:sde}. Then $M_{\Higgs}(X)$ is a quasiprojective variety of   dimension  $6g-6$. 
  By a slight abuse of notation, 
when $(\dbar_E,\Phi)$ is polystable and $(A,\Psi)$ the associated solution to \eqref{eq:sde} as in the previous paragraph, 
  we shall write $[(\dbar_E,\Phi)]$ to mean the  gauge equivalence class $[(A,\Psi)]\in M_{\Higgs}(X)$.

A very important fact used in this paper is the following: If $(A,\Psi)$ is a solution to \eqref{eq:sde} then the $\SL(2,\CBbb)$-connection
\begin{equation} \label{eqn:nabla}
\nabla:= d_A+\Psi
\end{equation}
is flat. This follows from the first two equations of \eqref{eq:sde}.

\subsubsection{Quadratic differentials}

Recall the notation $\QD(X)$, $\QD^\ast(X)$, $\SQD(X)$, and $\SQD^\ast(X)$ from \textsection{\ref{sec:notation}}.
We define a norm on $\QD(X)$ by
$$
\Vert q\Vert_1 := \int_X|q(z)|\, \frac{i}{2}dz\wedge d\bar z\ ,
$$
where $q=q(z)dz^2$ in local conformal coordinates,
and we let $Z(q)$ denote the set of zeroes of $q$.

 The map 
\begin{equation} \label{eqn:Hitchin-map}
\Hscr
\colon M_{\Higgs}(X)\to\QD(X),\qquad [(A,\Psi)]\mapsto\frac{1}{2}\tr(\Psi\otimes\Psi)^{2,0}=-\det\Phi\ ,
\end{equation}
 is holomorphic, proper and surjective. Its restriction to $M_{\Higgs}^{\ast}(X):=\Hscr^{-1}(\QD^{\ast}(X))$ is a  fibration with fibers consisting of half-dimensional complex  tori (\emph{cf}.\ \textsection{\ref{sec:spectralcurve}} below).

   As shown in \cite{Hitchin:87}, the Hitchin fibration $\Hscr$ has a global section described as follows. The bundle $E$ has a distinguished holomorphic structure $\dbar_0$ coming from the splitting \eqref{eq:decompvbE} and the holomorphic structure on $K_X^{\pm 1/2}$. Let $A_0$ be the Chern connection associated to  $(\dbar_0,h)$.
Then the section is given by:
\begin{equation} \label{eqn:hitchin-section}
\Sscr_H\colon \QD(X)	\to  M_{\Higgs}(X)\ ,\qquad q\mapsto \left[\left(\dbar_{A_0}, \Phi=
\begin{pmatrix}
 0&1\\q&0	
 \end{pmatrix}\right)\right]\ 
\end{equation}
(recall the convention concerning the notation $[(\dbar_E,\Phi)]$ from the previous section).

\subsubsection{The  partial  compactification by limiting configurations} \label{subsec:mswwComp}

Properness of the Hitchin fibration implies that every sequence $(A_n,\Psi_n)$, $n\in\N$, of solutions of eq.\ \eqref{eq:sde} such that the sequence $q_n=-\det\Phi_n$ (recall, $\Phi=\Psi^{1,0}$) is bounded has a subsequence that converges smoothly modulo the action of unitary gauge transformations. Conversely, a sequence $(A_n,\Psi_n)$  diverges  if the sequence $q_n$ of   holomorphic quadratic differentials diverges: \emph{i.e.}\  $\|q_n\|_1\to\infty$ as $n\to\infty$. Notice that the latter is equivalent to  $\|\Psi_n\|_2\to\infty$ as $n\to\infty$ (here the subscript refers to the $L^2$-norm).

\medskip
By the results in \cite{MSWW:Ends} (see also \cite{Taubes}), the open and dense region $M_{\Higgs}^{\ast}(X)$ of $M_{\Higgs}(X)$ admits a bordification
 by the set $\partial M_{\Higgs}^{\ast}(X)$ of so-called limiting configurations, as we explain next. To this end, we introduce the {\bf decoupled self-duality equations}
\begin{equation}\label{eq:dsde}
\begin{cases}
F_A=0\ ,\
[\Psi\wedge\Psi]=0 \\
d_A\Psi=0\\
d_A(\ast\Psi)=0	
\end{cases}
\end{equation}
for a Higgs field $\Psi$ and a unitary connection $A$.

\begin{definition}\label{def:limitingconfiguration}\upshape
Let $q\in\QD^{\ast}(X)$. A pair $(A,\Psi)$ is called a \textbf{limiting configuration for $q$} if $\det\Phi=-q$ and $(A,\Psi)$ is a smooth solution of \eqref{eq:dsde} on the punctured surface $X_q^{\times}:=X\setminus Z(q)$.
\end{definition}

This definition only applies to  solutions for differentials $q \in \QD^{\ast}(X)$.
 We refer to \cite{Mochizuki:16} for the definition and description of limiting configurations for points $q \in \QD(X)\setminus \QD^{\ast}(X)$.

\begin{exlabel} \label{ex:fuchsian}\upshape
 Recall the connection $A_0$ from \eqref{eqn:hitchin-section}. 
For  $q\in\QD^{\ast}(X)$, we define
\begin{align}
\begin{split}\label{eq:stdlimconn}
	A_{\infty}^0(q)&=A_0 + \frac{1}{2} \left(\Im\, \bar \partial \log \Vert q\Vert\right)  \, 
\begin{pmatrix}
 -i & 0 \\ 0 & i
 \end{pmatrix}\ , \\
  \Phi_{\infty}(q)&= \begin{pmatrix}
 0 & \Vert q\Vert^{1/2}  \\ \Vert q\Vert^{-1/2}q & 0	
 \end{pmatrix}\ ,\\
  \Psi_\infty(q)&=\Phi_\infty(q)+\Phi^{\ast_h}_\infty(q)\ ,
\end{split}
  \end{align}
  where $\Vert q\Vert$ means the (pointwise) norm with respect to the conformal metric $ds^2$. 
The pair $(A_{\infty}^0(q),\Psi_{\infty}(q))$ is a limiting configuration for $q$. It will later become important as the limiting configuration corresponding to a pleated surface with zero bending cocycle. We therefore call it the {\bf Fuchsian limiting configuration} associated to $q$.
\end{exlabel}
We shall often write $(A_\infty^0, \Psi_\infty)$, where the quadratic differential is understood. More generally,
  any other limiting configuration $(A_{\infty}, \Psi_{\infty})$ representing a point   in the fiber   $\Hscr^{-1}(q)$ 
is of the form
\begin{equation}\label{eq:genlimitingconf}
A_\infty=A^0_\infty+\eta\ ,\quad
  [\eta \wedge \Psi_{\infty}]=0\ , \quad \mbox{and} \quad d_{A^0_{\infty}}\eta = 0\ ,
\end{equation}
where $\eta\in \Omega^1(X_q^{\times},\mathfrak g_E)$.
The   group $\mathcal G=\mathcal G(E,h)$ of unitary gauge transformations of $E$ acts on the space of solutions $(A_{\infty}, \Psi_{\infty})$ to eq.\ \eqref{eq:dsde}, and we define the moduli space
\begin{equation*}
	\partial M^\ast_{\Higgs}(X) = \{ \mbox{all solutions to \eqref{eq:dsde} for } q\in \QD^\ast(X)  \} / \mathcal G\times \RBbb^+.
\end{equation*} 

We make two remarks concerning this definition. First,
there is some ambiguity  in that we can either divide out by gauge transformations that are smooth across the zeroes of $q$ or
by those that are singular at these points.  The latter group is more complicated to define because it depends on $q$, and most elements 
in its gauge orbit are singular.   We therefore 
take a view consonant with the original definition of limiting configurations in \cite{MSWW:Ends}, where each $(A_{\infty}, \Psi_{\infty})$ is assumed to 
take a particular normal form in disks  $\D_p$ around each zero of $q$. 
This normal form is given on each  by  $\D_p$ by the Fuchsian limiting configuration $(A_{\infty}^0(q),\Phi_{\infty}(q))$ and identically vanishing $\eta\equiv0$. With this restriction, we divide out by unitary gauge transformations that are the identity near each $\D_p$ (\emph{cf}.\ \cite{MSWW:HitchinMetric}).  
\medskip\\
Second, since there is an equivalence up to
 positive real multiples of $\Psi$, it is natural to define the projection:
\begin{equation} \label{eqn:H-infty}
\Hscr_\infty : \partial M^\ast_{\Higgs}(X)\lra \SQD^\ast(X)
\end{equation}
defined by mapping $(A_\infty, \Psi_\infty)\mapsto q/\Vert q\Vert_1$, where $q=-\det\Phi_\infty$.
\medskip\\
We now describe the structure of the set  $\partial M^\ast_{\Higgs}(X)$ of limiting configurations more closely, summarizing the results in \cite[\textsection 4.4]{MSWW:Ends}. For  $(A_{\infty}, \Psi_{\infty})\in  \Hscr_\infty^{-1}(q)$, define the  real line bundle $L_{q}\to X_q^{\times}$  by
\begin{equation}\label{eq:lbundlephiinfty}
L_{q}=\{\eta\in  \g_E\mid [\Phi_{\infty}\wedge\eta]=0\}\ .	
\end{equation}
Let $L_{q}^{\C}=L_{q}\otimes_\RBbb\CBbb$ denote the complexification. Then  $L_{q}$ and $L_{q}^{\C}$ are  $d_{A_{\infty}}$-invariant line subbundles of $\g_E$ and $ \g_E^{\C}$, real and complex respectively. Notice that the second component $\Phi_{\infty}$ of a limiting configuration is completely determined modulo unitary gauge by the holomorphic quadratic differential $q$. Hence, the flat bundle $L_{q}$ also only depends on $q$, which justifies the notation. The ungauged vertical deformation space 
  at $(A_{\infty}, \Phi_{\infty})$ is identified with
  \begin{equation*}
  	Z^1(X_q^{\times};L_{q}):=\{\eta\in\Omega^1(X_q^{\times},L_{q})\mid d_{A^0_{\infty}}\eta=0\}.
  \end{equation*}
Next consider the subgroup $\Stab_{\Phi_{\infty}}$ of unitary gauge transformations which stabilize $\Phi_{\infty}$. If $g \in \Stab_{\Phi_{\infty}}$ lifts to a section of $L_{q}$, \emph{i.e.}, $g = \exp(\gamma)$, $\gamma\in\Omega^0(X_q^{\times},L_{q})$, 
then $g$ acts on $A_\infty=A_\infty^0+\eta$,  $\eta\in\Omega^1(X_q^{\times},L_{q})$ by
$$
g(A_{\infty})=g^{-1}\eta g+g^{-1}(d_{A_{\infty}}g)=\eta+d_{A^0_{\infty}}\gamma
$$
(recall that $L_{q}$ is an $A_{\infty}$-parallel line subbundle of $\g_E$, so $g^{-1}\eta g=\eta$ and $d_{A_{\infty}}\exp(\gamma)=\exp(\gamma)d_{A_{\infty}}\gamma$). Hence the infinitesimal vertical deformation space is 
$$
H^1(X_q^{\times};L_{q})=Z^1(X_q^{\times};L_{q})/B^1(X_q^{\times};L_{q})\ ,
$$
where 
\begin{equation*}
B^1(X_q^{\times};L_{q}):=\{d_{A_{\infty}}\gamma\mid\gamma\in\Omega^0(X_q^{\times}, L_{q})\}\ .
\end{equation*}
If all zeroes of $q$ are simple, then 
\begin{equation*}
\dim_{\R} H^1(X_q^{\times}; L_{q}) = 6g-6\ ,
\end{equation*}
where $g$ is the genus of $\Sigma$. To obtain the moduli space, we must also divide the infinitesimal deformation space by the residual action of the component group 
$\pi_0(\Stab_{\Phi_{\infty}})$.  Under the correspondence above, this consists of an integral lattice $H^1_\ZBbb(X_q^\times, L_{q})$ under the exponential map. 

\begin{prop}\label{bound.stratum}
The moduli space of limiting configurations with a fixed  $q\in\QD^{\ast}(X)$  is
$$
\Hscr_\infty^{-1}(q)=H^1(X_q^\times, L_{q})/H^1_\ZBbb(X_q^\times, L_{q})\ .
$$
This is a torus of real dimension $6g-6$.
\end{prop}

\subsubsection{Approximate solutions}\label{subsect:approxsolutions}

 Following \cite[\textsection{3.2}]{MSWW:Ends}, for suitable functions $f$, $h$, and $\chi$ to be specified below,  we define the family of  \textbf{approximate solutions} 
$S_t^{\app}(q):=(A_t^{\app}(q)+\eta,t\,\Psi_t^{\app}(q))$ by  
\begin{eqnarray}\label{eq:atappr}
\begin{aligned}
	 \hspace{10mm} A_t^{\app} (q) &:= A_0 + \bigl(\tfrac{1}{2}+\chi(\Vert q\Vert)(4 f_t(\Vert q\Vert)-\tfrac{1}{2})\bigr)\Im\, \bar \partial \log \Vert q\Vert \begin{pmatrix}  -i & 0 \\ 0 & i \end{pmatrix}\ , \\[0.5ex]
	 \Phi_t^{\app}(q) &:= \begin{pmatrix}  0 & \Vert q\Vert^{1/2}e^{\chi(\Vert q\Vert)h_t(\Vert q\Vert)}
	 \\
	 \Vert q\Vert^{-1/2} e^{-\chi(\Vert q\Vert)h_t(\Vert q\Vert)}q
	  & 0 
 \end{pmatrix}\ , \\[0.5ex]
   \Psi^{\app}_{t}(q) &:= \Phi_t^{\app}(q)+(\Phi_t^{\app}(q))^{\ast_h}
   \ .
\end{aligned}
\end{eqnarray}
Regarding the formula for $\Psi^{\app}_{t}$, we follow our convention that $\Phi=\Psi^{1,0}$ (\emph{cf}.\ the beginning of \textsection{\ref{subsec:mswwComp}}). We may view these approximate solutions as desingularizations of the limiting configurations introduced before. Indeed, as $t\to \infty$ there is smooth local  convergence $S_t^{\app}(q)\to ( A^0_{\infty} (q)+\eta,\Phi_{\infty}(q))$  on $X_q^{\times}$. Here the $1$-form $\eta$ satisfies  \eqref{eq:genlimitingconf}.

\medskip
We now turn to a more detailed explanation of the functions used to define the approximate solution  in  \eqref{eq:atappr}. Here $h_t(r)$ is the unique solution to $(r\partial_r)^2 h_t = 8  t^2 r^3 \sinh (2h_t)$ on $\R^+$ with specific asymptotic 
properties at $0$ and $\infty$, and $f_t:=\frac{1}{8}+\frac{1}{4}r\partial_r h_t$. Further, $\chi: \R^+ \to [0,1]$ is a suitable cutoff-function. The parameter $t$ can be removed from the equation for $h_t$ by substituting $\rho = \frac{8}{3} t r^{3/2}$; thus
if we set $h_t(r) = \psi(\rho)$ and note that $r\partial_r = \frac{3}{2} \rho \partial_{\rho}$, then
\[
(\rho \partial_{\rho})^2 \psi = \frac{1}{2} \rho^2 \sinh (2\psi)\ .
\]
This is a Painlev\'e III equation; there exists a unique solution which decays exponentially as $\rho \to \infty$ and with 
asymptotics as $\rho \to 0$ ensuring that $A_t^{\app}$ and $\Phi_t^{\app}$ are regular at $r=0$. More specifically, 
\[
\begin{array}{rl}
\bullet\ & \psi(\rho) \sim -\log (\rho^{1/3} \left( \sum_{j=0}^{\infty} a_j \rho^{4j/3}\right)\ , \quad \rho \downarrow 0; \\[0.5ex]
\bullet\ & \psi(\rho) \sim K_0(\rho) \sim \rho^{-1/2} e^{-\rho}\sum_{j=0}^{\infty} b_j \rho^{-j}\ , \quad \rho \uparrow \infty; \\[0.5ex]
\bullet\ & \psi(\rho)\mbox{ is monotonically decreasing (and strictly positive) for $\rho > 0$}.
\end{array}
\]
These are asymptotic expansions in the classical sense, \emph{i.e.}, the difference between the function and the first $N$ terms
decays like the next term in the series, and there are corresponding expansions for each derivative. 
The function $K_0(\rho)$ is the Bessel function of imaginary argument of order $0$. 
\medskip\\
In the following result, any constant $C$ which appears in an estimate is assumed to be independent of $t$.

\begin{lem}\cite[Lemma 3.4]{MSWW:Ends}\label{f_t-h_t-function}
The functions $f_t(r)$ and $h_t(r)$ have the following properties: 
\begin{compactenum}[(i)] 
\item As a function of  $r$, $f_t$ has a double zero at $r=0$ and increases monotonically from $f_t(0) = 0$ to the limiting value $1/8$ as 
$r \uparrow \infty$.  In particular, $0 \leq f_t \leq \frac{1}{8}$.
\item As a function of $t$, $f_{t}$ is also monotone increasing. Further, $\lim_{t \uparrow \infty} f_t = f_{\infty} \equiv \frac{1}{8}$ 
uniformly in $ C^{\infty}$ on any half-line $[r_0,\infty)$, for $r_0 > 0$. 
\item There are estimates 
\[
\sup_{r >0}  r^{-1} f_t(r) \leq C t^{2/3} \quad \text{and}\quad \sup_{r >0} r^{-2} f_t(r) \leq C t^{4/3}\ .
\]
\item When $t$ is fixed and $r \downarrow 0$, then $h_t(r) \sim -\tfrac{1}{2} \log r + b_0 + \ldots$, where $b_0$ is an explicit constant. 
On the other hand, $|h_t(r)| \leq C \exp( -\tfrac{1}{8} t r^{3/2})/ ( t r^{3/2})^{1/2}$  for $t \geq t_0 > 0$, $r \geq r_0 > 0$.    
\item Finally,
\[
\sup_{r \in(0,1)} r^{1/2} e^{\pm h_t(r)} \leq C\ , \quad t \geq 1\ . 
\]
\end{compactenum} 
\end{lem}

It follows from the results in \cite{MSWW:Ends} that the approximate solution  $S_t^{\app}$  satisfies the self-duality equations up to an 
exponentially decaying error as $t\to\infty$ (which is uniform on the closed surface $X$), and there is an exact solution $(A_t,t\Phi_t)$ in its complex gauge orbit (unique up to real gauge transformations) which is no further than $Ce^{-\beta t}$ pointwise away (w.r.t.\ any $C^{\ell}$ norm) for some $\beta > 0$.

\subsubsection{Converging families of Higgs pairs}

For a holomorphic quadratic differential $q\in \mathcal{QD}^{\ast}(X)$, 
recall the fiber $\Hscr^{-1}(q)$, where $\Hscr$ is the Hitchin map \eqref{eqn:Hitchin-map},

\begin{definition} \label{def:ConvergenceOfHiggsPairs}\upshape
	Consider a family $[(A_{t},t\Psi_{t})]\in  \Hscr^{-1}(t^{2}\,q_t)$, where $q_t\in \SQD^\ast(X)$ and $q_t\to q\in \SQD^\ast(X)$.
	Then $[(A_{t},t\Psi_{t})]$ is said to \textbf{converge} to
	$
		[(A_{\infty},\Psi_{\infty})] \in \Hscr_\infty^{-1}(q)
	$
 as $t \to \infty$ if, after passing to a subsequence and modifying by unitary gauge transformations (which we suppress from the notation), the family of   pairs $(A_{t},\Psi_{t})$ satisfies the following:
\begin{description}[topsep=2ex,itemsep=2ex,itemindent=-5mm]
	\item[(Convergence)] The sequence $(A_{t},\Psi_{t})$ converges to $(A_{\infty},\Psi_{\infty})$ as $t \to \infty$ in $L^{p}(X)$ for all  $1\leq p<2$, locally in $C^{\ell}(X_{q}^{\times})$ for all $\ell\geq0$ at an exponential rate in $t$.
	\item[(Singularities)] For every zero $p\in Z(q_t)$, locally on the punctured disk $\D_{p}^{\times}$ (equipped with polar coordinates $(r,\theta)$) the connections $A_{t}$ are in radial gauge:
\[
A_{t} = F_{t} \begin{pmatrix}-i&0\\0&i \end{pmatrix} d\theta
\]
for some uniformly $C^{0}$-bounded family of smooth functions $\map{F_{t}}{\D_{p}}{\R}$ such that $F_{t} \to F_{\infty}$ pointwise for some smooth function $\map{F_{\infty}}{\D_{p}}{\R}$ as $t\to \infty$.
	\item[(Approximation)] For every integer $\ell\ge 0$ there exist constants $\beta, C>0$, not depending on $t$, and a $1$-form $\eta_t \in \Om^{1}(X, \g_E)$ satisfying  $[\Phi_{\infty}^{\app}(q_t), \eta_t]=0$ and $d_{A^{\app}_{\infty}(q_t)}\eta_t=0$
	such that
\[
\big\lVert (A_{t},\Psi_{t}) - (A^{\app}_{t}(q_t) + \eta_t,\Psi^{\app}_{t}(q_t)) \big\rVert_{C^{\ell}(X)} \leq C   e^{-\beta \, t}
\]
for all $t>0$.
\end{description}
\end{definition}

\begin{thm}[\cite{MSWW:Ends}] \label{thm:ConvergenceOfHiggsPairs}
	Every family  $[(A_{t},t\,\Psi_{t})] \in \Hscr^{-1}(t^{2}\,q_t)$ with $q_t\in K\subseteq \SQD^\ast(X)$, where $K$ is any compact subset, subconverges to a limiting pair $[(A_{\infty},\Psi_{\infty})] \in \Hscr^{-1}_\infty(q_{\infty})$ as $t \to \infty$ in the sense of Definition \ref{def:ConvergenceOfHiggsPairs}. Conversely, every limiting configuration arises in this way.
\end{thm}

\begin{proof}
By compactness of the set $K$, one has subconvergence $q_t\to q_{\infty}$ for some $q_{\infty}\in K$. The main part of the assertion follows from \cite[Thm.\ 6.6]{MSWW:Ends} 	 and its proof, which yields the (Approximation) axiom for any such family of Higgs pairs. There, only a polynomial bound for the $C^{\ell}$-norm of the difference $(A^{\app}_{t}(q_t) + \eta_t,\Phi^{\app}_{t}(q_t)) - (A_{t},\Phi_{t})$ is stated, but the proof shows that it can be improved to an exponential bound. The other two axioms then follow, since they are satisfied by the approximating family $ (A^{\app}_{t}(q_t) + \eta_t,t\,\Phi^{\app}_{t}(q_t))$ (again by the construction in \cite{MSWW:Ends}), and therefore also by $[(A_{t},t\,\Phi_{t})]$. The last statement is contained in \cite[Thm.\ 1.2]{MSWW:Ends}. 
\end{proof}

\subsection{Spectral curves} \label{sec:spectralcurve}
 
\subsubsection{The BNR correspondence} \label{subsec:BNR correspondence}
 Let $\pi: |K_X|\to X$ be the projection from the total space $|K_X|$ of   $K_X$. 
 There is a tautological section $\lambda$ of the holomorphic line bundle $\pi^\ast K_X\to |K_X|$. 
 Given $q\in\QD^{\ast}(X)$, the pullback $\pi^\ast q$ is a section of $\pi^\ast K_X^2\to |K_X|$.
 Let
\begin{equation} \label{eqn:spectral_curve}
\widehat X_q=\{\hat x\in |K_X|\mid \lambda^2(\hat x)=\pi^\ast q(\hat x)\}\subset |K_X|\ .
\end{equation}
Then $\widehat X_q$ is a  compact Riemann surface 
 called the \textbf{spectral curve} associated to  $q$ (nonsingular, since $q \in \QD^{\ast}(X)$ has simple zeroes). The restriction of the projection, $\pi: \widehat X_q \to X$, realizes $\widehat X_q$ as a ramified double covering of $X$ with simple branch points at the zeroes $Z(q)$ of $q$. By the Riemann-Hurwitz formula, $\widehat X_q$ has genus $4g-3$. Moreover, $\widehat X_q$ admits an involution   $\hat x\mapsto -\hat x$  which we denote by  $\sigma$.

Recall that the {\bf Prym variety} associated to the covering is
 \begin{equation*}
 	\Prym(\widehat X_q, X)=\{\Lcal\in\Pic(\widehat X_q)\mid\sigma^*\Lcal\simeq\Lcal^*\}\ .
 \end{equation*}

 \begin{thm}[{\cite[Prop.\ 3.6]{BNR}}] \label{thm:bnr}
There is a 1-1 correspondence between points in $\Prym(\widehat X_q, X)$ and isomorphism classes of Higgs bundles $(\Ecal,\Phi)$ with $\det\Phi=-q$.
 \end{thm}
The association in Theorem \ref{thm:bnr} goes as follows, see also  \cite[\textsection 2.2]{MSWW:HitchinMetric}}.  Recall that we have fixed a square root $K_X^{1/2}$. Given $\Lcal \in \Prym(\widehat X_q, X)$, let $\Ucal=\Lcal\otimes \pi^\ast(K_X^{1/2})$. 
Then $\Ecal=\pi_\ast(\Ucal)$ is a rank $2$ holomorphic bundle on $X$ with trivial determinant, and multiplication by $\lambda$ gives a map:
$$
\Phi: 
\Ecal=\pi_\ast(\Ucal)\stackrel{\lambda}{\lra} \pi_\ast(\Ucal\otimes\pi^\ast(K_X))=\Ecal\otimes K_X\ ,
$$
with $\det\Phi=-q$. 
In the other direction, given a Higgs bundle $(\Ecal, \Phi)$, $\Ucal$ is defined by the exact sequence (\emph{cf}.\ \cite[Rem.\ 3.7]{BNR})
\begin{equation} \label{eqn:bnr}
0\lra \Ucal\otimes \Ical_Z\lra \pi^\ast(\Ecal)\stackrel{\lambda-\pi^\ast\Phi}{\xrightarrow{\hspace*{1cm}}}\pi^\ast(\Ecal\otimes K_X)\lra \Ucal\otimes \pi^\ast(K_X)\lra 0\ ,
\end{equation}
where $\Ical_Z$ is the ideal sheaf of $Z=Z(q)$ and we regard $\pi^\ast\Phi$ as a holomorphic section of $\pi^\ast(\End_0\Ecal\otimes K_X)$. 
Since the details of this will be important in the sequel, we briefly elaborate eq.\ \eqref{eqn:bnr}.
The first thing to note is that we have  an exact sequence
\begin{equation} \label{eqn:pullback}
0\lra \pi^\ast(\Ecal)\lra \Ucal\oplus \sigma^\ast(\Ucal)\lra \Ucal\otimes \Ocal_Z\lra 0\ .
\end{equation}
The last map is given by mapping sections $(u,v)\in \Ucal\oplus \sigma^\ast(\Ucal)$
to $u(p)-v(p)$, for $p\in Z$. To prove this statement, let $A\subset \widehat X_q$ be an open set. Then by definition, as $\Ocal_A$-modules,
$$
\pi^\ast(\Ecal)(A)=\pi_\ast(\Ucal)(\pi(A))\otimes_{\Ocal_{\pi(A)}} \Ocal_A=\Ucal(\pi^{-1}\pi(A))\otimes \Ocal_A=\Ucal(A\cup \sigma(A))\otimes \Ocal_A\ .
$$
Now, as an $\Ocal_A$-module, $\Ucal(\sigma(A))=\sigma^\ast(\Ucal)(A)$. Hence, local sections of $\pi^\ast(\Ecal)$ are sections of $\Ucal\oplus \sigma^\ast(\Ucal)$ that agree at $Z$; thus, \eqref{eqn:pullback}.

Let $(u,v)\in \pi^\ast(\Ecal)\subset \Ucal\oplus\sigma^\ast(\Ucal)$. 
Now $\Phi$ acts by multiplication by $\lambda$. Since $\sigma^\ast\lambda=-\lambda$, we have $\pi^\ast\Phi(u,v)=(\lambda u, -\lambda v)$. Therefore, $(u,v)\in \ker(\lambda-\pi^\ast\Phi)$ if and only if $v=0$. The condition in \eqref{eqn:pullback} forces the image to consist of sections $u$ of $\Ucal$ that vanish at $Z$, which is the first term in \eqref{eqn:bnr}.
 Similarly, the image of $\lambda-\pi^\ast\Phi$ consists of local sections of the form
  $(0,2\lambda v)$, \emph{i.e.}\ sections of $\sigma^\ast(\Ucal)\otimes \Ical_Z\otimes \pi^\ast(K_X)$. This is precisely the kernel of the projection $\pi^\ast(\Ecal)\otimes \pi^\ast(K_X)\to \Ucal \otimes \pi^\ast(K_X)$ given by projection onto the first factor. This proves exactness of \eqref{eqn:bnr}. 

\begin{remlabel} \label{rem:spec-curve-scale} \upshape
The following will be important.
\begin{compactenum}[(i)]
	\item For $t>0$ there is a natural biholomorphism $f_t: \widehat X_q\to \widehat X_{t^2q}$, and pulling back line bundles gives an isomorphism $\Prym(\widehat X_{t^2q},X)\isorightarrow \Prym(\widehat X_q,X)$.  Under this correspondence and Theorem \ref{thm:bnr}, $(\Ecal, t\Phi)\mapsto (\Ecal,\Phi)$.
	\item
Note that $d\pi$ is a holomorphic section of $K_{\widehat X_q}\otimes\pi^\ast K_X^{-1}$ with simple zeroes at $Z(q)$. Since $\lambda\in H^0(\widehat X_q,\pi^\ast K_X)$ also vanishes at $Z(q)$, it follows that $K_{\widehat X_q}=\pi^\ast K_X^2$. 
\end{compactenum}
\end{remlabel}

 \subsubsection{Prym differentials}\label{subsect:prymdifferentials}
Let $\pi: \widehat X_q\to X=\widehat X_q/\sigma$, be as in the previous section. Recall  the exponential sequence
$$
0\lra 2\pi i\ZBbb\lra \Ocal_{\widehat X_q}\xrightarrow{\exp}\Ocal^\ast_{\widehat X_q}\lra 1\ ,
$$
and associated long exact sequence in cohomology:
$$
0\lra H^1(\widehat X_q,2\pi i\ZBbb)\lra H^1(\widehat X_q,\Ocal_{\widehat X_q})\lra H^1(\widehat X_q,\Ocal_{\widehat X_q}^\ast)\stackrel{2\pi ic_1}{\xrightarrow{\hspace*{.75cm}}}H^2(\widehat X_q, 2\pi i\ZBbb)\lra 0\ .
$$
This gives an identification:
\begin{equation*} \label{eqn:pic0}
p: H^1(\widehat X_q,\Ocal_{\widehat X_q})/H^1(\widehat X_q,2\pi i \ZBbb)\isorightarrow \Pic_0(\widehat X_q):=\ker c_1\subset H^1(\widehat X_q,\Ocal_{\widehat X_q}^\ast)\ .
\end{equation*}
Via the Dolbeault isomorphism, we obtain an isomorphism:
\begin{equation} \label{eqn:dolbeault}
\delta : H^{0,1}_{\dbar}(\widehat X_q)/H^1(\widehat X_q,2\pi i \ZBbb)\isorightarrow
H^1(\widehat X_q,\Ocal_{\widehat X_q})/H^1(\widehat X_q,2\pi i \ZBbb)\ .
\end{equation}
Now, consider a $\dbar$-operator $\dbar_L=\dbar+\alpha$ on a trivial complex line bundle $L$, where  $\alpha\in \Omega^{0,1}(\widehat X_q)$. Let $\Lcal$ denote the associated holomorphic bundle. Then $\alpha$ defines a class $[\alpha]\in H^{0,1}_{\dbar}(\widehat X_q)/H^1(\widehat X_q,2\pi i \ZBbb)$, and $\Lcal$ defines a class $[\Lcal]\in \Pic_0(\widehat X_q)$. We have the following well known result.
\begin{lem} \label{lem:dolbeault}
 $p\circ\delta(-[\alpha])=[\Lcal]$.
\end{lem}

The map $\alpha\mapsto \alpha-\bar\alpha$
gives a real isomorphism
$
 H^{0,1}_{\dbar}(\widehat X_q)\simeq
H^1(\widehat X_q, i\RBbb)
$.
Combined with the Lemma \ref{lem:dolbeault} we have
\begin{equation} \label{eqn:pic}
\Pic_0(\widehat X_q)\simeq H^1(\widehat X_q, i\RBbb)/H^1(\widehat X_q, 2\pi i\ZBbb)\ ,
\end{equation}
The involution $\sigma$ acts on $H^1(\widehat X_q, i\RBbb)$, giving a decomposition into even and odd cohomology:
$$
H^1(\widehat X_q, i\RBbb)=H^1_{\ev}(\widehat X_q, i\RBbb)\oplus H^1_{\odd}(\widehat X_q, i\RBbb)\ .
$$
 Clearly, 
$H^1_{\ev}(\widehat X_q, i\RBbb)\simeq H^1( X, i\RBbb)
$.
Let 
\begin{equation} \label{eqn:odd-integral}
H^1_{\odd}(\widehat X_q, 2\pi i\ZBbb):=H^1(\widehat X_q, 2\pi i\ZBbb)\cap
H^1_{\odd}(\widehat X_q, i\RBbb)\ .
\end{equation}
Using \eqref{eqn:pic}, we see that there is an isomorphism
\begin{equation} \label{eqn:prym-prym}
\Prym(\widehat X_q,X)\simeq H^1_{\odd}(\widehat X_q, i\RBbb)/H^1_{\odd}(\widehat X_q, 2\pi i\ZBbb)\ .
\end{equation}

Canonical representatives of elements of $H^1_{\odd}(\widehat X_q, i\RBbb)$
 are given by odd, imaginary,  harmonic forms,
 and the space of such will be denoted by $\Hcal^1_{\odd}(\widehat X_q, i\RBbb)$.
We call $\Hcal^1_{\odd}(\widehat X_q, \CBbb)$
the space of {\bf harmonic Prym differentials}.
 Let $H^0_{\odd}(\widehat X_q, K_{\widehat X_q})$
denote the space of holomorphic differentials on $\widehat X_q$ that are odd with respect to the involution.
We shall call   $H^0_{\odd}(\widehat X_q, K_{\widehat X_q})$  the space of 
{\bf holomorphic Prym differentials}\footnote{The terminology we use here is somewhat nonstandard: Prym differentials as odd classes are usually defined for unramified covers (see \cite{Gunning67,Mumford:74}). Here we follow   \cite[p.\ 86]{Fay:73}.}.
 We have an isomorphism:
\begin{equation} \label{eqn:prym-harmonic-holomorphic}
\Hcal^1_{\odd}(\widehat X_q, i\RBbb)\isorightarrow
H^0_{\odd}(\widehat X_q, K_{\widehat X_q}) : \widehat\eta\mapsto \widehat\eta^{1,0} \ .
\end{equation}

There is a distinguished nontrivial holomorphic 
Prym differential associated to the Liouville form on $|K_X|$. The dual of the tangent sequence associated to $\pi$ gives
$$
0\lra \pi^\ast K_X\stackrel{(d\pi)^t}{\xrightarrow{\hspace*{.75cm}}}T^\ast|K_X|\lra \pi^\ast K_X^{-1}\lra 0\ .
$$
The (holomorphic) Liouville $1$-form on $|K_X|$ is by definition the image of the section $\lambda\in H^0(|K_X|,\pi^\ast K_X)$ under this exact sequence, or in other words,
$(d\pi)^t\circ\lambda=\lambda\circ d\pi$. Its restriction to $\widehat X_q$ is a holomorphic $1$-form on $\widehat X_q$ that is odd with respect to the involution. We call this the {\bf Seiberg-Witten differential}: 
\begin{equation} \label{eqn:SW-differential}
\lambda_{\mathsf{SW}}:=\lambda\otimes d\pi\ .
\end{equation}
We shall see in Corollary~\ref{cor:thurston's-pleated-surface} that the spectral data in $\Prym(\widehat X_q,X)$ associated to the Seiberg-Witten differential  are closely related to complex projective structures.

\begin{remlabel} \label{rem:SWdifferential} \upshape
It is customary in the literature to suppress $d\pi$ from the notation in \eqref{eqn:SW-differential} and denote the form $\lambda_{\mathsf{SW}}$ on $\widehat X_q$ simply by $\lambda$. Note that the latter is a section of $\pi^\ast K_X$ and not $K_{\widehat X_q}\simeq \pi^\ast K_X^2$. 
Since both of these differentials will figure prominently below, we prefer to keep the notational distinction.
\end{remlabel}

\subsubsection{Prym differentials and limiting configurations} \label{sec:prym-limiting}
Continue with the notation of the previous section.
Let us introduce 
\begin{equation} \label{eqn:W1}
W_2=\left(\begin{matrix}0&\lambda^{-1}\Vert \lambda\Vert \\ \lambda\Vert\lambda\Vert^{-1}&0\end{matrix}\right)\in \End(\pi^\ast E)\ .
\end{equation}
Then we may write  
\begin{equation} \label{eqn:pullback-Phi}
\pi^\ast\Phi_\infty=\lambda_{\SW}\otimes W_2\ ,
\end{equation} where $\Phi_\infty$ is defined in
 \eqref{eq:stdlimconn}, $\lambda_{\SW}$ in \eqref{eqn:SW-differential}, and we emphasize that here we regard $\pi^\ast\Phi$ as the pullback of an endomorphism valued $1$-form.
It follows that $W_2$ lies in $\pi^\ast L_{q}^\CBbb$. 
 Moreover, $W_2$ is hermitian, and by a direct computation
(cf.\ the proof of  Proposition \ref{lem:W-1}) we have that $d_{\widehat A^0_\infty}W_2=0$.

Let $\hat\eta\in \Hcal^1_{\odd}(\widehat X_q, i\RBbb)$. Then, because $\hat\eta\otimes W_2$ commutes with $\Phi_{\infty}$ we see that $\eta=\hat\eta\otimes W_2\in \Omega^1(\widehat X_q^\times , \pi^\ast L_{q})$ is  invariant with respect to $\sigma$, and so $\eta$ descends to $X$. By the flatness of $W_2$, the form $\eta$ is also $d_{A^0_\infty}$-harmonic. Hence, it defines a class in $H^1(X^{\times}, L_{q})$. 
 Notice that  $\Vert \eta\Vert$ is bounded,  and therefore,  in $L^2$. Conversely, suppose $\eta$ is a $d_{A^0_\infty}$-harmonic form in  $\Omega^1(X^\times , L_{q})$ that is in $L^2$. 
Then we can write $\pi^\ast\eta=\hat \eta\otimes W_2$ for $\hat\eta$,  a pure imaginary form on $\widehat X_q^\times$ 
that is anti-invariant with respect to $\sigma$. Moreover, since $\eta$ is harmonic and
 in $L^2$, the form $\hat\eta$ satisfies $d\hat\eta=d^\ast\hat\eta=0$ weakly, and so by elliptic  regularity it is a smooth harmonic Prym differential. 
This leads to the following identification of harmonic   Prym differentials  with the space of limiting configurations.

\begin{prop}\label{prop:prym-limiting}
The  maps $\widehat X_q^\times\hookrightarrow \widehat X_q$, and  $\hat\eta\mapsto \eta$, $\pi^\ast\eta= \hat\eta\otimes W_2$   induce  isomorphisms 
\begin{equation} \label{eqn:prym-local-system}
	\Hcal^1_{\odd}(\widehat X_q,i\R) \cong H^1_{\odd}(\widehat X_q^{\times};i\R) \cong H^1(X^{\times};L_{q})
\end{equation}	
which send the integral lattices $\Hcal^1_{\odd}(\widehat X_q, 2\pi i\ZBbb)$ to $H^1_\ZBbb(X^\times , L_{q})$. 
Hence, combined with \eqref{eqn:prym-prym}, this gives an identification
\begin{equation} \label{eqn:prym-limiting}
\Prym(\widehat X_q, X)\simeq \Hscr_{\infty}^{-1}(q)
\end{equation}
which is natural with respect to scaling by $t>0$.
\end{prop}

\begin{proof}
The first isomorphism in \eqref{eqn:prym-local-system} was shown in {\cite{MSWW:HitchinMetric}}, and the second holds by the above discussion.	
  It remains to show that under these identifications 
  the lattices $\Hcal^1_{\odd}(\widehat X_q, 2\pi i\ZBbb)$ and $H^1_\ZBbb(X^{\times}, L_{q})$ are preserved.
Indeed, suppose $[\hat\eta]\in H^1_{\odd}(\widehat X_q, 2\pi i\ZBbb)$, and choose a representative $\hat\eta$ that is odd. 
Choose a base point $w_0\in Z\subset\widehat X_q$, and for  $w\in \widehat X_q$, let
$$
g(w)=\exp\left( \int_{w_0}^w\hat\eta\otimes W_2(w)\right)\ .
$$
Since $\hat\eta$ has $2\pi i\ZBbb$  periods, and
 \begin{equation} \label{eqn:periods}
 \exp\left( 2\pi ik W_2(w)\right)=I\ ,
 \end{equation}
 for $k\in {\mathbb Z}$, one sees that $g(w)$ is well defined independent of the path of integration from $w_0$ to $w$. Moreover, notice that 
\begin{equation}\label{eqn:int-eta}
\int_{w_0}^{\sigma(w)}\hat\eta=\int_{w_0}^w \sigma^\ast\hat\eta=-\int_{w_0}^w \hat\eta\quad\mod 2\pi i\ZBbb \ .
\end{equation}
Therefore,
\begin{align*}
g&(\sigma(w))g(w)^{-1}=
\exp\left\{  \int_{w_0}^{\sigma(w)}\hat\eta\otimes W_2(\sigma(w))-\int_{w_0}^w\hat\eta\otimes W_2(w)\right\} \\
&=
\exp\left\{ -  \left(\int_{w_0}^{\sigma(w)}\hat\eta+\int_{w_0}^w\hat\eta\right)\otimes W_2(w)\right\}\qquad{\rm since\ } W_2(\sigma(w))=-W_2(w)\\
&=I\qquad{{\rm by\; \eqref{eqn:int-eta}}}\text{ and }\eqref{eqn:periods}\ .
\end{align*}
Hence, $g$ is a well defined $U(1)$-gauge transformation on $X^\times$, and 
$
\hat\eta\otimes W_2=g^{-1}dg$.
Conversely, as mentioned in the discussion prior to Proposition \ref{bound.stratum},  the group
$H^1_{\mathbb Z}(X^\times, L_{q})$
of components of the stabilizer of $\Phi_{\infty}$ is generated by global gauge transformations of this form.  

The final statement holds, since if $\pi_t : \widehat X_{tq}\to X$, $\pi : \widehat X_q\to X$, $f_t: \widehat X_{q}\to \widehat X_{tq}$ is given by multiplication by $t^{1/2}$, then writing 
$$
\pi_t^\ast\eta=\widehat\eta_t\otimes W_2\ ,\ \pi^\ast\eta=\widehat\eta\otimes W_2\ ,
$$
it is easy to see that $f_t^\ast\widehat\eta_t=\widehat \eta$.
\end{proof}

\subsubsection{Limiting configurations and spectral data}
Recall the sequence \eqref{eqn:bnr}. In the case where $\Ucal=\pi^\ast(K_X^{1/2})$, we have $\Ecal=K_X^{-1/2}\oplus K_X^{1/2}$, and $\sigma^\ast(\Ucal)=\Ucal$. The isomorphism between the description of $\pi^\ast(\Ecal)$ in \eqref{eqn:pullback}  and the pullback of this bundle is given by:
$$
(u,v)\mapsto \left( \frac{1}{2}\lambda^{-1}(u-v), \frac{1}{2}(u+v)\right)\ .
$$
Note that the first factor on the right hand side above is regular because $u-v$ vanishes at the zeroes of $\lambda$.

In fact, the correspondence in \eqref{eqn:prym-limiting} occurs at the level of spectral data as well, in a manner we now describe. The image of the map
\begin{equation} \label{eqn:spectral_subbundle}
\pi^\ast(K_X^{1/2})\otimes \Ical_Z\lra \pi^\ast E : s\mapsto 
\left(
  \Vert\lambda\Vert^{1/2}\lambda^{-1}s, \Vert\lambda\Vert^{-1/2}s\right)
\end{equation}
 is the kernel of $\lambda-\pi^\ast\Phi_\infty$.   This is a \emph{holomorphic} embedding for a limiting configuration $d_{A_\infty}=d_{A^0_\infty}+\eta$ if and only if as a  holomorphic bundle, $\Ucal=\Lcal\otimes\pi^\ast(K_X^{1/2})$, and $\Lcal$ is the trivial bundle on $\widehat X_q$ with $\dbar$-operator determined by the $(0,1)$ part of $\widehat\eta$. 

Thus, the correspondence in Proposition \ref{prop:prym-limiting} is between limiting configurations and ``limiting spectral data''. To make sense of the latter, consider the following situation. Let $q_n\to q\in \SQD^\ast(X)$, and let $B\subset \SQD^\ast(X)$ be a  neighborhood of $q$. Then there is a smooth holomorphic fibration 
  $p:\widehat \Xcal\to B$ of complex manifolds, where for $b\in B$,  $p^{-1}(b)$ is the branched covering $\widehat X_b\to X$. For $j$ large, $q_n\in B$, and 
 the Gauss-Manin connection on $\widehat\Xcal$ gives an identification of Prym differentials 
on $\widehat X_{q_n}$ and $\widehat X_q$
 which preserves the integral lattice; and hence also an identification of spectral data. 
With this understood, we have the following.
\begin{thm} \label{thm:convergence_spectral_data}
Suppose
$\widehat\eta_n$ is a sequence of imaginary harmonic Prym differentials on $\widehat X_{q_n}$ converging to a differential $\widehat\eta$ on $\widehat X_q$, in the sense of the paragraph above.  Let $t_n\to +\infty$.  Let $(\Ecal_n, t_n\Phi_n)$ be the Higgs bundles associated to $\widehat\eta_n$ via 
the identification \eqref{eqn:prym-prym} and Theorem \ref{thm:bnr} (see also Remark \ref{rem:spec-curve-scale} (i)), and let $(A_n, \Psi_n)$ be the corresponding solutions to the self-duality equations. 
Then any  accumulation point of the sequence $(A_n, \Psi_n)$ in the space of limiting configurations is gauge equivalent to $(A_\infty^0+\eta, \Psi_\infty)$, where $\pi^\ast\eta=\widehat \eta\otimes W_2$. 
\end{thm}

\begin{proof}
Suppose without loss of generality that $(A_n,\Psi_n)$ converges to a limiting configuration $(A_\infty, \Psi_\infty)$. Then $(A_\infty, \Psi_\infty)$ is in the fiber $\Hscr^{-1}_\infty(q)$, and $A_\infty$ is gauge equivalent to a connection of the form $A_\infty^0+\eta_0$, where 
$\pi^\ast\eta_0=\widehat \eta_0\otimes W_2$ for some $\widehat\eta_0\in \Prym(\widehat X_q,X)$.
We must show $[\widehat\eta_0]=[\widehat\eta]$. For this, it suffices to show that  $\widehat\eta_0$ and $\widehat\eta$ have the same periods on the homology $H_1^{\odd}(\widehat X_q)$, modulo integers. 
For any class $[\widehat \gamma]\in H_1^{\odd}(\widehat X_q)$, we may choose a representative $\widehat\gamma\subset \widehat X^\times_q$. The pullback connections $\pi^\ast A_n$ converge to $\pi^\ast A_\infty$ in $C^\infty_{\textrm{loc}}$ on $\widehat X^\times_q$ with respect to the fibration $\widehat \Xcal$ introduced above. On the other hand, as discussed above,
the class $[\widehat\eta_n]$ of the spectral data for $(\Ecal_n, t_n\Phi_n)$ is determined by the restriction of $\pi^\ast A_n$ to the line subbundle in the embedding \eqref{eqn:spectral_subbundle}, and the same is true for $[\widehat\eta_0]$. Hence, convergence of the connections away from the branching locus implies the periods of $\widehat \eta$ and $\widehat \eta_0$ agree.
\end{proof}

Theorem \ref{thm:convergence_spectral_data}
 states that the partial compactification of
$\Hscr^{-1}(\SQD^\ast(X))$ via spectral data
 mentioned in the comment following
Corollary \ref{cor:invariance-of-basepoint} is compatible,
via Proposition \ref{prop:prym-limiting},   with
the description of ideal points in terms of  limiting
configurations. 

\subsection{Equivariant harmonic maps} \label{subsect:hitchineqharmmaps}

The goal of this section is to relate the Riemannian geometry of
the hyperbolic space $\HBbb^3$ to the gauge theory of Higgs
bundles. The main result is Theorem
\ref{thm:DonaldsonIsomorphism} below. All of this material is
standard and is explicitly or implicitly described in Hitchin
\cite{Hitchin:87} and Donaldson \cite{Donaldson:87}, and more
generally in Corlette \cite{Corlette:88}, Jost-Yau \cite{JostYau:91}, 
and Labourie \cite{Labourie:91}.
 Nevertheless, in order to make the exposition here self-contained
and to get the correct normalizations,  we wish to reformulate the general description to suit the purposes of this paper.

\subsubsection{Statement of the result}  \label{eq:resultdonaldsonharmonic}

 With a choice of lift $\widetilde p_0\in \widetilde X$ of $p_0\in X$,  the fundamental group $\pi_1=\pi_1(X,p_0)$ acts by deck transformations on $\widetilde X$. Given $\rho: \pi_1\to \SL(2,\CBbb) $, we say that a map $u:\widetilde X\to \HBbb^3$ is {\bf $\rho$-equivariant} if $u(\gamma z)=\rho(\gamma)u(z)$ for all $z\in\widetilde X$, $\gamma\in \pi_1$.  If $u$ is $C^2$, we say that $u$ is {\bf harmonic} if $d_{\nabla^{\LC}}^\ast du=0$. Here, we let $du\in \Omega^1(\widetilde X, u^\ast T\HBbb^3)$ denote the differential of the map $u$, and $\nabla^{\LC}$  the Levi-Civita connection on $\HBbb^3$. The key existence result is stated here.

\begin{thm}[\cite{Corlette:88,Donaldson:87,JostYau:91, Labourie:91}] \label{thm:DonaldsonCorlette}
 Suppose that the representation $\rho: \pi_1 \to \SL(2,\CBbb)$ is completely reducible. Then there exists a $\rho$-equivariant harmonic map $u:\widetilde X\to \HBbb^3$. If $\rho$ is irreducible, then $u$ is unique.
\end{thm}

The {\bf Hopf differential} of a  map $u:\widetilde X\to \HBbb^3$ is defined as the $(2,0)$-part of the pull-back of the metric tensor of $\HBbb^3$: 
\begin{equation} \label{eqn:hopf}
{\rm Hopf}(u)=(u^\ast ds_{\HBbb^3}^2)^{2,0}\ .
\end{equation}
A very important and classical fact is that  ${\rm Hopf}(u)$ is a holomorphic quadratic differential if $u$ is harmonic.

As before, let $E\to X$ be a hermitian rank $2$ vector bundle and  recall that  $\g_E$ and $\sqrt{-1}\g_E$ denote the bundles of traceless skew-hermitian and   hermitian endomorphisms of $E$, respectively. A central construction used in this paper is the following.

\begin{thm} \label{thm:DonaldsonIsomorphism}
Let $(A,\Psi)$
be an irreducible solution of the self-duality equations \eqref{eq:sde}.
Choose $p_0\in X$ and $\widetilde p_0\in \widetilde X$ as above and a unitary frame of the fiber $E_{p_0}$ of $E$ at $p_0$. 
 Let $\map{\rho}{\pi_{1}(X,p_{0})}{\SL(2,\CBbb)}$ be the holonomy representation 
 of the flat connection $\nabla=d_A+\Psi$.  
Then the unique $\rho$-equivariant harmonic map from Theorem \ref{thm:DonaldsonCorlette} satisfies  the following properties. 
\begin{compactenum}[(i)]
\item The pullback $u^{\ast}T\H^{3}$ descends to a bundle on $X$ that is isometrically isomorphic to $\sqrt{-1}\g_E$. Under this identification:
	\item The orthogonal connection $d_{A}$ on $\sqrt{-1}\g_E$ corresponds to the pull-back of the Levi-Civita connection $\nabla^{\LC}$ on $\H^{3}$.
	\item The hermitian $1$-form $-2\,\Psi \in \Omega^{1}(X,\sqrt{-1}\g_E)$ corresponds to the differential $du \in \Omega^{1}(X,u^{\ast}T\H^{3})$.
	\item The Higgs field $\Psi$ and the harmonic map $u$ determine the same quadratic differential in the sense that
$ \mathrm{Hopf}(u)=2\tr(\Psi\otimes\Psi)^{2,0}$.
\end{compactenum}
\end{thm}

\begin{remlabel} \upshape
Indeed, while our focus in this paper is on harmonic maps,  Theorem \ref{thm:DonaldsonIsomorphism} holds for general $\rho$-equivariant maps $u\colon \widetilde X\to \HBbb^3$, as will become clear from the  discussion below.
\end{remlabel}

The remainder of \textsection{\ref{subsect:hitchineqharmmaps}} is devoted to the proof of Theorem \ref{thm:DonaldsonIsomorphism}.

 \subsubsection{The matrix model of $\HBbb^3$}\label{subsect:matrixmodel}
 We view the hyperbolic space $\HBbb^3$ as the homogeneous space $\SL(2,\CBbb)/\SU(2)$.  The latter may in turn be identified with 
$$
\Dscr=\left\{ h\in {\rm Mat}_{2\times 2}(\CBbb) \mid h=h^{\ast}, \, \det h=1, \, h > 0 \right\},
$$
where the identification maps the coset $[g]\mapsto gg^{\ast}$.  Note that the left action by $\SL(2,\CBbb)$ then corresponds to $g\cdot h=ghg^{\ast}$, and that  $\Dscr$ has a distinguished point corresponding to $h=\id$. 

The tangent space is given by
\begin{equation} \label{eqn:h-model}
T_h\HBbb^3\simeq T_h\Dscr=\left\{ H\in {\rm Mat}_{2\times 2}(\CBbb) \mid H=H^{\ast},\ \tr(Hh^{-1})=0\right\}.
\end{equation}
It will be useful to have another description of the tangent space as
\begin{equation} \label{eqn:t-model}
T_h\HBbb^3\simeq \left\{ K\in {\rm Mat}_{2\times 2}(\CBbb) \mid (Kh)^{\ast} =Kh,\ \tr(K)=0  \right\}. 
\end{equation}
The correspondence between the two descriptions is given by $H\mapsto K=Hh^{-1}$.  We shall refer to \eqref{eqn:h-model} as the \textbf{hermitian model} and to \eqref{eqn:t-model} as the \textbf{traceless model}.
From the traceless model, we see that the complexification $T\HBbb^3\otimes \CBbb \cong \HBbb^3\times \CBbb^3$ is trivial, and the fiber is identified with the space of traceless $2\times 2$ complex matrices. The real bundle $T\HBbb^3$ is recovered as the fixed point set of the complex antilinear map $\tau=\tau_{h}$ given by
\begin{equation} \label{eqn:tau}
\tau_{h}(M)=hM^{\ast} h^{-1}.
\end{equation}

The invariant constant curvature $-1$ Riemannian metric on $\HBbb^3$ is defined by
\begin{equation} \label{eqn:metric-on-H3}
	\langle H_1, H_2 \rangle_{\H^3,h} = \frac{1}{2}\tr(H_1h^{-1} H_2h^{-1})
\end{equation}
for $H_i\in T_h\HBbb^3$ in the hermitian model. If we define the hermitian structure on $T\HBbb^3\otimes \CBbb$ by
\begin{equation} \label{eqn:M-metric}
\langle M_1, M_2\rangle_{\H^3,h} = \frac{1}{2}\tr(M_1 h M_2^{\ast} h^{-1})\ ,
\end{equation}
then the map $H\mapsto K$ between models is a real isometry for the induced metric on the fixed point set of $\tau$. 

\begin{lem} \label{lem:levi-civita}
In the  traceless model the Levi-Civita connection of $\HBbb^3$ is given by
\begin{equation} \label{eqn:levi-civita}
{\nabla^{\LC}} K = dK-\frac{1}{2}[dh h^{-1}, K]\ .
\end{equation}
\end{lem}

\begin{proof}
On \cite[p.\ 129]{Donaldson:87} it is shown that the connection in the hermitian model is given by
\begin{equation} \label{eqn:levi-civita-hermitian}
{\nabla^{\LC}} H= dH-\frac{1}{2}(dh h^{-1} H + Hh^{-1}dh)\ ,
\end{equation}
for $H\in T_h\HBbb^3$.
Pulling back this connection to the traceless model means computing ${\nabla^{\LC}}(Kh)h^{-1}$, and this gives \eqref{eqn:levi-civita}.
\end{proof}

\subsubsection{Flat bundles on $\HBbb^3$}  We define a rank $2$ hermitian bundle $V\to \HBbb^3$ using the homogeneous space description. More precisely, endow $\SU(2)$ with a right action on $\CBbb^2$: $v \cdot h=h^{-1}v$, for $v\in \CBbb^2$, $h\in \SU(2)$, and then define
$$
V=\left(\SL(2,\CBbb)\times \CBbb^2\right)/\SU(2)\ ,
$$
for the diagonal action. Smooth sections of $V$ then correspond to functions $s\colon \SL(2,\CBbb)\to \CBbb^2$ satisfying: $s(gh)=h^{-1}s(g)$ for $h\in \SU(2)$.  A hermitian structure on $V$ is then derived from the standard inner product on $\CBbb^2$. 
We define a connection on $V$ as follows: $(\widehat\nabla s)(g) = ds(g)+g^{-1}dg\cdot s(g)$.
One easily  verifies that this connection is well defined and \emph{flat}.

Now consider the bundle $\End_0 V\to\HBbb^3$ of traceless endomorphisms of $V$, with its flat connection  $\widehat \nabla$ induced from the connection on $V$ described in the previous paragraph.  This is a rank $3$ complex vector bundle. Recall from the previous section that the trivial bundle $T\HBbb^3\otimes\CBbb$ is also a rank $3$ complex hermitian bundle. We endow it with the trivial connection: $\nabla^\CBbb M:=dM$.

\begin{prop} \label{lem:V-M}
There is a bundle isometry $\End_0 V\isorightarrow T\HBbb^3\otimes\CBbb$ which intertwines the flat connections $\widehat\nabla$ and $\nabla^\CBbb$.
\end{prop}

\begin{proof}
Endomorphisms $T$ of $V$ are given by functions $T:\SL(2,\CBbb)\to \End_0 \CBbb^2$
such that on sections $s$, $(Ts)(g)=T(g)s(g)$.  Equivariance with respect to $\SU(2)$ implies, $(Ts)(gh)=h^{-1}(Ts)(g)$, or
$$
h^{-1}T(g)s(g)=T(gh)s(gh)=T(gh)h^{-1}s(g)\ .
$$
Since the section is arbitrary, it follows that we must have $T(gh)= h^{-1}T(g)h$.  In particular, 
\begin{equation} \label{eqn:M_iso}
M(h)=gT(g)g^{-1}\ ,\ h=gg^{\ast}\ ,
\end{equation}
 is a well defined traceless $2\times 2$-matrix valued function on $\HBbb^3$, and so this defines the map above. This is an isometry, since the hermitian structures are given by:
$$
\langle T_1, T_2\rangle=\frac{1}{2}\tr(T_1 T_2^{\ast})= \frac{1}{2}\tr(g^{-1}M_1g(g^{-1}M_2 g)^{\ast})= \frac{1}{2}\tr(M_1 hM_2^{\ast} h^{-1})=\langle M_1,M_2\rangle_h\ .
$$
The induced connection on the endomorphism bundle $\End_0 V$ is: 
$$\widehat\nabla T=dT+[g^{-1}dg, T]\ .$$
  On the other hand,
$$
\nabla^\CBbb M = d(gTg^{-1}) = g dT g^{-1} + g[g^{-1}dg, T]g^{-1} = g(\widehat\nabla T)g^{-1}\ ,
$$
which via \eqref{eqn:M_iso} proves that the connections are intertwined. 
\end{proof}

The main result of this subsection is now the following.

\begin{prop} \label{prop:main-connection}
Recall that $\nabla^\CBbb$ denotes the trivial connection on $T\HBbb^3\otimes \CBbb$. With respect to its hermitian structure, $\nabla^\CBbb=d_A+\Psi$, where $\Psi(h)=\frac{1}{2}[dhh^{-1}, \cdot]$ is a hermitian endomorphism valued $1$-form and $d_A$ is unitary. The real structure $\tau$ is flat with respect to $d_A$, and $d_A$ induces the Levi-Civita connection $\nabla^{\LC}$ on the fixed point set of $\tau$ which is isomorphic to $T\HBbb^3$.
\end{prop}

\begin{proof}
We calculate the hermitian part of the connection $\Psi$. From \eqref{eqn:M-metric},
\begin{align*}
d\langle M_1, M_2\rangle_h&=\langle dM_1, M_2\rangle_h+\langle M_1, dM_2\rangle_h+\langle [M_1, dhh^{-1}], M_2\rangle_h \\
0&=2\langle \Psi_h M_1, M_2\rangle_h+\langle [M_1, dhh^{-1}], M_2\rangle_h 
\end{align*}
which implies $\Psi$ has the form in the statement above. Hence,
\begin{equation} \label{eqn:dA}
d_A=d-\frac{1}{2}[dhh^{-1}, \cdot]
\end{equation}
Next, from \eqref{eqn:tau}
\begin{align*}
(d_A\tau)(M)&:= d_A(\tau(M))-\tau(d_AM) \\
&=
d(hM^{\ast} h^{-1})-\frac{1}{2}[dhh^{-1}, hM^{\ast} h^{-1}] -h\Bigl( dM-\frac{1}{2}[dhh^{-1},M]\Bigr)^{\ast} h^{-1} \\
&=
hdM^{\ast} h^{-1}+[dhh^{-1}, hM^{\ast} h^{-1}]-\frac{1}{2}[dhh^{-1}, hM^{\ast} h^{-1}]  \\
&\qquad\qquad
-hdM^{\ast} h^{-1}-\frac{1}{2}h[h^{-1}dh, M^{\ast}] h^{-1} \\
&=0\ .
\end{align*}
Hence, $\tau$ is flat with respect to $d_A$, and so $d_A$ induces an $\SO(3)$ connection on $T\HBbb^3$. Comparing \eqref{eqn:dA} with \eqref{eqn:levi-civita}, we see that this is the Levi-Civita connection.
\end{proof}

\subsubsection{Flat connections and equivariant maps} \label{sec:connections-maps}

Recall from the previous paragraph the definition of the flat bundle $V\to \HBbb^3$. Sections of the dual bundle $V^{\ast}$ are functions
$s^{\ast} : \SL(2,\CBbb)\to (\CBbb^2)^{\ast}$ satisfying the condition $s^{\ast}(gh)=s^{\ast}(g)\circ h$ for all $h\in \SU(2)$. Moreover, the flat connection on $V$ induces one on $V^{\ast}$, which we denote with the same notation $\widehat \nabla$. In terms of this description of sections, 
$\widehat\nabla s^{\ast} = ds^{\ast}-s^{\ast}\circ g^{-1}dg$. Fix a unitary frame $\{v_1, v_2\}$ for $\CBbb^2$, and let $\{v_1^{\ast}, v_2^{\ast}\}$ be the dual frame. We express the matrix elements of $g\in \SL(2,\CBbb)$ as $g_{ij}$.

\begin{prop} \label{lem:global-flat}
The functions $s^{\ast}_i(g)=\sum_{j=1}^2 g_{ij}v_j^{\ast}$ give global parallel sections of $V^{\ast}$. Moreover, 
$s^{\ast}_i(g_1g_2)= (g_1)_{ij}s^{\ast}_j(g_2)$.
\end{prop}

\begin{proof}
We have $ds^{\ast}_i(v_k)=dg_{ik}$. Similarly,
$$
s_i^{\ast}\circ g^{-1}dg(v_k)= s_i^{\ast}\circ (g^{-1})_{jm}dg_{mk} v_j= g_{ij}(g^{-1})_{jm}dg_{mk}=dg_{ik}\ .
$$
The second statement is clear.
\end{proof}

We now present the general construction. Let $E\to X$ be a hermitian vector bundle, and let $\nabla$ be a flat  $\SL(2,\CBbb)$ connection on $E$.
The pullback of $E$ and $\nabla$ to the universal cover $\widetilde X\to X$ will be denoted with the same notation.
Choose a base point $p_0\in X$, and a lift $\tilde p_0\in \widetilde X$.
 Fix a unitary frame $\{e_1, e_2\}$ of the fiber $E_{p_0}$ (and therefore also $E_{\tilde p_0}$).  We have a uniquely determined global frame $\{\widetilde e_1, \widetilde e_2\}$ for $E\to \widetilde X$ that is parallel with respect to $\nabla$ and
  which agrees with $\{e_1, e_2\}$ at $\tilde p_0$.  Let $u_{ij}$ be the hermitian matrix $u_{ij}(p)=\langle \widetilde e_i, \widetilde e_j\rangle(p)$. Then $u_{ij}$ is hermitian and positive.  We claim that $\det u=1$.  
Indeed, write $\nabla=d_A+\Psi$, where $d_A$ is a unitary connection on $E$ and $\Psi$ a $1$-form with values in $\sqrt{-1}\g_E$.
Let $\{\hat e_1, \hat e_2\}$ be a unitary frame at $p$, $\widetilde e_i(p)=g_{ij}\hat e_j$, $\Psi(p) \hat e_i =\Psi_{ij}\hat e_j$. Then
at the point $p$, 
\begin{equation} \label{eqn:computation-du}
	du_{ij}=\langle d_A\widetilde e_i, \widetilde e_j\rangle+\langle \widetilde e_i, d_A\widetilde e_j\rangle=-2\langle \Psi(p)\widetilde e_i, \widetilde e_j\rangle=-2(g\Psi g^{\ast})_{ij}
\end{equation}
At the same time, $u_{ij}(p)=\langle\widetilde e_i, \widetilde e_j\rangle(p)=(gg^{\ast})_{ij}$. Hence,
$$
d\log\det u=\tr(u^{-1}du)=-2\tr\left( (gg^{\ast})^{-1}(g\Psi g^{\ast})\right)=-2\tr\Psi=0\ ,
$$
since $\Psi_{ij}$ is traceless. Therefore, $\det u(p)=\det u(\widetilde p_0)=1$ for all $p\in \widetilde X$. Hence, $u(p)\in \Dscr$, and we have therefore defined a map 
$u:\widetilde X\to \HBbb^3$ which
 sends the point $\widetilde p_0$ to the base point of $\Dscr$.  We also note for future reference that from \eqref{eqn:computation-du},
\begin{equation} \label{eqn:u_iso}
du\, u^{-1}=-2g\Psi g^{-1}\ .
\end{equation}

Let $\rho:\pi_1\to \SL(2,\CBbb)$ be the holonomy representation of $\nabla$ with respect to the frame $\{e_1,e_2\}$. Via the choice of base point $\tilde p_0$ we may view $\pi_1$ as acting on $\widetilde X$ by deck transformations. By definition, if $\rho(\gamma)=(g_{ij})$, then
$\tilde e_i(\gamma  p)=g_{ij}\tilde e_j(p)$ for any $p\in \HBbb^2$. Therefore,
$$
u_{ij}(\gamma p)=\langle g_{ik}\tilde e_k(p), g_{jm}\tilde e_m(p)\rangle=g_{ik}u_{km}(p)g^{\ast}_{mj},
$$
or $u(\gamma p)=\rho(\gamma)u(p)(\rho(\gamma))^{\ast}$. Thus, $u$ is equivariant with respect to the action of the holonomy representation $\rho$ on $\HBbb^3$.

\begin{prop} \label{lem:pull-back}
There is a $\pi_1$-equivariant isometry $u^{\ast} V^{\ast}\isorightarrow E$ that intertwines the flat connections $u^\ast\widehat\nabla$ and $\nabla$.
\end{prop}

\begin{proof}
Recall the sections of $V^{\ast}$ from Lemma \ref{lem:global-flat}.  Then the bundle isomorphism is defined by identifying 
$[g, s^{\ast}_i(g)]\mapsto \widetilde e_i(p)$, where $u(p)=gg^{\ast}$. By the second statement of Lemma \ref{lem:global-flat}, this identification is equivariant with respect to the action of $\pi_1$. Since the identification is between flat sections, the connections are manifestly intertwined. It remains to check that this is an isometry.  But
$$
\langle s^{\ast}_i, s^{\ast}_j\rangle(u(p))= g_{ik}\overline{g_{jm}}\langle v_k^{\ast}, v_m^{\ast}\rangle=  g_{ik}\overline{g_{jk}} = u_{ij}(p) =\langle \widetilde e_i, \widetilde e_j\rangle(p)\ .
$$
This completes the proof.
\end{proof}

The next proposition is the main consequence of the discussion above.

\begin{prop} \label{prop:levi-civita-main}
Let $E\to X$ be a hermitian rank $2$ vector bundle with a flat $\SL(2,\CBbb)$ connection $\nabla$ and holonomy representation $\rho\colon \pi_1\to \SL(2,\CBbb)$.
Write $\nabla=d_A+\Psi$, where $d_A$ is a unitary connection on $E$ and $\Psi$ is a $1$-form with values in $\sqrt{-1}\g_E$.
Let $u\colon\widetilde X\to \HBbb^3$
 be the $\rho$-equivariant map described above. Then $\sqrt{-1}\g_E$ may be isometrically identified with $u^{\ast} T\HBbb^3$. Under this identification the induced connection $d_A$ corresponds to the pullback of the Levi-Civita connection on $\HBbb^3$, and the $1$-form $-2 \Psi$ corresponds to the differential $du$ of the map $u$.
\end{prop}

\begin{proof}
By Proposition \ref{lem:pull-back}, the connection on $E$ pulls back from the one on $V^{\ast}$. The induced connection on $\End_0 E$ is therefore the pullback of the one on $\End_0 V$. Since these bundles are isometric, the subbundle $\sqrt{-1}\g_E$ identifies with the bundle of traceless hermitian endomorphisms of $V$. By Propositions \ref{lem:V-M} and \ref{prop:main-connection}, the latter is isometric to $T\HBbb^3$, and the induced connection is Levi-Civita.
 The computation in \eqref{eqn:u_iso} shows $du\, u^{-1}=-2g\Psi g^{-1}$. 
 Combined with the identification \eqref{eqn:M_iso}  this yields the claimed relation between $\Psi$ and the differential of the map $u$ in the traceless model. 
 \end{proof}

\begin{proof}[Proof of Theorem \ref{thm:DonaldsonIsomorphism}]
 The proof of (i-iii)  follow from Proposition \ref{prop:levi-civita-main} . 
For (iv), use \eqref{eqn:u_iso} and definition of the metric \eqref{eqn:metric-on-H3} to compute:
$$
{\rm Hopf}(u)=\frac{1}{2}\tr(duu^{-1}\otimes duu^{-1})^{2,0}=2\tr(\Psi\otimes\Psi)^{2,0}.
$$
This completes the proof of the Proposition.
\end{proof}

\begin{remlabel} \upshape
The  construction above is natural with respect to the action of unitary gauge transformations on pairs  $(A,\Psi)$. Namely, modifying $(A,\Psi)$ by $g^{\ast}(A,\Psi)$ where $g\in \mathcal G$ is a unitary gauge transformation results in conjugating the representation $\rho$ and the map $u$ by some element in $\SU(2)$.
\end{remlabel}

\subsubsection{The self-duality equations and harmonic maps}

Up to this point, the choice of hermitian metric on the bundle $E$ was arbitrary and not related to the holonomy representation $\rho$ determined by the flat $\SL(2,\CBbb)$ connection $\nabla$. For this reason, the pair $(A,\Psi)$ resulting from the decomposition of $\nabla$ into its unitary and hermitian part as in Proposition \ref{prop:levi-civita-main} will in general not satisfy any equation apart from the  flatness of  $\nabla$, which is equivalent to the first two equations
of \eqref{eq:sde}.
 Likewise, the construction of the $\rho$-equivariant map $u$ depends on the hermitian metric on $E$ and hence this map will in general not  enjoy any special properties. 
 The link to the extra structure is provided by the following.

\begin{prop}[\cite{Donaldson:87}] \label{prop:donaldson}
	Let $E\to X$ be a rank $2$ vector bundle with hermitian metric $h$ and a flat $\SL(2,\CBbb)$ connection $\nabla$ and corresponding holonomy representation $\map{\rho}{\pi_1}{\SL(2,\CBbb)}$. Denote by $
\nabla=d_A+\Psi
$,
the unique decomposition of $\nabla$ into a unitary connection $d_A$ on $(E,h)$ and a $1$-form $\Psi$ with values in $\sqrt{-1}\g_E$. Let moreover
$
\map{u}{\Xtilde}{\HBbb^3}
$
be a $\rho$-equivariant smooth map as in Proposition
 \ref{prop:levi-civita-main}. Then the pair $(A,\Psi)$  satisfies the self-duality equations \eqref{eq:sde}  if and only if the map $u$ is harmonic.
\end{prop}

\begin{remlabel} \upshape
A hermitian metric $h$ on the bundle $E$ such that the corresponding $\rho$-equivariant map $u$ is harmonic is called a \textbf{harmonic metric}. In \cite{Donaldson:87} it is shown that  a harmonic metric exists whenever the representation $\rho:\pi_1\to \SL(2,\CBbb)$ is completely reducible. 
It is unique if $\rho$ is irreducible. In this case, the resulting solution $(A,\Psi)$ of the self-duality equation is also irreducible.	 
In this paper, we consider monodromies associated to pleated surfaces, and the representations are therefore automatically irreducible (\emph{cf}.\ \cite[p.\ 36]{Bonahon:shearing}).
\end{remlabel}

 \subsection{Laminations} \label{sec:laminations} In this section, we briefly review some of the topological objects that will be used in our description of the images of high energy harmonic maps.

 \subsubsection{Measured foliations and laminations} \label{sec:measured-foliations}
 A {\bf measured foliation} on a surface $\Sigma$ is a partial foliation $\Fcal$ of the surface with a finite number of $k$-pronged singularities, equipped with a measure on transverse arcs.  The  examples we consider in this paper are the {\bf horizontal and vertical foliations} of a holomorphic quadratic differential $q$ with simple zeroes, $q\in \QD^\ast(X)$, which we denote by $\Fcal_q^h$ and $\Fcal_q^v$, respectively. 
In the notation of \textsection{\ref{sec:spectralcurve}}, 
these can be defined as follows. At each point of the spectral curve $\widehat X_q^\times$, consider a (real) unit tangent vector $\hat u$
with $\imag\left(\lambda_{\SW}(\hat u)\right)=0$. Then the flow lines of $\hat u$ integrate locally to give a foliation independent of the choice of sign of $\hat u$ and invariant under the involution $\sigma$. It therefore projects to a foliation on $X^\times$, and this is $\Fcal_q^h$,
the horizontal foliation of $q$.
The  vertical foliation of $q$, $\Fcal_q^v$, is transverse to $\Fcal_q^h$, and is defined similarly using the real part.  We denote the lifts of the foliations to the universal cover $\widetilde X$ by $\widetilde \Fcal_q^h$ and $\widetilde\Fcal_q^v$.

A {\bf critical leaf} of  $\Fcal_q^h$ is a segment of a horizontal leaf  terminating at a zero of $q$. A {\bf saddle connection} of the horizontal (\emph{resp}.\ vertical) foliation is a horizontal (\emph{resp}.\ vertical) leaf joining two zeroes. Following \cite[\textsection{3}]{Levitt83}, when we  refer to a path in $\widetilde \Fcal_q^h$ as a  \emph{horizontal leaf}, we implicitly mean that it either contains no zeroes of $\widetilde q$, or when it meets critical points it either turns consistently to the right or to the left with respect to the cyclic ordering on the critical leaves terminating at a give zero. Saddle connections will play an important technical role in this paper, but there is a distinction between vertical and horizontal saddles, as discussed in the Introduction.

The foliations $\Fcal_q^h$ and $\Fcal_q^v$ come equipped with transverse measures. 
If $k$ is a $C^1$ arc transverse to $\Fcal_q^h$, then we can lift $k$  to a parametrized arc $\hat k$ in $\widehat X_q$ in such a way that\footnote{The negativity is dictated in order to agree with Bonahon's convention; see \cite[\textsection{2}]{Bonahon:shearing} and \textsection{\ref{subsec:BendingCocycles}} below.} 
$\imag(\lambda_{\SW}(\dt{\hat k}))< 0$ at all points of $\hat k$. The measure of $k$ is then the integral of $-\imag\lambda_{\SW}$ along $\hat k$.
We will say that a piecewise $C^1$ arc $k$ is {\bf quasitransverse
to $\Fcal_q^h$} if 
it is a finite union of $C^1$  arcs in $X\setminus Z(q)$, and if it admits a piecewise $C^1$ lift $\hat k$ in $\widehat X_q$ in such a way that
$\imag(\lambda_{\SW}(\dt{\hat k}))\leq 0$ at all points of $\hat k$.
 The definition of a path quasitransverse to $\Fcal_q^v$ is defined similarly using the real part. 
 
 A {\bf measured geodesic lamination} $\Lambda$ on a hyperbolic surface $S$ is a partial foliation of the surface by simple (not necessarily closed) geodesics, together with a measure on transverse arcs. 
 A measured foliation may be ``straightened'' to a measured lamination. For example, given $\Fcal$,  each bi-infinite leaf of $\widetilde \Fcal\subset\widetilde S\simeq \HBbb^2$ defines a unique pair of distinct points in the circle at infinity, and hence a unique geodesic in $\HBbb^2$. The collection of geodesics thus obtained are noninterlacing and form a closed set, and so define a lamination $\widetilde \Lambda$ of $\HBbb^2$.  The construction is equivariant with respect to the action of the fundamental group, and so there is a well defined quotient $\Lambda\subset S$.
 The transverse measure on $\Fcal_q^h$  may then be transported to a measure on arcs transverse to $\Lambda_q^h$. For more details on this construction see \cite{Levitt83}.  We will denote the measured laminations associated to $\Fcal_q^h$ and $\Fcal_q^v$ by $\Lambda_q^h$ and $\Lambda_q^v$, respectively.

 The {\bf Hubbard-Masur theorem} \cite{HM} gives a converse to this construction.  Given a measured foliation $\Fcal$ (\emph{resp}.\ measured lamination $\Lambda$) there is a unique nonzero $q\in \QD(X)$ such that $\Fcal$ is measure equivalent to $\Fcal_q^h$ (\emph{resp}. $\Lambda$ to $\Lambda_q^h$). We shall denote this differential by $\phi_{\HF}(\Fcal)$ (\emph{resp}.\ $\phi_{\HF}(\Lambda)$).  (See \cite{Wolf:96} for a proof closer to the perspective in this paper.)

For a lamination $\Lambda\subset S$, the 
 components of $\HBbb^2\setminus\widetilde \Lambda$ are called {\bf plaques}, and we denote the set of such by $\Pcal(\Lambda)$. When all the plaques are ideal triangles, we say that $\Lambda$ is {\bf maximal}. 
If $\Fcal_q^h$ has saddle connections, then $\Lambda_q^h$ will not be a maximal lamination, and we describe this in more detail in \textsection{\ref{sec:maximalization}}. For a distinct pair $P,Q\in \Pcal(\Lambda)$, we say that $R\in \Pcal(\Lambda)$ {\bf separates} $P$ and $Q$ if any path from $P$ to $Q$ in $\HBbb^2$ intersects $R$.

We end this section with two clarifying remarks.
  First, while a simple example of a measured lamination is a multicurve equipped with atomic transverse measures, a more typical example (obtained as a limit of multicurve examples) will meet any transverse arc in a Cantor set. Second, while geodesic laminations appear to depend on the hyperbolic structure of the surface, using the idea of straightening curves, a geodesic lamination $\Lambda$ in any marked hyperbolic structure $S$ on $\Sigma$ induces a unique geodesic lamination in any other marked hyperbolic surface $S'$ on $\Sigma$. See \cite[p.\ 7]{Bonahon:shearing}. We will often denote these $\Lambda$ without reference to the hyperbolic structure.

 \subsubsection{Train tracks} \label{sec:train-tracks}
 An ingenious construction of Thurston provides for a way to organize nearby measured foliations/laminations as data on a geometric object.  A {\bf train track} on a surface $\Sigma$ is an embedded finite complex $\tau$ of $C^1$-arcs (called {\bf branches}) on $\Sigma$ meeting at vertices (called {\bf switches}) with a well defined common tangency. We can and will assume the switches are always trivalent. Then  one branch at a switch is {\bf incoming} and two are {\bf outgoing} (see \cite[p.\ 11]{PenHar}).
 Let $G=\RBbb$ or $S^1\simeq \RBbb/2\pi\ZBbb$.
 A {\bf weight} on a train track track $\tau$ is an assignment of an element of $G$ to each branch that obeys the {\bf switch conditions}: the weight  on the incoming branch equals the sum of the weights on the outgoing branches.
  We denote by $\Hcal(\tau, G)$ the set of $G$-weights on $\tau$. 

 One way to construct a train track is to consider a small $\epsilon$ neighborhood of a measured geodesic lamination, foliate that neighborhood by leaves transverse to the lamination, and then collapse the neighborhood to the leaf space of the foliation. If the resulting branches are weighted by the measure of arcs that cross the neighborhood, a measured train track that {\bf carries} the lamination results (\emph{cf}.\ \cite[p.\ 73]{PenHar}).

 A useful operation on train tracks is the  {\bf right and left splitting} (see \cite[p.\ 119]{PenHar}).  As one chooses an increasingly small parameter $\epsilon$ in the construction above the train tracks obtained are related by splitting. 
 Let us define splitting carefully.  Recall that a branch between two switches is called {\bf long} if it is  incoming at both ends (see \cite[p.\ 118]{PenHar}). The orientation of $\Sigma$ then orders the outgoing branches at the switches on each end of a long branch, and we label them left (L) and right (R) accordingly. A right splitting is then obtained by modifying the train track locally by replacing the long branch with two branches joining left and right at each switch, and then adding a third branch between them at whose switches the branches labeled L are incoming. The left splitting adds a branch so that the right branches  are incoming. See Figure \ref{fig:splitting}.

\begin{figure}
\begin{tikzpicture}
\draw [thick, blue] (0.5,1.1) -- (0.5,3.1);
\draw (0.5,1.1) to [out=270, in=0] (0,0.5);
\draw (0.5,1.5) to [out=270,in=180] (1,0.5);
\draw (0.5,3) to [out=90, in=0] (0,3.5);
\draw (0.5, 3) to [out=90, in=180] (1,3.5);
\node at (0.5,0) {\small long branch};
\node at (-.25,0.5) {\tiny L};
\node at (1.25,0.5) {\tiny R};
\node at (-.25,3.5) {\tiny R};
\node at (1.25,3.5) {\tiny L};
\draw (3,1) -- (3,3);
\draw (3,1) to [out=270, in=0] (2.5,0.5);
\draw (3,3) to [out=90, in=0] (2.5,3.5);
\draw (4,1) -- (4,3);
\draw (4,1) to [out=270, in=180] (4.5,.5);
\draw (4,3) to [out=90, in=180] (4.5,3.5);
\draw (3,1.25) to [out=90,in=270] (4,2.75);
\node at (3.5,0) {\small right splitting};
\draw (6,1) -- (6,3);
\draw (6,1) to [out=270, in=0] (5.5,0.5);
\draw (6,3) to [out=90, in=0] (5.5,3.5);
\draw (7,1) -- (7,3);
\draw (7,1) to [out=270, in=180] (7.5,.5);
\draw (7,3) to [out=90, in=180] (7.5,3.5);
\draw (6,2.75) to [out=270,in=90] (7,1.25);
\node at (6.625,0) {\small left splitting};
\end{tikzpicture}
\caption{Splitting of train tracks.}
 \label{fig:splitting}
\end{figure}
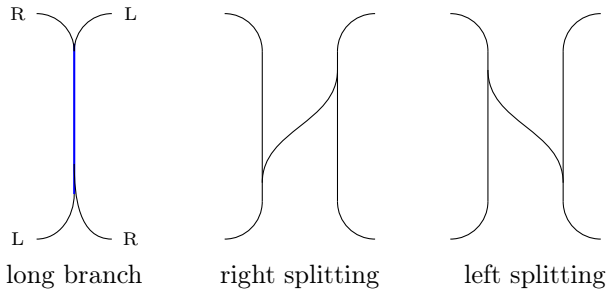

\subsubsection{Maximalizations} \label{sec:maximalization}
As mentioned above, in the case of horizontal saddle connections
the lamination $\Lambda_q^h$ is not maximal.
A maximal lamination can be obtained by adding finitely many leaves to $\Lambda_q^h$ \cite[p.\ 76]{CEG}.
Here we describe this mechanism precisely in terms of the foliation $\Fcal_q^h$. 
Consider a connected configuration $\Scal\subset \Fcal_q^h$ of saddle connections (along with their external critical leaves). We can make a train track $\tau_{\Scal}$ out of $\Scal$ by replacing each zero with a triangle, each of whose vertices is outgoing. 

\begin{definition} \label{def:maximalization}\upshape
A {\bf maximalization of $\Scal$} is a choice of left or right splitting $\widehat \tau_\Scal$ of each branch in $\tau_\Scal$ corresponding to a saddle connection, in such a way that the resulting train track  $\widehat \tau_\Scal$ contains no long branches (see Figure \ref{fig:max}). A {\bf maximalization of $\Fcal_q^h$} is a choice of maximalization of every maximal connected configuration of saddle connections. 
\end{definition}

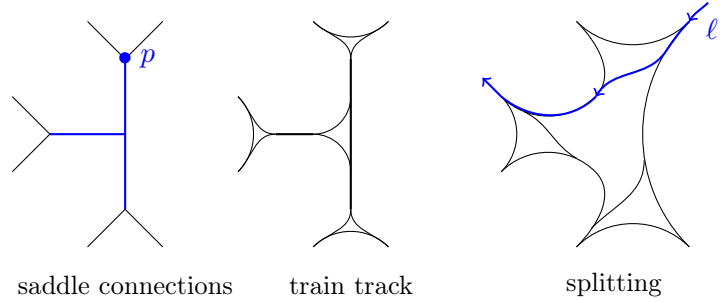
\begin{figure} 
\begin{tikzpicture}
\draw [thick, blue] (1.5,1) -- (1.5,3);
\draw [thick, blue] (0.5,2) -- (1.5,2) ;
\draw (1,3.5) -- (1.5,3) ;
\draw (2,3.5) -- (1.5,3) ;
\draw (1,0.5) -- (1.5,1) ;
\draw (2,0.5) -- (1.5,1) ;
\draw (0,2.5) -- (0.5,2) ;
\draw (0,1.5) -- (0.5,2) ;
\node at (1.5,3) {$\color{blue}\bullet$};
\node at (1.8, 3) {$\color{blue}p$};
\node at (1.5,0) {\small saddle connections};
\draw [thick] (4.5,1) -- (4.5,3);
\draw  [thick] (3.5,2) -- (4,2) ;
\draw (4.5,3) to [out=90, in=315] (4,3.5);
\draw (4.5,3) to [out=90, in=225] (5,3.5);
\draw (4,3.5) to  [out=315,in=225] (5,3.5); 
\draw (4.5,1) to [out=270, in=45] (4,0.5);
\draw (4.5,1) to [out=270, in=135] (5,0.5);
\draw (4,0.5) to [out=45, in=135] (5,0.5);
\draw (4,2) to  [out=0, in=270] (4.5,2.5);
\draw (4,2) to  [out=0, in=90] (4.5,1.5);
\draw (3,2.5) to [out=315,   in=180] (3.5,2);
\draw (3,1.5) to [out=45,   in=180] (3.5,2);
\draw (3,1.5) to [out=45, in=315] (3,2.5);
\node at (4.5,0) {\small train track};
\draw (7.5,3.5) to [out=315, in=225] (9,3.5);
\draw (7.5,0.5) to [out=45, in=135] (9,0.5);
\draw (9,0.5) to [out=135, in=225] (9,3.5);
\draw (7.5,3.5) to [out=315, in=45] (7.75,2.5);
\draw (7.5,0.5) to [out=45, in=315] (7.75, 1.5);
\draw (6.5,1.5) to [out=45, in=135] (7.75,1.5);
\draw (6.5, 2.5) to [out=315, in=225] (7.75, 2.5);
\draw (6.5,2.5) to [out=315, in=45] (6.5,1.5);
\draw (7.7,.75) to [out=45, in=270] (8.4,1.65) ;
\draw (6.8,2.3) to [out=335, in=150] (7.5,1.68) ;
\draw [thick, blue, ->] (8.68,3.09) to [out=235, in=60] (7.76,2.5) ; 
\draw [thick, blue] (6.5,2.5) to [out=320, in=225] (7.76,2.5);
\draw [thick, blue, <-] (9,3.5) to [out=60,in=225] (9.25,3.75);
\draw [thick, blue] (9,3.5) to  [out=225, in=55]  (8.68,3.09); 
\draw [thick, blue, <-] (6.25,2.75) to [out=315,in=135] (6.5,2.5);
\node at (9.3, 3.4) {\color{blue}$\ell$};
\node at (8,0) {\small splitting};
\end{tikzpicture}
\caption{Maximalization.}
\label{fig:max}
\end{figure}
Note that maximalizations always exist: for example, one may choose right splittings for all the saddle connections.
The terminology is justified by the following. 

\begin{lem} \label{lem:maximalization}
A maximalization of $\Fcal_q^h$ uniquely determines a maximal lamination $\Lambda$ containing $\Lambda_q^h$ as a sublamination. 
\end{lem}

\begin{proof}
 Let $\Scal$ be a maximal connected component of saddle connections,
 and let $c$ be a saddle connection. There are two cases: (1) $c$ is part of a closed loop $\gamma$ of saddle connections, (2) there is a zero
$p$ of $q$ at one end of $c$, one of whose
 critical leaves is in the complement of  $\Scal$
 (call this an \emph{external} zero). In case (1), 
$\Lambda$ contains a closed geodesic $\overline\gamma$ homotopic to $\gamma$. 
 The splitting now determines a train path  from the critical leaf of one end point of $c$ that is not part of $\gamma$ to $\overline \gamma$. This corresponds to a leaf of $\Lambda$ that spirals into $\overline \gamma$.  
In case (2),
  the splitting of the saddle connection $c$ selects one of the other critical leaves
 of $p$; namely, the one which is incoming with respect to the switches 
created in the splitting.  Denote this leaf $\ell\subset\Fcal_q^h$.
 By maximality of the component $\Scal$, 
the lift $\widetilde \ell$ of $\ell$ determines a geodesic
 half ray $g$ in $\HBbb^2$ that is asymptotic to
  a leaf of $\widetilde\Lambda_q^h$ on one end. 
Viewing $\ell$ as a train path in $\widehat\tau_\Scal$, there is a 
unique continuation  to a path (still denoted by $\ell$) 
that crosses the split saddle connection which ends at $p$, and
then either
exits through a branch of another external zero, or 
spirals around a closed branch homotopic to a closed loop of
saddle connections.
 Uniqueness follows because all
 further switches the path encounters
 are outgoing by assumption. 
Thus the lift of $\ell$ determines a
 bi-infinite geodesic that is asymptotic to different leaves of $\widetilde \Lambda_q^h$ on either side. By the condition that there are no long branches in $\widehat \tau_\Scal$, the geodesics added in this way are disjoint, and since the interior complementary regions of $\widehat \tau_\Scal$ are triangles, the resulting lamination is maximal.  
\end{proof}

\begin{remlabel} \label{rem:maximalization}  \upshape
Given a maximalization of $\Fcal_q^h$ as in Definition \ref{def:maximalization} with lamination $\Lambda$ as in Lemma \ref{lem:maximalization}:
\begin{compactenum}[(i)]
\item
The  leaves $\Lambda\setminus \Lambda_q^h$ may be represented by  paths that are quasi-transverse to $\Fcal_q^v$, consisting of horizontal leaves coming into and exiting a neighborhood of $\Scal$,  with a small vertical arc cutting one saddle connection of $\Fcal_q^h$ (we shall refer to these as {\bf additional leaves} of $\Fcal_q^h$ or $\Lambda$); 
\item  there is a 1-1 correspondence between zeroes $\widetilde Z(q)\subset\HBbb^2$ and the plaques of $\HBbb^2\setminus \Lambda$. 
\end{compactenum}
\end{remlabel}

The lift $\widehat \Lambda$ to $\widehat X_q$ can be oriented.
For convenience, we always choose this so that the oriented leaves of $\Lambda_q^h$ have $\real\Lambda_{\SW}>0$.
 This then gives an orientation to the homology classes in $H_1^{\odd}(\widehat X_q,\ZBbb)$ corresponding to the saddle connections. Suppose there is a saddle connection $c$ from $p$ to $q$. Then we can change $c$ to an arc consisting of a vertical leaf emanating from $p$, followed by a horizontal leaf shadowing the saddle connection, and then another vertical leaf terminating at $q$. Indeed,  there are two such ways of constructing such a path. However, the orientation of $\widetilde \Lambda$ chooses one of these: namely, the one whose lift intersects $\widetilde \Lambda$ positively.
Notice that these  paths are \emph{not} quasitransverse with respect  to $\Fcal_q^v$. 
We shall call such arcs {\bf modified saddle connections}.

\subsection{High energy harmonic maps} \label{subsec: high energy harmonic maps}

This paper focuses on asymptotics of the $\PSL(2,\C)$ character variety for $\pi_1$, especially as reflected in the associated classes of Higgs bundles.  The previous sections related these bundles to equivariant harmonic maps $u:\tilde{X} \to \H^3$, and it will turn out that one leaves all compacta in the character variety (and the associated moduli spaces of Higgs bundles) exactly when the energy of the associated harmonic maps grows without bound.
In this section, we collect some of the basic analytic estimates  on the geometry of harmonic maps whose energies are tending to infinity. These will be used throughout the paper.

\subsubsection{Minsky's results}\label{subsect:Minskysresults}

The following result due to Minsky plays a crucial role in the subsequent qualitative estimates involving high energy harmonic maps. It  will later also be needed in  \textsection \ref{subsect:realizinglaminations}.

Let $u_n:\widetilde X\to \HBbb^3$ be a sequence of $\rho_n$-equivariant harmonic maps with Hopf differentials $t_n^2q_n$, $q_n\to q$ in $\SQD^\ast(X)$. 
Recall that $Z(q_n)$ is the set of zeroes  of $q_n$, which we  assume to be simple. 
 For a parameter $s_n$, let $\Omega_{s_n}(p)$ be a hexagonal domain for each $p\in Z(q_n)$. The $s_n$ will be chosen so that these domains are disjoint for distinct zeroes of $q_n$. 
 Set 
\begin{equation} \label{eqn:Qn}
\Qcal_n=\bigcup_{p\in Z(q_n)} \Omega_{s_n}(p)\ .
\end{equation} 
 
We also assume that the boundary of each hexagon  $\Omega_{s_n}(p)$ is  formed from alternating horizontal and vertical edges. We let  $\widetilde Z(q_n)$ denote the preimage of the set $Z(q_n)$ under the projection map $\pi\colon \widetilde X\to X$.

\begin{prop}[\emph{cf}.\ {\cite[Thm.\ 4.2]{Minsky:HarmonicThreeManifold}}] \label{prop:realizing-3d}
There are  constants $A, c_0, C_0$, all independent of $n$, and $N$ such that the following hold. For $n\geq N$ and $s_n\leq c_0$, there is a $\rho_n$-equivariant map $\Pi^\ast$  from the leaves of $\widetilde\Fcal_{q_n}^h$ in the complement of $\widetilde\Qcal_n$  to a 
collection $\Lambda_n^{h,\ast}$ of geodesics in $\HBbb^3$ which factors through $u_n$. Moreover, for any $p\in \widetilde X\setminus\widetilde\Qcal_n$, 
$$
d_{\HBbb^3}(u_n(p), \Pi^\ast(p))\leq A\exp(-t_nC_0)\ ,
$$
and the derivative along the horizontal leaf through $p$ $($in the
 $|q_n|$ metric$)$ is
$$
\left| |d\Pi^\ast|-2\right| \leq A\exp(-t_nC_0)\ .
$$
\end{prop}

Proposition \ref{prop:realizing-3d} is proven in \cite{Minsky:HarmonicThreeManifold} in the context of harmonic maps from surfaces to complete hyperbolic $3$-manifolds, but the arguments apply equally well in the equivariant case. One important simplification in the situation here is that the domain Riemann surface $X$ is  fixed. As a consequence, the technical issues of ``thin flat cylinders'' that are dealt with in  \cite{Minsky:HarmonicThreeManifold} do not play a role here. In particular, the set $\Pscr_R$ in that reference may be taken to equal to $\Qcal_n$ defined in \eqref{eqn:Qn}.

The Proposition is a consequence of the following construction. For $s_n$ chosen sufficiently small and $n$ sufficiently large there is a train track $\tau_n\subset X\setminus \Qcal_n$ and $\varepsilon_n>0$,
$\varepsilon_n\to 0$ as $n\to +\infty$, satisfying the following.
\begin{compactenum}[(i)]
\item Let $\widetilde \tau_n\subset \widetilde X$ be the preimage of $\tau_n$, and set $\widetilde\tau_n^\ast=u_n(\tau_n)$. Then the branches of $\widetilde\tau_n^\ast$ have length $O(t_n)$ and geodesic curvature $O(\varepsilon_n)$.
\item The images by $u_n$ of the leaves of the horizontal foliation $\widetilde\Fcal_{q_n}^h$ in the complement $\widetilde X\setminus\widetilde \Qcal_n$ can be straightened to give a lamination $\widetilde\Lambda_n^{h,\ast}\subset\HBbb^3$. 
\item The  lamination $\widetilde\Lambda_n^{h,\ast}$ is $C^1_{\varepsilon_n}$-carried by $\widetilde \tau_n^\ast$.
\end{compactenum}

In the case where $\Fcal_{q_n}^h$ has saddle connections and we have chosen a  maximalization in the sense of Definition \ref{def:maximalization}, we can enlarge $\Lambda_n^{h,\ast}$ to a lamination $\Lambda_n^\ast$ as follows. By Remark \ref{rem:maximalization}, the maximalization gives rise to  finitely many quasi-transverse paths in $X$, which we may assume lie in the complement of $\Qcal_n$. (A technical point is that Minsky creates his track by extending components of $Q_n$ to \enquote{slice} through long rectangles of vertical trajectories: it is straightforward to check that this slicing can be done in a way to correspond to the maximalization discussed here.) The image by $u_n$ of the lifts of these can be straightened to geodesics that are asymptotic on one side to leaves in $\Lambda_n^{h,\ast}$. The map $\Pi^\ast$ in Proposition \ref{prop:realizing-3d} can be extended to a map on these leaves satisfying the same estimates. 
\medskip\\ 
We next choose coordinates, which we refer to as {\bf (canonical) $q_n$-coordinates} that are adapted to $q_n$ and hence to the map $u_n$.  To this end, note that, away from the zeroes of $q_n$, we may choose coordinates $z_n=x_n+iy_n$ in a patch so that, in those coordinates, the quadratic differential $q_n$ is expressed as $q_n= dz_n^2$. These are useful because the horizontal lines in this coordinates are both the leaves of the horizontal foliation of $q_n$, and also integrate the directions of the maximal stretch (eigendirection) of the tangent map $du_n$. Naturally, both the domain and the pull-back metric diagonalizes with respect to these coordinates. Following Minsky \cite[eq.\ (3.1)]{Minsky:HarmonicThreeManifold},     the pullback metric $u_n^*ds^2_{\H^3}$ with respect to the harmonic map $u_n$ as above  can be written  in terms of $q_n$-coordinates $(x_n,y_n)$ as 
\begin{equation} \label{eqn:pullback metric}
u_n^*ds^2_{\H^3} = 2t_n^2 (\cosh \sG_n +1) dx_n^2 + 2t_n^2 (\cosh \sG_n - 1) dy_n^2,	
\end{equation}
where $\sG_n = \sinh^{-1} (2\sJ_n)$ and $\sJ_n$ is the Jacobian determinant of the map $u_n$. The factor $t_n^2$ enters since the harmonic map $u_n$ has Hopf differential $t_n^2q_n$.

\begin{prop}\label{prop:pullbackmetric}
The pullback metric by $u_n$ in terms of canonical coordinates for $q_n$ satisfies
\begin{align}
u_s^*ds^2_{\H^3} = 4t_n^2 dx^2 + O\big(\exp(-2ct_n)\big) \label{eqn:small additive error to geodesic}
\end{align}
in $C^k$ for some constant $c>0$.
\end{prop}

\begin{proof}
As shown in  \cite[Lemma 3.4]{Minsky:HarmonicThreeManifold} there is a constant $B$ such  the quantity $\sG_n$ satisfies the pointwise estimate
\begin{equation*}  
 \sG_n(p)<\frac{B}{\cosh d}
\end{equation*}
for  every point $p$  at  $t_n^2q_n$-distance at least $d>0$ to the  zero set of $q_n$. Since we are here considering points  outside some fixed neighborhood of the zero set of $q_n$, this distance is bounded below by  $ct_n$ for some constant $c>0$. It follows that
\begin{equation*}
	\sG_n(p)<2Be^{-ct_n}\ ,
\end{equation*}
and consequently
\begin{equation*}
	\cosh\sG_n(p)<1+ 4B^2 e^{-2ct_n}\ .
\end{equation*}
Inserting this last estimate into \eqref{eqn:pullback metric} implies the claim.
\end{proof}
This last estimate says that, away from the zeroes of the Hopf differential, the horizontal trajectories have image under a high energy map $u_n$ that are stretched by the factor $t_n$, up to a small and rapidly decaying error; that the image of those trajectories have exponentially small geodesic curvature; and that the vertical trajectories have exponentially decaying length.

\subsubsection{High energy harmonic maps near the zeroes of $q$}\label{subsect:highenergyharmonicnearzeroes}

 We continue with the notation of the previous section.

\begin{prop}\label{prop: harmonic map C1 localizes near zeroes}
For every fixed $\varepsilon>0$ there exists a constant $N$ 
  such that the following holds. There is an ideal hyperbolic triangle $\Delta\subset\H^3$ such that  for every  $n\geq N$  the distance between the tangent plane $T_p(u_n(\widetilde X))\subset \H^3$ to $\Delta$ is less than $\varepsilon$, for every point $p\in u_n(\widetilde Z(q_n))$. 	
\end{prop}

An analogous statement for two dimensional targets is the main theorem of \cite{Wolf:HighEnergy}.  The present version is a reflection for harmonic maps of aspects of the approximate solutions constructions in \textsection{\ref{subsect:approxsolutions}}. 

\begin{proof}
	For each fixed $n$ and $p\in Z(q_n)$, consider a lift $\widetilde\Omega_{s_n}(p)$ of the hexagon $\Omega_{s_n}(p)\subset X$ to $\widetilde X$. Let $h_1$, $h_2$, $h_3$ denote the three horizontal edges of $\Omega_{s_n}(p)$, which we parametrize in an orientation-preserving way by a parameter $0\leq s\leq 1$. Proposition \ref{prop:realizing-3d} shows the existence of geodesics $c_i\colon [0,1]\to \H^3$ such that the distance between $u_n(h_i(s))$ and $c_i(s)$ is less than   $A\exp(-t_nC_0)$, for all $0\leq s\leq 1$. By Proposition \ref{prop:pullbackmetric} the length of each $c_i$ is of order $t_n$. Furthermore, the distance between  each consecutive  pair of endpoints $u_n(h_i(1))$ and $u_n(h_{i+1}(0))$ satisfies an exponentially small    bound. It follows from elementary hyperbolic geometry that there is an ideal hyperbolic triangle $\Delta\subset\H^3$  which is at distance at most  $\varepsilon$ to the lines $u_n(h_i)$. Since $u_n(\widetilde X)$ is contained in the convex hull of these lines, it follows that $\Delta$ and $u_n(\widetilde X)$ have at most distance $\varepsilon$, for all sufficiently large $n$. To see that also the tangent plane $T_p(u_n(\widetilde X))$ lies  $\varepsilon$-close to $\Delta$, we compare the harmonic map $u_n$ with the harmonic map $v_n$ which maps $\Omega_{s_n}(p)$ to $\H^3$ and has boundary values the edges of $\Delta$. Its image is contained in $\Delta$ and, since the boundary values of $u_n$ and $v_n$ differ by at most $\varepsilon$, it follows  by standard estimates on harmonic maps that both are $C^1$-close in the interior of $\Omega_{s_n}(p)$. This implies the assertion.
	\end{proof}

\subsection{Harmonic maps to $\RBbb$-trees} 

\subsubsection{Definitions} \label{sec:trees}
An  $\RBbb$-{\bf tree} is a complete length space $T$ such that any
two points can be joined by a unique path parametrized by arc
length.  This path is called the geodesic between the points, say
$p,q$, and it is denoted $\overline{pq}$.  We shall be interested in trees admitting isometric actions of $\pi_1$, and we will always assume the action is minimal in the sense that there is no proper $\pi_1$-invariant subset of $T$.
 In such a situation, we obtain a {\bf length function}
$$
\ell_T : \pi_1\lra \RBbb_{\geq 0} : [\gamma]\mapsto \inf_{p\in T}d_T(p,\gamma p)\ .
$$
Scaling the metric (and hence $\ell_T$) by positive constants defines a {\bf projective class} of length functions (\emph{cf}. \cite{Chiswell:76}). 

Examples of $\RBbb$-trees come from the following construction.
Let  $\Fcal$ be a measured foliation on $\Sigma$ with transverse measure $\mu$.  Define the {\bf dual tree} $T_{\Fcal}$ to the foliation as follows: if $\widetilde{\Fcal}$ is the lift to the universal cover, define a pseudodistance $\tilde d$ by
$$
\tilde d(p,q)=\inf\{ \tilde\nu(c) : c \ \text{a rectifiable path between}\ p, q\}\ .
$$
Then the Hausdorffication of $(\widetilde \Sigma, \tilde d)$ is an $\RBbb$-tree with an isometric action of $\pi_1$ (\emph{cf}.\ \cite[Cor.\ 2.6]{Bowditch:98}, and also \cite{MorganShalen:91,Otal:96}).
In the case of a nonzero holomorphic quadratic differential $q$ on a Riemann surface $X$,  we set $T_q:= T_{\Fcal_q^v}$.

 A {\bf morphism} of $\RBbb$-trees is a continuous map $f:T\to T'$ such that given any segment $e\subset T$, then either $f$ is constant on $e$ 
or $e$ decomposes into a finite union of subsegments $e_1\cup \cdots\cup e_k$ such that $f$ restricted to each $e_i$ is an isometry onto its image. It is a fact that in the latter case $f$ is either an isometry on $e$ or a {\bf folding}, meaning that it identifies two or more subsegments.  

Trees are examples of nonpositively curved metric spaces (NPC). Following ideas of Gromov \cite{GroSch}, Korevaar-Schoen \cite{KorSch1,KorSch2}, and independently Jost \cite{Jost:94}, developed a theory of energy minimizing maps from Riemannian domains  to NPC spaces. The last author studied the  case of maps to $\RBbb$-trees (see \cite{WolfT,Wolf:96}), which is the one relevant to this paper.
We  will need only very little from these  results, and we package a summary statement as follows (\emph{cf}.\ \cite{DaWen} for more details).

\begin{thm} \label{thm:KS}
Let $q$ be a nonzero holomorphic quadratic differential on a Riemann surface $X$.
Then the leaf space  projection map $u:\widetilde X\to T_q$ is an equivariant harmonic map. In general, let 
$T$ be an $\RBbb$-tree with an isometric action of $\pi_1$, and let $v:\widetilde X\to T$ be an equivariant harmonic map. Then
\begin{compactenum}[(i)]
\item the map $v$ is uniformly Lipschitz with constant
proportional to $E(u)^{1/2}$ $($the
 constant depends on the choice of conformal metric on $X$$)$;
\item the Hopf differential $\Hopf(v)=4q$
 is well defined and is a holomorphic quadratic differential that is nonzero unless $v$ is constant and the action is trivial;
\item  $v=p\circ u$, where $u: \widetilde X\to T_q$ is projection as above, and $p: T_q\to T$ is a folding.
\end{compactenum}
\end{thm}

We shall also need a version of the Korevaar-Schoen strong compactness theorem,
stated here in the limited context that we require. 
For positive constants $t_n\to +\infty$, let $\HBbb_n$ denote  the hyperbolic space $\HBbb^3$, but where the 
 metric has been rescaled:
$ds_{\HBbb_n}=t_n^{-1}ds_{\HBbb^3}$.  
For the following result, see also \cite[Thms.\ 2.2 and 3.1]{DDW:98}.

\begin{thm}[{\cite[Prop.\ 3.7 and Thm.\ 3.9]{KorSch2}}] \label{thm:KS-compactness}
Suppose $u_n:\widetilde X\to \HBbb_n$ is a sequence of $\rho_n$-equivariant continuous finite energy maps, and assume that $u_n$ have a uniform modulus of continuity: for each $z$ there is a monotone function $\omega(z,R)$ so that $\lim_{R\downarrow 0}\omega(z,R)=0$ and
$$
\max_{w\in B_R(z)}d(u_n(z), u_n(w))\leq \omega(z,R) \ .
$$
Then there is an $\RBbb$-tree $T$ with an isometric action of $\pi$ such that the convex hulls of the images of the $u_n$ converge in the Gromov-Hausdorff sense to $T$. Moreover,
\begin{compactenum}[(i)]
\item 
 the $u_n$ converge to a continuous finite energy map $u:\widetilde X\to T$ that is equivariant for this action;
\item 
 if $\displaystyle\lim_{k\to \infty} E(u_k)\neq 0$, then $u$ is nonconstant;
\item if the $u_n$ are equivariant harmonic maps then so is $u$;
and in this case, if $q_n$ $($\emph{resp}.\ $q$$)$ is the Hopf differential of $u_n$ $($\emph{resp}.\ $u$$)$, then $t_n^{-2}q_n\to q$. 
 \end{compactenum}
\end{thm}
We refer to the limiting tree $T$ as a {\bf Korevaar-Schoen limit}. 
Note that by Theorem \ref{thm:KS}, $T$ is a folding of $T_q$. 

\subsubsection{The Morgan-Shalen compactification} \label{sec:morgan-shalen}
There is a compactification of $R(\Sigma)$ that restricts on the Fricke space to Thurston's compactification of Teichm\"uller space.  The ideal points are given by projective classes of nontrivial isometric actions of $\pi_1$ on $\RBbb$-trees. 

Given $[\rho]\in R(\Sigma)$, define
$$
\ell_{\rho} : \pi_1\lra \RBbb_{\geq 0} : [\gamma]\mapsto \inf_{x\in \HBbb^3}d_{\HBbb^3}(x,\rho(\gamma) x)\ .
$$
\begin{thm}[\cite{MorganShalen:84}] \label{thm:ms}
Consider a sequence $[\rho_n]\in R(\Sigma)$. Then up to passing to subsequences, one of the following occurs:
\begin{compactenum}[(i)]
\item there is $[\rho]$ such that $[\rho_n]\to [\rho]\in R^0(\Sigma)$;
\item there is a minimal nontrivial action of $\pi_1$ by isometries on an $\RBbb$-tree $T$, and numbers  $\varepsilon_n\downarrow 0$, such that for all $\gamma\in \pi_1$,
$$
\lim_{n\to\infty}\varepsilon_n\ell_{\rho_n}(\gamma)=\ell_T(\gamma)\ .
$$
\end{compactenum}
\end{thm}

For the next result we refer to \cite[Thm.\ 3.2]{DDW:98}, and we note that in the proof of that result harmonicity is not used.

\begin{thm} \label{thm:DDW}
Suppose that there is a constant $C>0$ such that the rescalings $t_n$
in Theorem \ref{thm:KS-compactness} satify:
$$
C^{-1} E_n^{1/2}\leq t_n\leq C E_n^{1/2} \ ,
$$
where $E_n$ is the energy of the $\rho_n$-equivariant harmonic map. Then 
the length function  of the action of $\pi_1$ on the Korevaar-Schoen
limit appearing in Theorem \ref{thm:KS-compactness} is in the projective class of the Morgan-Shalen limit of the sequence $[\rho_n]$.
\end{thm}

\section{Bending} \label{sec:bending}

In this section we introduce a geometric notion of bending along $\rho$-equivariant   maps $u\colon\tilde X\to \H^3$,  and  of pairs $(A,\Psi)$. When $(A,\Psi)$ is a Higgs pair, the connection $\nabla=d_A+\Psi$ has monodromy $\rho$, and $u$ is the $\rho$-equivariant harmonic map from Theorem \ref{thm:DonaldsonIsomorphism}, then
we prove that these notions coincide asymptotically at high energy (see Theorem \ref{thm:asymptotic-bending}).

\subsection{Bending  of maps and connections}
\subsubsection{Bending of  maps} \label{sec:bending-maps}
We begin with a definition.
\begin{definition} \label{def:tent}\upshape
A {\bf tent} $T$ in $\HBbb^3$ is a pair of totally geodesic compatibly  oriented half planes meeting along a geodesic (see Figure \ref{fig:tent}). The geodesic is called the {\bf crease} and is denoted $\gamma_T$. By \enquote{compatibility of the orientations}, we will mean the induced orientation on $\gamma_T$ from the two half planes is opposite.
The dihedral angle of the two planes is called the {\bf angle of the tent} and is denoted $\beta_T$. By convention, if the half planes coincide with same orientation, then $\beta_T=0$; if equality holds with opposite orientation, then $\beta_T=\pi$. 
\end{definition}

We will use the following intrinsic way of measure the angle of a tent. 
A {\bf crossing} of a tent $T$ is a continuous path $c:[0,L]\to T\subset\HBbb^3$ satisfying the following conditions. 
\begin{compactenum}[(i)]
\item $c(0)$ and $c(L)$ lie in different components of $T\setminus \gamma_T$, say $T_-$ and $T_+$, respectively;
\item  there is $0<L_1<L$ such that $c$ restricted to the interval $[0,L_1]$ is a $C^1$ curve in $T_-$ meeting $\gamma_T$ at $c(L_1)$ transversely;
\item  there is $L_1\leq L_2<L$ such that $c$ restricted to the interval $[L_2,L]$ is a $C^1$ curve in $T_+$ meeting $\gamma_T$ at $c(L_2)$ transversely;
\item $c$ restricted to $[L_1,L_2]$ is a portion of $\gamma_T$.
\end{compactenum}
The orientation of the tent gives a choice of tangent $N(L_1)$ to $\gamma_T$ at the crease where a crossing intersects $c(L_1)$. More precisely, $N(L_1)$ is oriented to the left with respect to $c$. Let $N(t)$ be the parallel translate of $N(L_1)$ along $c$.
Let $n_0$ and $n_L$ denote the unit normals to $T_-$ and $T_+$, compatible with the orientations. Let $\widetilde n(t)$ denote the parallel translation of $n_0$ along $c$. Then $\widetilde n(L)$ and $n_L$ lie in the plane orthogonal to $N(L)$. This plane inherits an orientation from $N(L)$ and the orientation on $\HBbb^3$.  Then $\beta_T$ is the angle from $n_L$ to $\widetilde n(L)$ with respect to this orientation.

\begin{figure} 
\begin{tikzpicture}
\draw  [pattern=north east lines, pattern color=yellow!50] 
(1,-2) to [out=175, in=275] (-3,0) to [out=85, in=200] (-1,2) to (-1,2) -- (1,-2);
\draw  [pattern= horizontal lines, pattern color=gray!20] 
(-1,2) to [out=-5, in=95] (1.75,0.5) to [out=-90, in=60] (1,-2) to (1,-2) -- (-1,2);
\draw [thick] (-0.2,0.4) -- (0.3,0.9);
\draw [thick] (-0.2,0.4) -- (0.35,0.4);
\draw [thick] (0.2,0.4) to [out=90, in=-45] (0.1,0.7);
\node at (0.5, 0.6) {$\scriptstyle \beta_T$};
\end{tikzpicture}
\caption{A tent.}
\label{fig:tent}
\end{figure}
Let $u:\widetilde X\to \HBbb^3$ be a continuous $\rho$-equivariant  map, and fix $\widetilde p, \widetilde q\in \widetilde X$. We make  the following

\medskip\noindent{\bf Assumption 1.} 
  $u$ is smooth at both $\widetilde p$ and $\widetilde q$ and $du$ has maximal rank there.

\begin{definition} \label{def:bending-map}\upshape
The {\bf bending $\Theta_u(\widetilde p, \widetilde q)\in \RBbb/2\pi \ZBbb$ of $u$ from $\widetilde p$ to $\widetilde q$},  is
 defined as follows. Let $\DBbb_{u(\widetilde p)}$ and $\DBbb_{u(\widetilde q)}$ denote the oriented totally geodesic planes in $\HBbb^3$ tangent to the images of $du(p)$ and $du(q)$.
\begin{compactenum}[(i)]
\item If $\DBbb_{u(\widetilde p)}$ and $\DBbb_{u(\widetilde q)}$  meet along a geodesic $\gamma_T$, 
let $\beta_T$ be the dihedral angle of the tent constructed from the two half planes in $\DBbb_{u(\widetilde p)}\setminus\gamma_T$ and $\DBbb_{u(\widetilde q)}\setminus\gamma_T$ which contain $u(\widetilde p)$ and $u(\widetilde q)$, respectively, and where the orientation of the tent comes from the orientation on $\DBbb_{u(\widetilde p)}$.
Then set
$\Theta_u(\widetilde p, \widetilde q)$ equal to $\beta_T$ 
if 
the orientation of $\DBbb_{u(\widetilde q)}$ is compatible with the orientation of the tent in the sense of Definition \ref{def:tent}, and to $\pi+\beta_T$
 if the orientation is incompatible.
\item If $\DBbb_{u(\widetilde p)}$ and $\DBbb_{u(\widetilde q)}$  do not intersect, let $c$ be the geodesic between the planes, oriented at one endpoint to agree with the normal of $\DBbb_{u(\widetilde p)}$. If this orientation of $c$ agrees with normal to $\DBbb_{u(\widetilde q)}$ at the other end point, set 
$\Theta_u(\widetilde p, \widetilde q)=0$. If the orientation is opposite, set $\Theta_u(\widetilde p, \widetilde q)=\pi$.
\item If $\DBbb_{u(\widetilde p)}$ and $\DBbb_{u(\widetilde q)}$  coincide, set $\Theta_u(\widetilde p, \widetilde q)=0$ if they have the same orientation, and set $\Theta_u(\widetilde p, \widetilde q)=\pi$ if they have opposite orientations.
\end{compactenum}
\end{definition}
\noindent
We note that
by $\rho$-equivariance of $u$ we clearly have 
$$\Theta_u(g\widetilde p, g\widetilde q)
=\Theta_u(\widetilde p, \widetilde q)
$$
for all $\widetilde p, \widetilde q$ satisfying Assumption 1, and all $g\in \pi_1$. It is also clear that 
$$
\Theta_u(\widetilde p, \widetilde q)=\Theta_u(\widetilde q, \widetilde p)\ .
$$

\subsubsection{Bending of connections} \label{subsec:HorizontalLifts}

Let $d_A$ be a unitary connection on $E$, inducing a connection (also denoted by $d_A$)  on the bundle $\sqrt{-1}\g_E$
of traceless hermitian endomorphisms of $E$.
Fix a $1$-form $\Psi \in \Omega^{1}(X,\sqrt{-1}\g_E)$. 
We will suppose that the connection $\nabla = d_A+\Psi$ is flat with monodromy $\rho$. 
Let $\map{k}{[0,L]}{X}$ be  a piecewise  $C^1$ curve.

\medskip\noindent{\bf Assumption 2.} 
 The  linear map
  \begin{equation*}
  \Psi(k(\sigma))\colon T_{k(\sigma)}X\to \sqrt{-1}\g_{E,k(\sigma)}
  \end{equation*}
has maximal rank at $\sigma=0, L$.
\medskip

By analogy to the bending of maps  in the previous section, we let $N(0)$ and $N(L)$ be endomorphisms that are a positive multiple of $\Psi(J(\dt k(0))$ and $\Psi(J(\dt k(L))$.
We define the bending angle $\Theta_k(A, \Psi)$ of the pair $(A,\Psi)$ along $k$  using parallel translation with respect to $A$ in place of the Levi-Civita connection. Namely, consider any endomorphism field $V(\sigma)$ along $k$ that is a positive multiple of the endomorphism field $\sigma\mapsto \Psi(\dt k(\sigma))$.
If $k$ is $C^1$ on subintervals $[\sigma_{i-1}, \sigma_i]$, $i=1,\ldots, m$, 
let $\Pi_\sigma^{k,A}$ denote parallel transport in $\sqrt{-1}\gfrak_E$ along $k$ with respect to $A$. 
   Then for $\sigma\in[\sigma_{i-1}, \sigma_i]$ let 
\begin{equation} \label{eqn:TotalParallelTransport}
  \widetilde n(\sigma)
:=\Pi_{\sigma}^{k,A}\Pi_{\sigma_{i-1}}^{k,A}
\cdots \Pi_{\sigma_1}^{k,A}n(0)\ ,
\end{equation}
be the total parallel transport,
where $n(0)$ denotes the endomorphism $\sqrt{-1}[N(0),V(0)]$. Let $P(L)\subset \sqrt{-1}\g_{E,k(\sigma)}$ be the orthogonal complement to $N(L)$, 
and use $N(L)$ and  the orientation of $\sqrt{-1}\g_{E,k(L)}$  to give $P(L)$ an orientation. 
We now define the \textbf{bending of the pair $(A,\Psi)$} along the path $k$,
\[
\Theta_k(A, \Psi)\in \RBbb/2\pi\ZBbb\ ,
\]
to be the angle from $\sqrt{-1}[N(L),V(L)]$ to the orthogonal projection of the endomorphism $\widetilde n(L)$ to $P(L)$, when the latter is nonzero (otherwise bending is undefined).

It is immediate from this definition that the bending $\Theta_k(A,\Psi)$ is invariant under the action of unitary gauge transformations on $(A,\Psi)$. We may therefore write $\Theta_k([(A,\Psi)])$ for the bending of the gauge equivalence class of the pair $(A,\Psi)$.

\subsection{Asymptotic bending of Higgs pairs}\label{subsect:limbendingHiggs}

In this section we relate the total bending in connections to periods of Prym differentials on the spectral curve.

\subsubsection{Horizontal lifts for limiting connections}

Recall  the definition of the spectral curve $\pi: \widehat X_q\to X$ associated to $q\in \QD^{\ast}(X)$ in \eqref{eqn:spectral_curve}.
Our first goal here is to compute the parallel transport in the bundle $\sqrt{-1}\g_E$ of hermitian endomorphisms with respect to the (singular) flat connection $d_{A_{\infty}}$ from \eqref{eq:genlimitingconf}. The   connection  $ A_{\infty}$  induces a unitary connection on the pullback bundle $\pi^\ast E$, which we denote by $\hat A_{\infty}$. After pulling back to the spectral curve, the   calculation of the parallel transport can be carried out in terms of a suitably  chosen oriented frame 
which we define as follows (\emph{cf}.\ \eqref{eqn:W1}):
\begin{align}
\begin{split}\label{eqn:W}
W_1&=\left(\begin{matrix} 0& -i\lambda^{-1}\Vert\lambda\Vert \\ 
 i\lambda\Vert\lambda\Vert^{-1}&0\end{matrix}\right), \\
W_2&=\left(\begin{matrix} 0& \lambda^{-1}\Vert\lambda\Vert \\  \lambda\Vert\lambda\Vert^{-1}&0\end{matrix}\right), \\
W_3&=\left(\begin{matrix} -1&0 \\ 0&1\end{matrix}\right),
\end{split}
\end{align}
with commutation relations
\begin{equation} \label{eqn:commutation}
[W_i, W_j]=2i\, \text{sgn}(ijk)\, W_k \ ,
\end{equation}
 
From \eqref{eqn:pullback-Phi} we have
\begin{equation} \label{eqn:pullback-Psi}
\pi^{\ast}\Psi_{\infty} =\pi^{\ast}\Phi_{\infty}+\pi^{\ast}\Phi_{\infty}^{\ast}=2{\rm 
Re}(\lambda_{\SW})\otimes W_2 \ .
\end{equation}

\begin{prop}\label{lem:W-1} The following hold: 
\begin{compactenum}[(i)]
	\item The hermitian endomorphism $W_2$ lies in $\pi^{\ast} L^{\mathbb C}_{\Phi_{\infty}}$. 
	\item 
The collection $\{W_1, W_2, W_3\}$ gives an $\hat A_{\infty}^0$-parallel oriented orthonormal frame for the bundle $\sqrt{-1} \gfrak_{\widehat E}$.
\end{compactenum}
\end{prop}

\begin{proof}
	The proof    is a straightforward calculation. We only check that $d_{\hat A_{\infty}^0}W_1=0$. For this we use that one can locally express $\lambda$ as $q^{\frac{1}{2}}$ so that
	\begin{equation*}
	d\big(\lambda^{-1}\Vert\lambda\Vert\big)	=
	 d\big(q^{-1/4}\bar q^{1/4} \big)=\frac{1}{4}q^{-1/4}\bar q^{-3/4}\bar\partial \bar q - \frac{1}{4}q^{-5/4}\bar q^{1/4} \partial   q= \frac{1}{4}q^{-1/4}\bar q^{1/4}\big(\bar\partial\log \bar q-\partial\log q\big),
	\end{equation*}
	using that $\bar\partial q=0$.  On the other hand, recall from \eqref{eq:stdlimconn} that
	\begin{equation*}
		A_{\infty}^0 =A_0 + \frac{1}{2} \left(\Im\, \bar \partial \log \Vert q\Vert\right)\begin{pmatrix}
 -i & 0 \\ 0 & i
 \end{pmatrix}=A_0 +\frac{1}{8} \big(\bar\partial \log\bar q-  \partial \log  q\big)\begin{pmatrix}
 -1 & 0 \\ 0 & 1
 \end{pmatrix},
	\end{equation*}
where $A_0$ denotes the Chern connection. Now with $d_{\hat A_{\infty}^0}W_1=dW_1+[\hat A_{\infty}^0\wedge W_1]$, the last two calculations show that the upper right entry of $d_{\hat A_{\infty}^0}W_1$ vanishes, and similarly for the other entries.
\end{proof}

For the following, we make the same assumptions on the path $k$ as in \textsection{\ref{subsect:pathanditslift}}.
\begin{prop}  \label{prop:LimitingHorizontalLift}
	Let $A_{\infty}$ be the unitary connection associated to a limiting configuration in $\Hscr_\infty^{-1}(q)$, and write
$
A_{\infty} = A^{0}_{\infty} + \eta
$, 
$\eta=\widehat\eta\otimes W_2$, for $\widehat\eta$ a harmonic Prym differential (see \textsection{\ref{sec:prym-limiting}}).
 Define the
 function $\map{\vartheta}{[0,L]}{\R}$  by
\begin{equation} \label{eqn:theta}
	\vartheta(\sigma) \deq -2i \cdot \int_{\khat([0,\sigma])} \widehat\eta\ .
\end{equation}
Then for $j\in\{1,2,3\}$ the parallel transport of the hermitian endomorphisms $W_{i}(\khat_{0})$ of $\pi^{\ast} E_{\khat(0)}$  along the path $\khat$ with respect to the connection $d_{\Ahat_{\infty}}$ is given by
\begin{align}
\begin{split} \label{eqn:W-parallel}
	\Pi_{\sigma}^{\khat,\Ahat_{\infty}}W_{2}(\khat(0)) &= W_{2}(\khat(\sigma)) \\
	\Pi_{\sigma}^{\khat,\Ahat_{\infty}}W_{1}(\khat(0)) &= \cos(\vartheta(\sigma)) \cdot W_{1}(\khat(\sigma)) - \sin(\vartheta(\sigma) \cdot W_{3}(\khat(\sigma)) \\
	\Pi_{\sigma}^{\khat,\Ahat_{\infty}}W_{3}(\khat(0)) &= \sin(\vartheta(\sigma) \cdot W_{1}(\khat(\sigma)) + \cos(\vartheta(\sigma) \cdot W_{3}(\khat(\sigma))\ ,
\end{split}
\end{align}
for $0\le \sigma \le L$.
\end{prop}

\begin{proof}
The first line of \eqref{eqn:W-parallel}  is clear, since $W_2$ is parallel and commutes with $\eta$. For the rest, 
 using \eqref{eqn:commutation},
\begin{align*}
d_{\widehat A_\infty}W_1 &=\hat\eta\otimes [W_2, W_1]= -2i\hat\eta\otimes W_3\ , \\
d_{\widehat A_\infty}W_3 &=\hat\eta\otimes [W_2, W_3]= 2i\hat\eta\otimes W_1\ .
\end{align*}
Writing
\begin{align*}
\tilde W_1&= \cos\vartheta\cdot W_1 - \sin\vartheta\cdot W_3\ , \\
\tilde W_3&= \sin\vartheta\cdot W_1 + \cos\vartheta\cdot W_3\ ,
\end{align*}
we see that $ d_{\widehat A_\infty}\tilde W^i=0$ if the derivative $\dt{\vartheta}=-2i\hat\eta$. The result follows. 
\end{proof}

\subsubsection{Quasi-transverse paths with vertical ends}\label{subsect:pathanditslift} 
Fix a holomorphic quadratic differential $q \in \SQD^{\ast}(X)$, and consider a piecewise $C^1$ path $k:[0,L]\to X$ that is quasitransverse to the horizontal foliation $\mathcal{F}_{q}^{h}$ and meets the zeroes of $q$ precisely at its endpoints.
In particular,
 this means that the parameter interval of $k$ admits a subdivision $0=\sigma_0<\sigma_1<\cdots<\sigma_m=L$ such that $k$ restricted to $[\sigma_{i-1},\sigma_i]$ alternates between vertical and horizontal paths. We say that $k$ has \textbf{vertical ends} if the following conditions are satisfied:

\begin{compactenum}[(i)]
	\item the limits $\lim_{\sigma\downarrow 0}\dt k(\sigma)$ and $\lim_{\sigma\uparrow L}\dt k(\sigma)$ are both nonzero;
	\item  the restrictions $k\bigr|_{[0,\sigma_1]}$ and $k\bigr|_{[\sigma_{m-1},L]}$ are both vertical.
\end{compactenum}

We will denote by $\partial k$ the section of $k^\ast(K_X^{-1})$ induced by the derivative $\dt k$ of $k$.
The quadratic differential $q$ may be viewed as a section of ${\rm Sym}^2(K_X)$, and so it defines a function on ${\rm Sym}^2(K_X^{-1})$. We will denote this function applied to $\partial k\otimes\partial k$ by $q(\partial k,\partial k)$. In local coordinates where $q=q(z)dz^2$, this is simply
$q(\partial k,\partial k)(\sigma)=q(z(\sigma))(\dt z(\sigma))^2$. With this understood, if $k$ is parametrized by arc length locally near $\sigma=0,L$, condition (ii) above implies that
\begin{equation} \label{eqn:asymp-vert}
  \frac{q(\partial k,\partial k)(\sigma)}{\Vert q\Vert (k(\sigma))} = -1 
\end{equation}
for $\sigma$ in $[0,\sigma_1]$ or $[\sigma_{m-1},L]$. 
Recall from \textsection{\ref{sec:measured-foliations}} that since $k$ is assumed to be quasitransverse, we may find a lift $\hat k :[0,L]\to \widehat X_q$ of the path $k$ to the spectral curve 
such that $\imag(\lambda_{\SW}(\dt{\hat k}))\leq 0$. The path $\hat k$ is piecewise $C^1$ and meets the zeroes of $\lambda_{\SW}$ precisely at its endpoints. Using condition (ii) above, it is easy to show
that  the endomorphisms $W_i(\hat k(\sigma))$ in \eqref{eqn:W} extend continuously to the closed interval $[0,L]$.

In the following we suppose that $t>t_0$ is sufficiently large. Let $(A_t,\Psi_t)$ be a solution to the self-duality equations and consider a nearby approximate solution $(A_t^{\app},\Psi_t^{\app})$ such that the difference between these two pairs is exponentially small in $t$ (\emph{cf}.~\textsection\ref{subsect:approxsolutions}).
We indicate with a hat the respective pullbacks of $\Psi_t$ and $\Psi_t^{\app}$ to $\sqrt{-1}\gfrak_{\widehat E}$-valued differential forms along $\hat k$ on the spectral curve, \emph{i.e.}\ we let
\begin{equation*}
 \widehat\Psi_{t} := \pi^{\ast} \Psi_{t} \in \Omega^{1}(\widehat X_q,\sqrt{-1}
 \gfrak_{\widehat E})\ ,
 \end{equation*}
 and similarly for $\Psi^{\app}_{t}$.

\begin{prop}\label{prop:NormalizedApproximatePsi}
Fix a piecewise $C^1$ path $k:[0,L]\to X$ that is quasitransverse to the horizontal foliation $\mathcal{F}_{q}^{h}$ with vertical ends and meets the zeroes of $q$ precisely at its endpoints. Then for $\sigma = 0,L$ we have that
\begin{equation*}
	\widehat\Psi_t^{\app}(\partial \hat k(\sigma)) = W_1(\hat k(\sigma))\qquad\textrm{and}\qquad \widehat\Psi_t^{\app}(J\circ \partial \hat k(\sigma)) = W_2(\hat k(\sigma)).
\end{equation*}
\end{prop}

\begin{proof}
  Recall that near the zeroes of $q$,
 \begin{equation*} 
 	\Phi^{\app}_{t} = \begin{pmatrix} 0 & e^{h_{t}(\Vert q\Vert)}\,\Vert q\Vert^{1/2} \\ e^{-h_{t}(\Vert q\Vert)}\,\Vert q\Vert^{-1/2}\,q & 0 \end{pmatrix}\ .
\end{equation*}
 In terms of the tautological section,  $\pi^{\ast} q=\lambda^2$, the pullback of $\Phi^{\app}_{t}$ to the spectral curve can be written in the form
\begin{equation} \label{eqn:PhihatApp}
	\widehat\Phi^{\app}_{t} = \begin{pmatrix} 0 & e^{h_{t}(\Vert\lambda\Vert^{2})}\,\Vert\lambda\Vert\,\lambda^{-1} \\ e^{-h_{t}(\Vert\lambda\Vert^{2})}\,\Vert\lambda\Vert^{-1}\,\lambda & 0 \end{pmatrix} \otimes \lambda_{\SW}\ .
\end{equation}
Similarly, 
\begin{equation}  \label{eqn:PhihatAppStar}
	\big(\widehat\Phi^{\app}_{t}\big)^{\ast} =  \begin{pmatrix} 0 & e^{-h_{t}(\Vert\lambda\Vert^{2})}\,\Vert\lambda\Vert\,\lambda^{-1} \\ e^{h_{t}(\Vert\lambda\Vert^{2})}\,\Vert\lambda\Vert^{-1}\,\lambda & 0 \end{pmatrix} \otimes \overline{\lambda_{\SW}}\ .
\end{equation}
Now we calculate:
\begin{align*}
\lambda_{\SW}(\partial \hat k)&=\lambda(\hat k)\otimes\pi^\ast(\partial k) \\
(\lambda_{\SW}(\partial \hat k))^2&= (\pi^\ast q)(\hat k)\otimes \pi^\ast(\partial k)^2 =\pi^\ast(q(\partial k, \partial k))\ .
\end{align*}
Since $k$ is assumed to be quasitransverse with vertical ends, by the condition \eqref{eqn:asymp-vert} it follows that, locally near $\sigma=0,L$,
$$
\left(\frac{\lambda_{\SW}(\partial \hat k)}{\Vert\lambda\Vert\circ\hat k}\right)^2=
\pi^\ast\left(\frac{q(\partial k, \partial k)}{\Vert q\Vert \circ k}\right) = -1\ ,
$$
and so by the choice of lift we have
\begin{equation} \label{eqn:sw-limit}
\frac{\lambda_{\SW}(\partial \hat k)}{\Vert\lambda\Vert\circ\hat k} = - i\ .
\end{equation}
Similarly,
\begin{equation} \label{eqn:sw-limit2}
\frac{\lambda_{\SW}(J\circ\partial \hat k)}{\Vert\lambda\Vert\circ\hat k} =  1\ .
\end{equation}
Inserting $\partial \hat k$ into \eqref{eqn:PhihatApp} and rearranging the resulting terms slightly yields along $\hat k$ the endomorphism field
\begin{equation*}
	\widehat\Phi_t^{\app}(\partial \hat k)= \begin{pmatrix} 0 & e^{h_{t}(\Vert\lambda\Vert^{2})}\, \|\lambda\|^2\lambda^{-1} \\ e^{-h_{t}(\Vert\lambda\Vert^{2})}\lambda & 0 \end{pmatrix} \frac{\lambda_{\SW}(\partial \hat k)}{\Vert\lambda\Vert },
\end{equation*}
and similarly
\begin{equation*}
	\big(\widehat\Phi_t^{\app}\big)^{\ast}(\partial \hat k)=\begin{pmatrix} 0 & e^{-h_{t}(\Vert\lambda\Vert^{2})}\|\lambda\|^2\lambda^{-1} \\ e^{h_{t}(\Vert\lambda\Vert^{2})}\lambda & 0 \end{pmatrix}  \frac{ \overline{\lambda_{\SW}}(\partial \hat k)}{\Vert\lambda\Vert }  .
\end{equation*}
By Lemma \ref{f_t-h_t-function} (iv),  $
\exp(\pm h_t(\Vert \lambda\Vert^2))\sim \Vert \lambda\Vert^{\mp 1}$, for $\Vert\lambda\Vert$ small. Together with \eqref{eqn:sw-limit} this implies the convergence
\begin{equation*}
	\widehat\Psi_t^{\app}(\partial \hat k) \to  \begin{pmatrix}
 0& -i	\Vert\lambda\Vert\lambda^{-1} \\
 i\Vert\lambda\Vert^{-1}\lambda &0
 \end{pmatrix} = W_1
\end{equation*}
as $\Vert\lambda\Vert\to0$.
In a completely analogous way one obtains that along $\hat k$  
\begin{equation*}
	\widehat\Psi_t^{\app}(J\circ\partial \hat k) \to  \begin{pmatrix}
 0&  	\Vert\lambda\Vert\lambda^{-1} \\
 \Vert\lambda\Vert^{-1}\lambda &0
 \end{pmatrix} = W_2
\end{equation*}
as $\Vert\lambda\Vert\to0$. This proves the Proposition.
\end{proof}

\subsubsection{Limit of bending for connections}\label{sect:limittotalbending}

We relate the limit as $t\to\infty$ of the bending $\Theta_k(A_t, \Psi_t)$ defined in \textsection{\ref{subsec:HorizontalLifts}} to periods of Prym differentials on the spectral curve. This is the key result of this section.

\begin{prop} \label{prop:LimitBendingOfPairs}
 	Fix a holomorphic quadratic differential $q \in \SQD^{\ast}(X)$. Let $\map{k}{[0,L]}{X}$ be a piecewise $C^{1}$ path, and fix a lift $\map{\khat}{[0,L]}{\Xhat_{q}}$ to the spectral curve such that $\pi \circ \khat = k$. Assume that $k$ is quasi-transverse to the horizontal foliation $\mathcal{F}_{q}^{h}$ with vertical ends, and meets the zeroes of $q$ precisely at its endpoints. Consider a family $[(A_{t},t\,\Psi_{t})] \in \Hscr^{-1}(t^2 q_{t})$ 
for $q_t\in \SQD^\ast(X)$. Letting $t\to \infty$, suppose that $q_t\to q$ and that $[(A_{t},t\,\Psi_{t})]$ converges to $[(A_{\infty},\Psi_{\infty})] \in \Hscr_\infty^{-1}(q)$ 
in the sense of Definition \ref{def:ConvergenceOfHiggsPairs}.
Write $A_{\infty} = A^{0}_{\infty} + \eta$ with a unique $1$-form $\eta \in \mathcal{H}^{1}(X^{\times}_{q},L_{q})$ as in Proposition \ref{prop:LimitingHorizontalLift}. Then
\begin{equation}\label{eq:limitbendingangle}
\lim_{t\to\infty} \Theta_k([(A_t, \Psi_t)]) = -2i  \int_{\khat} \hat\eta \mod 2\pi\ZBbb\ ,
\end{equation}
where $\hat\eta \in \mathcal{H}^{1}_{\mathrm{odd}}(\Xhat_{q}, i\RBbb)$ is the Prym differential corresponding to $\eta$ from Proposition \ref{prop:prym-limiting}.
\end{prop}

\begin{proof}
The proof is in seven steps.

\medskip\noindent \textbf{Step 1.}
By Definition \ref{def:ConvergenceOfHiggsPairs} (Approximation), there exists a family of $1$-forms $\eta_{t} \in \Om^{1}(X, \g_E)$ as in eq.\ \eqref{eq:genlimitingconf} such that the difference $(A^{\app}_{t}(q_{t})+\eta_{t},\Psi^{\app}_{t}(q_{t})) - (A_{t},\Psi_{t})$  satisfies an  exponentially decaying $C^{\ell}$ bound in the parameter $t$. Hence the difference of the holonomies along the path $k$ in \eqref{eqn:TotalParallelTransport} corresponding to the connections $A^{\app}_{t}(q_{t})+\eta_{t}$ and $A_{t}$ tends to zero as $t \to \infty$. We conclude that it suffices to prove the claim with the family $(A_{t},\Psi_{t})$ replaced by the family $(A^{\app}_{t}(q_{t})+\eta_{t}, \Psi^{\app}_{t}(q_{t}))$.

\medskip\noindent \textbf{Step 2.}
Recall from \textsection \ref{subsect:pathanditslift} that by our assumptions on the path $k$, the parameter interval of $k$ admits a subdivision $0=\sigma_0<\sigma_1<\cdots<\sigma_m=L$ such that $k$ restricted to the subintervals $[\sigma_{i-1},\sigma_i]$ alternates between vertical and horizontal paths. Since $q_{t} \to q$ as $t \to \infty$, and hence also the zeroes of $q_{t}$ converge to the zeroes of $q$, we may choose a family $\map{k_{t}}{[0,L]}{X}$ of piecewise $C^{1}$ paths with the following properties:
\begin{compactenum}[(i)]
	\item $k_{t}$ meets the zeroes of $q_{t}$ precisely at its endpoints;
	\item $k_{t}$ is quasi-transverse to the horizontal foliation $\mathcal{F}_{q_{t}}^{h}$ with vertical ends;
	\item $k_{t} \to k$ in $C^{1}$ as $t \to \infty$ on each subinterval $[\sigma_{i-1},\sigma_i]$ for $1 \le i \le m$.
\end{compactenum}
For each $t$, we then fix a lift $\map{\khat_{t}}{[0,L]}{\Xhat_{q_{t}}}$ to the spectral curve $\pi_{t}: \Xhat_{q_{t}} \to X$ such that $\pi_{t} \circ \khat = k$.

\medskip\noindent \textbf{Step 3.}
For each fixed parameter $t$ consider the family $(A^{\app}_{s}(q_{t})+\eta_{t}, \Psi^{\app}_{s}(q_{t}))$ for $s > 0$. We recall from \textsection \ref{subsec:HorizontalLifts} the definition of bending, and apply it to the pair $(A^{\app}_{s}(q_{t})+\eta_{t}, \Psi^{\app}_{s}(q_{t}))$ and the path $k_{t}$.

\medskip

We shall be working on the spectral curve $\Xhat_{q_{t}}$. Let us denote by $(\widehat A^{\app}_{s}(q_{t})+ \widehat \eta_{t}, \widehat \Psi^{\app}_{s}(q_{t}))$ the pull back of the pair $(A^{\app}_{s}(q_{t})+\eta_{t}, \Psi^{\app}_{s}(q_{t}))$ along the projection $\pi_{t}: \Xhat_{q_{t}} \to X$, and fix a lift $\map{\khat_{t}}{[0,L]}{\Xhat_{q_{t}}}$ of $k_{t}$ such that $\pi \circ \khat_{t} = k_{t}$. Here $\hat\eta_{t} \in \mathcal{H}^{1}_{\mathrm{odd}}(\Xhat_{q_{t}}, i\RBbb)$ is the Prym differential corresponding to $\eta_{t}$ as in \textsection \ref{subsect:prymdifferentials}. Keeping $\sigma = 0,L$ fixed, by Proposition \ref{prop:NormalizedApproximatePsi} we may define the endomorphisms
\begin{equation*}
\widehat V_{t}(\sigma):= \widehat\Psi^{\app}_{s}(q_{t})(\partial \hat k_t(\sigma)) =  W_{1}(\hat k_{t}(\sigma))  
\end{equation*}
and
\begin{equation*}
\widehat N_{t}(\sigma):= \widehat\Psi^{\app}_{s}(q_{t})(J\circ\partial \hat k_t(\sigma)) =  W_{2}(\hat k_{t}(\sigma))\ . 
\end{equation*}
 
Note that these do not depend on $s$. Using the commutation relations from \eqref{eqn:commutation} it follows that
\begin{equation*}
\sqrt{-1}[\widehat N_{t}(\sigma),\widehat V_{t}(\sigma)] = \sqrt{-1}[W_{2}(\hat k_{t}(\sigma)),W_{1}(\hat k_{t}(\sigma))] = 2 \, W_{3}(\hat k_{t}(\sigma))\ .
\end{equation*}
Next we define the endomorphism
\[
\widehat n_{t}(0) := \sqrt{-1}[\widehat N_{t}(0),\widehat V_{t}(0)] = 2 \, W_{3}(\hat k_{t}(0)) 
\]
and consider its parallel transport
\begin{equation} \label{eqn:ParallelTransportInStep2}
\widetilde n_{t,s}(L) := \Pi_{L}^{\hat k_{t},\widehat A^{\app}_{s}(q_{t})+ \widehat \eta_{t}} \, \widehat n_{t}(0)
\end{equation}
in $\sqrt{-1}\gfrak_E$ along the path $\hat k_{t}$ with respect to the connection $\widehat A^{\app}_{s}(q_{t})+ \widehat \eta_{t}$. Let $\hat P_{t}(L) \subset \sqrt{-1}\,\widehat\g_{E,\hat k_{t}(L)}$ be the orthogonal complement to $\widehat N_{t}(L) = W_{2}(\hat k_{t}(L))$. By Proposition \ref{lem:W-1}, a frame for this complement is determined by $W_{1}(\hat k_{t}(L))$ and $W_{3}(\hat k_{t}(L))$. We use this ordering of the frame to define an orientation on the plane $\hat P_{t}(L)$. The bending
\begin{equation} \label{eqn:bendingStep2}
\Theta_k(A^{\app}_{s}(q_{t})+\eta_{t}, \Psi^{\app}_{s}(q_{t})) \in \RBbb/2\pi\ZBbb
\end{equation}
is then given by the angle from $\sqrt{-1}[\widehat N_{t}(L),\widehat V_{t}(L)] = 2 \, W_{3}(\hat k_{t}(L))$ to the orthogonal projection of the endomorphism $\widetilde n_{t,s}(L)$ to $\hat P_{t}(L)$ with respect to this orientation.

\medskip\noindent \textbf{Step 4.}
Recall from \textsection \ref{subsect:approxsolutions} that $A^{\app}_{s}(q_{t}) \to A^0_{\infty}(q_{t})$ as $s \to \infty$ in $C^{\infty}$ locally on compact subsets of $X_{q_{t}}^{\times}$, where $A^0_{\infty}(q_{t})$ is the Fuchsian connection from \eqref{eq:stdlimconn}. Then clearly we also have the local $C^{\infty}$ convergence $A^{\app}_{s}(q_{t}) +\eta_t\to A^0_{\infty}(q_{t})+\eta_t$ as $s \to \infty$. 
In preparation for Step 5, we now prove that there exists $t_{0} = t_{0}(q)>0$ such that the following holds: For every $\varepsilon>0$ and $\ell\geq0$ there exists $s_{0} = s_{0}(\varepsilon,q,\ell) \geq t_{0}$ such that
\begin{equation} \label{eqn:LpConvergenceOfConnectionsOnPath}
	\big\lVert k_{t}^{\ast} \big(A^{\app}_{s}(q_{t})+\eta_t\big) - k_{t}^{\ast}\big( A^0_{\infty}(q_{t}) +\eta_t\big)\big\rVert_{ C^{\ell}([0,L])} < \varepsilon
\end{equation}
for all $s \geq s_{0}$ and every $t \geq t_{0}$.
Here $k_{t}^{\ast} \big(A^{\app}_{s}(q_{t})+\eta_t\big)$ denotes the pull back of the connection $A^{\app}_{s}(q_{t})$ along the path $k_{t}: [0,L] \to X$, and likewise for $k_{t}^{\ast} \big( A^0_{\infty}(q_{t}) +\eta_t\big)$.

\medskip

Locally on each punctured disk $\D_{p}^{\times}$,   endowed with polar coordinates $(r,\theta)$, the connection $A^{\app}_{s}(q_{t})$ takes the form
\begin{equation}
A^{\app}_{s}(q_{t})(r,\theta)=f_s(r)\begin{pmatrix}
-i&0\\0&i
\end{pmatrix}d\theta\ ,
\end{equation}
with a smooth function $f_s\colon[0,\infty)\to \R$ as in \textsection \ref{subsect:approxsolutions}. Hence  writing the radial and angular components of the path $\sigma\mapsto k_{t}(\sigma)$ as $k_{t}(\sigma)=(r(\sigma),\theta(\sigma))$ it follows that
\begin{equation}
k_{t}^{\ast} A^{\app}_{s}(q_{t})(\sigma) =f_s(r(\sigma)) \begin{pmatrix}
-i&0\\0&i
\end{pmatrix} \dt \theta(\sigma) \, d\sigma\ .
\end{equation}
Since by assumption $k_{t}$ has vertical ends and meets the zeroes of $q_{t}$ precisely at its endpoints, we see that $\dt\theta(\sigma) $ and hence $k_{t}^{\ast} A^{\app}_{s}(q_{t})$ vanishes identically outside  some proper subinterval $[L_1,L_2]\subset [0,L]$. This subinterval may be chosen independently of $t$. 
   Definition \ref{def:ConvergenceOfHiggsPairs} (iii)  implies that  (after shrinking the disk $\D_{p}$ slightly if necessary, so that $k_{t}([L_1,L_2])$ lies outside  $\D_{p}$)   the family of functions $s\mapsto f_s \circ r$ converges in $C^{\ell}([L_1,L_2])$ to the function  $f_{\infty}\circ r $ as $s \to \infty$. This proves the claim.

\medskip\noindent \textbf{Step 5.}
Keep the constant $t_{0} = t_{0}(q) > 0$ from Step 4. We consider the bending in \eqref{eqn:bendingStep2} for large $s$, and prove that for every $\varepsilon>0$ there exists $s_{0} = s_{0}(\varepsilon,q) \geq t_{0}$ such that
\begin{equation} \label{eqn:ConvergenceOfBending}
	\left| \Theta_k(A^{\app}_{s}(q_{t})+\eta_{t}, \Psi^{\app}_{s}(q_{t})) - \left( -2i  \int_{\khat_{t}} \hat\eta_{t} \mod 2\pi\ZBbb \right) \right| < \varepsilon
\end{equation}
for all $s \geq s_{0}$ and every $t \geq t_{0}$.

\medskip

To see this, first note that after passing to the spectral curve $\Xhat_{q_{t}}$, by Step 3 we have the estimate
\begin{equation} \label{eqn:LpConvergenceOfConnectionsOnPathUpstairs}
	\lVert \hat k_{t}^{\ast} \widehat A^{\app}_{s}(q_{t}) - \hat k_{t}^{\ast} \widehat A^0_{\infty}(q_{t}) \rVert_{L^{p}([0,L])} < \varepsilon
\end{equation}
for all $s \geq s_{0}$ and every $t \geq t_{0}$, where $\widehat A^0_{\infty}(q_{t})$ denotes the pull back of $A^0_{\infty}(q_{t})$ along the projection $\pi: \Xhat_{q_{t}} \to X$. Let us now compare the parallel transports
\[
\widetilde n_{t,s}(L) = -2 \, \Pi_{L}^{\hat k_{t},\widehat A^{\app}_{s}(q_{t})+ \widehat \eta_{t}} W_{3}(\hat k_{t}(0))
\]
from \eqref{eqn:ParallelTransportInStep2} with the parallel transport
\[
\widetilde n_{t,\infty}(L) := -2 \, \Pi_{L}^{\hat k_{t},\widehat A^0_{\infty}(q_{t}) + \widehat \eta_{t}} W_{3}(\hat k_{t}(0)).
\]
It follows from \eqref{eqn:LpConvergenceOfConnectionsOnPathUpstairs} that there is some constant $C>0$ such that
\begin{equation} \label{eqn:ConvergenceParallelTransport}
	 \left| \widetilde n_{t,s}(L) - \widetilde n_{t,\infty}(L) \right| < C \varepsilon
\end{equation}
for all $s \geq s_{0}$ and every $t \geq t_{0}$. Now by Proposition \ref{prop:LimitingHorizontalLift}, we have
\[
\widetilde n_{t,\infty}(L) = \sin(\vartheta_{t}(L)) \cdot W_{1}(\khat_{t}(L)) + \cos(\vartheta_{t}(L)) \cdot W_{3}(\khat_{t}(L)),
\]
where
\[
\vartheta_{t}(L) = -2i \cdot \int_{\khat_{t}} \widehat\eta_{t}\ .
\]
Observe that the endomorphism $\widetilde n_{t,\infty}(L)$ is contained in the plane $\hat P_{t}(L)$ defined in Step 2, and that the angle from $\sqrt{-1}[\widehat N_{t}(L),\widehat V_{t}(L)] = 2 \, W_{3}(\khat_{t}(L))$ to $\widetilde n_{t,\infty}(L)$ with respect to the orientation on $P_{t}(L)$ is given by $\vartheta_{t}(L)$.
 
The estimate in \eqref{eqn:ConvergenceOfBending} now follows from \eqref{eqn:ConvergenceParallelTransport} and the definition of bending in Step 2.

\medskip\noindent \textbf{Step 6.}
By assumption and Steps 1 and 2, letting $t \to \infty$ we have that $q_{t} \to q$, $\eta_{t} \to \eta$ and $k_{t} \to k$, which immediately implies that
\begin{equation} \label{eqn:ConvergenceEta}
\lim_{t\to\infty} \int_{\khat_{t}} \widehat\eta_{t} = \int_{\khat} \widehat\eta\ .
\end{equation}

\medskip\noindent \textbf{Step 7.}
Combining Steps 5 and 6 we infer that in the estimate
\begin{multline*}
	\left| \Theta_k(A^{\app}_{t}(q_{t})+\eta_{t}, \Psi^{\app}_{t}(q_{t})) - \left( -2i  \int_{\khat} \hat\eta \mod 2\pi\ZBbb \right) \right| \\
	\leq \left| \Theta_k(A^{\app}_{t}(q_{t})+\eta_{t}, \Psi^{\app}_{t}(q_{t})) - \left( -2i  \int_{\khat_{t}} \hat\eta_{t} \mod 2\pi\ZBbb \right) \right| \\
	 + \left| \left( -2i  \int_{\khat_{t}} \hat\eta_{t} \right) - \left( -2i  \int_{\khat} \hat\eta \right) \right|
\end{multline*}
both terms on the right-hand side tend to zero as $t \to \infty$. The Proposition is proved.
\end{proof}

\begin{remlabel} \upshape
Proposition \ref{prop:LimitBendingOfPairs} and eq.\ \eqref{eq:limitbendingangle} apply equally well to the modified saddle connections (which are not quasitransverse). 
\end{remlabel}

 \subsection{Comparison of bending} \label{sec:bending-comparison}

In this section we show that for large energy, the bending of equivariant harmonic maps defined in \textsection{\ref{sec:bending-maps}}
nearly coincides with the bending of the associated Higgs pair along quasitransverse paths. 
The main result is Theorem \ref{thm:asymptotic-bending} below. First, we need a somewhat standard preliminary result on parallel translation
 which we provide in the next subsection.

\subsubsection{Parallel translation for $C^1$-close curves}\label{sec:C1-close} 
Let $c$ and $c_0$ be piecewise $C^1$ curves $[0,L]\to \HBbb^3$. 
Fix $\varepsilon>0$. We say that $c$ and $c_0$ are {\bf $C^0_\varepsilon$-close} if 
\begin{equation} \label{eqn:C0-close}
\max_{0\leq t\leq L} d_{\HBbb^3}(c(t), c_0(t))<\varepsilon\ .
\end{equation} 

Let us view $c$ and $c_0$ as curves in the hermitian model $\mathscr D$ of $\HBbb^3$ (see \textsection{\ref{subsect:matrixmodel}}).
Recall the metric on $T\HBbb^3\otimes \CBbb$ defined in \eqref{eqn:M-metric} for the trace model. A $C^0$-bound on the distance in $\HBbb^3$ between $c$ and $c_0$ induces one on the pointwise norms
of $(1-cc_0^{-1})$ and $(1-c_0c^{-1})$. Using this fact it is easy to prove the following.

\begin{lem} \label{lem:trace-estimate}
There  are constants $C(\varepsilon)\geq 1$, $\displaystyle\lim_{\varepsilon\to 0}C(\varepsilon)=1$,
with the following significance. If $c$ and $c_0$ are $C^0_\varepsilon$-close, then for all $M\in T\HBbb^3\otimes \CBbb$, and all $t\in [0,L]$,
$$
C(\varepsilon)^{-1}\Vert M\Vert_{ c(t)}\leq \Vert M\Vert_{ c_0(t)}\leq C(\varepsilon)\Vert M\Vert_{ c(t)}\ .
$$
\end{lem}

\begin{definition} \label{def:C1-close}\upshape
Let $c$ and $c_0$ be as above. 
Fix $\varepsilon>0$. We say that  $c$ and $c_0$ are {\bf 
$C^1_\varepsilon$-close} if they are $C^0_\varepsilon$-close, and 
$$\max_{0\leq t\leq L}
\left\Vert \dt c c^{-1}-\dt c_0 c_0^{-1}
\right\Vert_{ c_0(t)}<\varepsilon\ .
$$
\end{definition}
We emphasize that here we view $\dt c c^{-1}$ and $\dt c_0 c_0^{-1}$ as sections  of the trivial bundle $T\HBbb^3\otimes\CBbb\simeq \HBbb^3\times \CBbb^3$, and using this trivialization we compare vectors at arbitrary fibers. 
Note that because of Lemma \ref{lem:trace-estimate}, the relationship of being $C^1$-close is symmetric (after possibly multiplying  $\varepsilon$ by a distortion that is nearly $1$).
The curves are not assumed to be parametrized by arc length.

\begin{lem} \label{lem:v-estimate}
Let $c$ and $c_0$ be curves in $\HBbb^3$. Suppose $v(0)\in T_{c(0)}\HBbb^3$, $v_0(0)\in T_{c_0(0)}\HBbb^3$ are unit vectors, and let $v(t)$ and $v_0(t)$ denote parallel translation along $c(t)$ and $c_0(t)$, respectively. If $c$ and $c_0$ are $C^1_\varepsilon$-close with $0<\varepsilon\leq 1/4L$, then
$$
\max_{0\leq t\leq L} \Vert v(t)-v_0(t)\Vert_{ c_0(t)}\leq 2
 \Vert v(0)-v_0(0)\Vert_{ c_0(0)}+4L\varepsilon\ .
$$
\end{lem}

\begin{proof}
By \eqref{eqn:levi-civita} we have
$$
\dt v(t)=\frac{1}{2}[\dt c c^{-1}, v(t)]\ ,\ \dt v_0(t)=\frac{1}{2}[\dt c_0 c_0^{-1}, v_0(t)]\ .
$$
Write:
$$
v(t)=v_0(t)+R(t)\ , \ \dt c c^{-1} = \dt c_0 c_0^{-1} + r(t)\ ,
$$
for traceless matrix valued functions $R(t), r(t)$. 
Hence, 
\begin{equation} \label{eqn:Rdot}
2\dt R(t)=[r(t), v_0(t)+R(t)]+[\dt c_0 c_0^{-1}, R(t)] \ .
\end{equation}
Now
\begin{align*}
\frac{d}{dt}\Vert R(t) \Vert^2_{ c_0}=
\frac{d}{dt}\tr(Rc_0 R^\ast c_0^{-1})&=
\tr(\dt Rc_0 R^\ast c_0^{-1})+\tr(Rc_0 \dt R^\ast c_0^{-1})\\
&\qquad +\tr(R\dt c_0 R^\ast c_0^{-1})-\tr(Rc_0 R^\ast c_0^{-1}\dt c_0 c_0^{-1})\ .
\end{align*}
One can see that the last two terms on the right hand side of  the equation above are cancelled by the last term on the right hand side of \eqref{eqn:Rdot} (and the similar equation for the adjoint). Thus we are left with
$$
\frac{d}{dt}\tr(Rc_0 R^\ast c_0^{-1})
=\frac{1}{2}\tr\left([r, v_0+R]c_0R^\ast c_0^{-1} \right)
+\frac{1}{2}\tr\left(Rc_0[v_0^\ast+R^\ast, r^\ast]c_0^{-1}\right)
\ .
$$
Since the norm of $r$ is less than $\varepsilon$, and $v_0(t)$ is a unit vector, we see 
that
\begin{equation} \label{eqn:diff-inequality}
\frac{d}{dt}\Vert R(t) \Vert^2_{ c_0}\leq 2\varepsilon\left(\Vert R(t) \Vert^2_{ c_0}+\Vert R(t) \Vert_{ c_0} \right)\ .
\end{equation}
Let $0\leq t_m\leq L$ be the point at which $\Vert R(t) \Vert^2_{ c_0}$ attains it maximum. Then from \eqref{eqn:diff-inequality} we have
\begin{align}
\Vert R(t_m) \Vert^2_{ c_0}-\Vert R(0) \Vert^2_{ c_0}
&= \int_0^{t_m}\frac{d}{dt}\Vert R(t) \Vert^2_{ c_0}\, dt\notag \\
&\leq 2\varepsilon \int_0^{t_m}\left(\Vert R(t) \Vert^2_{ c_0}+\Vert R(t) \Vert_{ c_0} \right)\, dt\notag\\
\Vert R(t_m) \Vert^2_{ c_0}-\Vert R(0) \Vert^2_{ c_0}
&\leq 2L\varepsilon\left(\Vert R(t_m) \Vert^2_{ c_0}+\Vert R(t_m) \Vert_{ c_0} \right)\ \label{eqn:final-inequality}.
\end{align}
Since we assume $\varepsilon\leq 1/4L$, it follows from \eqref{eqn:final-inequality} that 
$$
\Vert R(t_m) \Vert_{ c_0}\leq 2(\Vert R(0) \Vert_{ c_0}+2L\varepsilon)\ .
$$
This completes the proof.
\end{proof}

\subsubsection{Asymptotic equivalence of bending}

\begin{thm} \label{thm:asymptotic-bending}
Let $K\subset\QD^\ast(X)$ be a cone on a compact subset of $\SQD^\ast(X)$, and fix $\delta > 0$. 
Let $u:\widetilde X\to \HBbb^3$ be a $\rho$-equivariant harmonic
map, $\nabla=d_A+\Psi$ Higgs pair. Let $\Hopf(u)=-q\in K$, and
$\widetilde p^-$, $\widetilde p^+$ lifts of zeroes $p^-$, $p^+$
of $q$. Let $k$ be a quasi-transverse path $($or
 modified saddle connection$)$ from $p^-$ to $p^+$ that lifts to
a path $\widetilde k$ from $\widetilde p^-$ to $\widetilde p^+$.
Assume $k$ has small vertical ends and meets the zeroes of $q$
only at $p^-$ and $p^+$. Then if $\Vert q\Vert_1$ is
sufficiently large $($depending on $K,\varepsilon, k$$)$ we have
$$
| \Theta_u(\widetilde p^-, \widetilde p^+)-\Theta_k(A,\Psi)|<\delta\ .
$$
\end{thm}

\begin{proof}
Write $0=\sigma_0<\sigma_1<\cdots<\sigma_m=L$, so that $k$ restricted to $[\sigma_{i-1},\sigma_i]$ alternates between $C^1$ vertical and horizontal paths. By assumption, $k\bigr|_{[0,\sigma_1]}$ and $k\bigr|_{[\sigma_{m-1},L]}$ are vertical. 
Let $T$ be the tent with crease $\gamma_T$ associated to the totally geodesic planes $\DBbb_{u(\widetilde p^-)}$ and $\DBbb_{u(\widetilde p^+)}$.  
By Proposition \ref{prop: harmonic map C1 localizes near zeroes}, the images by $u$ of sufficiently small  hexagonal domains $\Omega_{\widetilde p^-}$ and $\Omega_{\widetilde p^+}$ are $C^1$-close to the planes $\DBbb_{u(\tilde p^-)}$ and $\DBbb_{u(\tilde p^+)}$. By Proposition \ref{prop:pullbackmetric} it follows that the image of $k\bigr|_{[\sigma_1,\sigma_{m-1}]}$ is $C^1_\varepsilon$-close to $\gamma_T$. 
By Proposition \ref{prop:NormalizedApproximatePsi},
for sufficiently small vertical ends,  the normal vector to $\DBbb_{u(\widetilde p^-)}$ is close to the vector $n(0)$
in \textsection{\ref{subsec:HorizontalLifts}}, and similarly at 
$\DBbb_{u(\widetilde p^+)}$. 
Hence, by Lemma \ref{lem:v-estimate}, parallel translation of the normal vector to $\Omega_{u(\widetilde p^-)}$ along $\gamma_T$ is close to the parallel translation $\widetilde n$ along $k$. 
The result now follows from Theorem \ref{thm:DonaldsonIsomorphism} (i) and (ii), and the discussion in \textsection{\ref{sec:bending-maps}}.
\end{proof}

\section{Pleated surfaces} \label{sec:pleated-surfaces}

In this section we review the notion of a transverse cocycle for a lamination. The key results are: Lemma \ref{lem:finite-approximation-bending} where we relate the bending cocycle of a pleated surface to its geometric bending in the sense of \textsection{\ref{sec:bending-maps}}; Theorem \ref{thm:prym-cocycle}, where we relate the group of bending cocycles to the torus of Prym differentials; and Theorem \ref{thm:bending-periods}, where we show that the limit of a bending cocycle is determined by the periods of a Prym differential.

\subsection{Transverse cocycles} \label{sec:transverse}

\subsubsection{Definitions} \label{subsec: transverse cocycles}
Let $\Lambda$ be a maximal geodesic lamination on a hyperbolic surface $S$ with underlying smooth surface $\Sigma$.  
Recall from \textsection{\ref{sec:measured-foliations}} that $\Pcal(\Lambda)$ denotes the set of plaques in $\HBbb^2\setminus\widetilde \Lambda$, and note that there is a free action of $\pi_1$ on $\Pcal(\Lambda)$ with finite quotient. 
 Let $G$ denote an abelian group that is either $\RBbb$ or $S^1\simeq \RBbb/2\pi\ZBbb$. A {\bf $G$-valued transverse cocycle} for $\Lambda$ is a map $\alphabold$ which sends every arc $k$ transverse to $\Lambda$ to an element $\alphabold(k) \in G$, and which satisfies the following two properties. First, for $k = k_1 \cup k_2$ a decomposition of $k$ into two subarcs with disjoint interiors, we have $\alphabold(k) = \alphabold(k_1) + \alphabold(k_2)$. Second, $\alphabold$ is invariant under  $\Lambda$-transverse homotopies, in the sense that if $k$ can be taken to $k'$ by a homotopy of $S$ that preserves $\Lambda$, then $\alphabold(k) = \alphabold(k')$.
In particular (see \cite[p.\ 7]{Bonahon:shearing}),  a $G$-valued transverse cocycle may be taken to be a function $\alphabold:\Pcal(\Lambda)\times\Pcal(\Lambda)\to G$ satisfying:
 \begin{compactenum}[(i)]
  \item Equivariance: $ \alphabold(g  P, g  Q)= \alphabold(  P,   Q)$ for all $g\in \Gamma$;
 \item Symmetry: $ \alphabold(  P,   Q)= \alphabold(  Q,   P)$;
  \item Additivity:  $ \alphabold(  P,   Q)= \alphabold(  P,   R)+ \alphabold(  R,   Q)$, if $  R$ separates $  P$ from $  Q$.
  \end{compactenum}
 We shall henceforth denote transverse cocycles by $\sigmabold$ for $G=\RBbb$ and $\betabold$ for $G=S^1$, and
 in the latter case  we will continue to use additive notation in (3).
 Denote by $\Hcal(\Lambda, G)$ the space of transverse cocycles. 
 For $G=\R$, this is a real vector space of dimension $6g-6$ (see \cite[p.\ 119]{Bonahon:97}). We refer to elements  $\sigmabold\in\Hcal(\Lambda, \RBbb)$ as {\bf shearing cocycles}. 
 The space $\Hcal(\Lambda, S^1)$ has two components, each of which is a $(6g-6)$-dimensional torus. We denote by $\Hcal^o(\Lambda, S^1)$ the component containing the identity cocycle: $\betabold(P,Q)=0$ for all $P,Q$. We refer to elements  $\betabold\in\Hcal^o(\Lambda, S^1)$ as {\bf bending cocycles}.  We will sometimes make reference to a norm $\Vert\cdot\Vert$ on $\Hcal(\Lambda, \RBbb)$, which is fixed once and for all.

Convergence of transverse cocycles for families of laminations  may be defined in a weak sense as functions on pairs of plaques. More precisely, recall from Remark \ref{rem:maximalization} that there is a 1-1 correspondence between the lifts of zeroes of $q$ and the plaques $\Pcal(\Lambda)$ of any maximalization $\Lambda$ of $\Lambda_q^h$. If $q_n\to q\in \QD^\ast(X)$, then $\Lambda_{q_n}^h$ converges in the Hausdorff sense to a lamination with $\Lambda_q^h$ as a sublamination. For maximalizations $\Lambda_n$, we define convergence $\Lambda_n\to \Lambda$  again in the Hausdorff sense. In this case, for $n$ sufficiently large we have bijections
$$
r^{\Lambda}_{\Lambda_n}: \Pcal(\Lambda)\isorightarrow \Pcal(\Lambda_n)\ .
$$
\begin{definition} \label{def:convergence-cocycles}\upshape
With the notation above, suppose $\Lambda_n\to \Lambda$. 
Let $\alphabold_n$ (\emph{resp}.\ $\alphabold$) be either shearing or bending cocycles for $\Lambda_n$ (\emph{resp}.\ $\Lambda$). We say that {\bf $\alphabold_n$ converges to $\alphabold$} (and write $\alphabold_n\to \alphabold$)
if 
$$
\lim_{n\to\infty}\alphabold_n(r^{\Lambda}_{\Lambda_n}(P),r^{\Lambda}_{\Lambda_n}(Q))=\alphabold(P,Q)\quad\forall\ P,Q\in \Pcal(\Lambda)\ .
$$
 \end{definition}
We note for clarification that $r^{\Lambda}_{\Lambda_n}$ does not, in general, preserve the separation relations of plaques, so the pull-back $\alpha_n\circ r^\Lambda_{\Lambda_n}$ of a cocycle $\alpha_n$ on $\Lambda_n$ will not necessarily satisfy the additivity condition on $\Pcal(\Lambda)$.

We will also need the following elementary properties of cocycles. 

\begin{prop} \label{prop:cocycle-properties}
Let $q\in \QD^\ast(X)$ and $\Lambda$ a maximalization of $\Lambda_q^h$.
\begin{compactenum}[(i)]
\item There is a finite set $\Pcal'\subset\Pcal(\Lambda)$ such that if $\alphabold\in \Hcal(\Lambda, G)$ vanishes on $\Pcal'\times\Pcal'$, then $\alphabold$ vanishes identically.
\item There is a complete train track  $\tau$ carrying $\Lambda_q^h$, and a bijection: $\Hcal(\tau,G)\simeq \Hcal(\Lambda, G)$. 
\end{compactenum}
\end{prop}

\begin{proof} (1) follows by finite dimensionality. More precisely, 
let $\tau'$ be a train track that \emph{snugly} carries $\Lambda$. Then by \cite[Thm.\ 11]{Bonahon:97}, $\Hcal(\Lambda, G)\simeq \Hcal(\tau', G)$. The identification assigns weights to branches of $\tau'$ that are equal to the value of the cocycle on the plaques defined by the complementary regions. Since there are only finitely many of these, the claim follows. 
For (2), the existence of $\tau$ follows from  \cite{PenHar}.
A train track $\tau''$ that snugly carries $\Lambda$ can be obtained by splitting $\tau$ along branches corresponding to saddle connections. Hence, 
$\Hcal(\Lambda, G)\simeq \Hcal(\tau'', G)$ as above. On the other hand, 
$\Hcal(\tau'', G)\simeq\Hcal(\tau, G)$ from properties of splittings. This completes the proof.
\end{proof}

 \subsubsection{Shearing cocycles} \label{subsec:ShearingCocycles} 
Given a marked hyperbolic surface $S$ with maximal geodesic lamination $\Lambda$ there is a uniquely defined transverse cocycle $\sigmabold\in \Hcal(\Lambda, \RBbb)$ called the {\bf shearing cocycle of $S$}. For the precise definition see \cite[p.\ 10]{Bonahon:shearing}. 
We will need the following formula for $\sigmabold$.

Let $P,Q\in \Pcal(\Lambda)$ and choose an arc $k$ from $P$ to $Q$ in $\widetilde S$ that is transverse to $\widetilde \Lambda$. 
 For each component $d$ of $k \setminus \widetilde\Lambda$ distinct from $P$ and $Q$,
  let $x_d^+$ and $x_d^-$ be the positive and negative endpoints of the (oriented) segment $d$, respectively. We let $d_-=P\cap k$ and $d_+=Q\cap k$. Define $x_{d_-}^+$ to be the positive endpoint of $d_-$, and  
  and $x_{d_+}^-$ the negative endpoint of $d_+$.  
We denote the leaves of $\widetilde \Lambda$ passing through $x_d^\pm$ by $g_d^\pm$, and similarly for $x_{d_\pm}$. 
  For each component $d$, $d_\pm$, let $h: g_d^\pm \to \RBbb$ denote the (signed) distance from the foot\footnote{The {\bf foot} of an edge of an ideal triangle is the point of intersection with the orthogonal geodesic from the opposing vertex.} determined by viewing the geodesics as boundaries of the ideal triangle corresponding to the component $d$. With this understood, we have the following expression for the shearing cocycle of $S$ (\cite[Lemma 7]{Bonahon:shearing}).

\begin{equation} \label{eqn:shearing-cocycle}
\sigmabold(P,Q)= h(x_{d_-}^+) - h(x_{d_+}^-)+\sum_{d\neq d_+, d_-} \left(h(x_d^+)-h(x_d^-)\right)\ .
\end{equation}
\noindent
We also note the following.
\begin{compactenum}[(i)]
\item
If $\sigmabold$ is the shearing cocycle of a hyperbolic surface $S$ and $\alphabold\in \Hcal(\Lambda, \RBbb)$ with $\Vert\alphabold\Vert$ sufficiently small, then $\sigmabold+\alphabold$ is the shearing cocycle of some hyperbolic surface \cite[Prop.\ 13]{Bonahon:shearing}. This is the generalization of Thurston's earthquake map.
\item
The map $T(\Sigma)\to \Hcal(\Lambda, \RBbb)$ which associates the shearing cocycle to a hyperbolic metric is injective onto an open convex polyhedral cone $\Ccal(\Lambda)$ \cite[Cor.\ 21]{Bonahon:shearing}.  
\end{compactenum}

\subsubsection{Bending cocycles} \label{subsec:BendingCocycles}
 Recall from the introduction that a pleated surface  $\mathsf{P}= (S, f, \Lambda, \rho)$ 
consists of a marked hyperbolic surface $S$, a maximal geodesic lamination $\Lambda\subset \Sigma$, and a map $f: \widetilde S\to \HBbb^3$  that is totally geodesic on the components of $\widetilde S\setminus\widetilde \Lambda$, maps leaves of $\widetilde\Lambda$ isometrically to geodesics, and is $\rho$-equivariant for a representation $\rho\colon \pi_1\to \PSL(2,\CBbb)$. Such a $\rho$, which in this paper we take to be in $R^o(\Sigma)$, is necessarily irreducible (see \cite[p.\ 36]{Bonahon:shearing}). We sometimes denote pleated surfaces by just $ f: \widetilde S\to \HBbb^3$ when context provides the other data.

In addition to the shearing cocycle for the hyperbolic surface $S$, there  a uniquely defined bending cocycle $\betabold\in \Hcal^0(\Lambda, S^1)$. As in the previous section, we will need a particular formula for this which we describe below.

Choose a $\rho$-equivariant differentiable vector field $v$ on $\H^3$ ``transverse to the image'' $f(\tilde{\Lambda})$ of $\tilde{\Lambda}$ under $f$. For the existence of such  we refer to \cite{Bonahon:97}, \textsection{11}.  
As in \textsection{\ref{subsec:ShearingCocycles}},
let $k$ be an arc transverse to $\widetilde\Lambda$ from plaque $P$ to $Q$.
At each endpoint $x_d^{\pm}$ we have two vectors: the ambient vector field $v$ restricted to $x_d^{\pm}$, and the vector $n$ which is normal to the plane containing the  plaque $R$ of $\tilde{S} \setminus \tilde{\Lambda}$ which contains $d$. 
Here the orientation of $n$ is such that the induced orientation of $f(R)$ by $f$ followed by $n$ is the orientation of $\H^3$.
Orient these leaves 
of $\Lambda$ (thought of here as a leaf of $\tilde{\Lambda} \subset \H^3$) 
so that its orientation is from 
\emph{right to left} with respect to $k$. The final geometric object we need is the normal plane $N$ to the image $f(g_d^\pm)$, which inherits an orientation from $f(g_d^\pm)$ and the orientation of $\H^3$.

 Set $a_{n,v}(x_d^{\pm})$ to be the angle from the projection of $v$ onto the normal plane $N$ to $n \in N$. Then we have the following expression for the bending cocycle (see \cite[Lemma~36]{Bonahon:97}). 

	\begin{equation} \label{equation: intrinsic bending formula}
	\betabold(P,Q) = a_{n,v}(x_{d_-}^+) - 	a_{n,v}(x_{d_+}^-)+ \sum_{d \neq d_+, d_-} [a_{n,v}(x_d^+) - a_{n,v}(x_d^-)] \in \R/2\pi\Z	\ .
	\end{equation}

We will use some of the details behind this expression. By \cite[Lemmas 4 and 5]{Bonahon:97} there are constants $K$, $A$, and $B$, depending only on $k$, such that   the number of components $d$ of $k\setminus \Lambda$ with divergence radius $r(d)=r\in \NBbb$ is at most $K$, and the length $\ell(d)$ of any such  component is bounded by $Be^{-Ar(d)}$.  Write the sum in \eqref{equation: intrinsic bending formula} as 
\begin{equation} \label{eqn:resum}
\sum_{d \neq d_+, d_-} [a_{n,v}(x_d^+) - a_{n,v}(x_d^-)]
=
\sum_{r=0}^\infty\sum_{\stackrel{d \neq d_+, d_-}{r(d)=r}}
[a_{n,v}(x_d^+) - a_{n,v}(x_d^-)]\ .
\end{equation}
Since $v$ is Lipschitz 
there is a constant $c_0>0$ 
such that
$$
\left|a_{n,v}(x_d^+) - a_{n,v}(x_d^-)\right|\leq c_0\ell(d)\leq c_oBe^{-Ar(d)}\ .
$$
Hence, the tail in the sum \eqref{eqn:resum} is estimated by 
\begin{equation} \label{eqn:estimate-tail}
\biggl| \sum_{r=R}^\infty\sum_{\stackrel{d \neq d_+, d_-}{r(d)=r}}
[a_{n,v}(x_d^+) - a_{n,v}(x_d^-)]\biggr|\leq
c_0KB\sum_{r=R}^\infty e^{-Ar}\leq \frac{c_0KB}{A}e^{-AR}\ .
\end{equation}

Bonahon proves that given a bending cocycle $\betabold\in \Hcal^o(\Lambda, S^1)$ and a hyperbolic surface $S$, there is an equivariant map $f: \widetilde S\to \HBbb^3$, well defined up to isometries, totally geodesic on the plaques and pleated along the  lamination $\Lambda$, and with bending cocycle $\betabold$. The map $f$ is by construction equivariant with respect to some representation whose conjugacy class $[\rho]\in R^o(\Sigma)$ depends only on the isomorphism class of the marked hyperbolic surface $S$, the lamination $\Lambda$, and the bending cocycle $\betabold$. Indeed, this construction gives a parametrization of (a portion of) $R^o(\Sigma)$. Set $[\rho]=B_\Lambda(\sigmabold, \betabold)$, where $\sigmabold\in \Ccal(\Lambda)$ is the shearing cocycle of $S$. 
Then we have
\begin{thm}[{\cite[Thm.\ D]{Bonahon:shearing}}]
The map \eqref{eqn:bonahon}
is a biholomorphism onto an open subset.
\end{thm}

\subsubsection{Bending cocycles and geometric bending} \label{sec:cocycle-vs-geometric-bending}
The following result will be crucial for the analysis later on. It  provides a relationship between the bending cocycle discussed here and the geometric bending introduced in \textsection{\ref{sec:bending-maps}}.
\begin{lem} \label{lem:finite-approximation-bending}
  Fix $\delta>0$ and some positive integer $M$. There is $\varepsilon_0>0$ depending only on $\delta$, $M$, and $\Lambda$, with the following property.  Let $f: \widetilde S\to \HBbb^3$ be a pleated surface with pleating lamination $\Lambda$.
Further, given $P,Q\in \Pcal(\Lambda)$ and a transverse arc $k$ from $P$ to $Q$, suppose $k$ can be written as a union $k_1\cup \cdots\cup k_m$, where $m\leq M$, and for each $i$ the pointed geodesics bounding the plaques intersecting $k_i$ are $C^1_\varepsilon$-close for $\varepsilon\leq \varepsilon_0$ at their intersections with $k_i$. Then there is a finite collection 
 $\{P_i\}_{i=1}^N$ of plaques separating $P$ and $Q$, $P_0=P$, $P_N=Q$,
such that for any choice of points $\widetilde p_i\in P_i$ in the interiors of the plaques,
\begin{equation} \label{eqn:geometric-estimate}
\left|
\betabold(P,Q)-\sum_{i=1}^N\Theta_{f}(\widetilde p_{i-1}, \widetilde p_i)
\right|<\delta\ .
\end{equation}
\end{lem}

\begin{proof}
Recall that $\beta(P,Q)=\beta(k)$. 
By the estimate in \eqref{eqn:estimate-tail}, we may find plaques $P_i$ as in the statement of the Lemma such that if we set $d_i=k\cap P_i$, then 
\begin{equation} \label{eqn:first-estimate}
\left| \beta(k)-\sum_{i=1}^N[a_{n,v}(x_{d_{i-1}}^+)- a_{n,v}(x_{d_i}^- )]\right|<\delta/2\ .
\end{equation}
We also assume, after a possible further subdivision, the leaves $g_{i-1}^+$ and $g_i^-$ of $\Lambda$ through $x_{d_{i-1}}^+$ and $x_{d_i}^-$ are $C^1_\varepsilon$-close for every $i=1,\ldots, N$ (where $\varepsilon$ is to be determined). This does not affect \eqref{eqn:first-estimate}.   It suffices to show that for $\varepsilon$ sufficiently small the following holds:
\begin{equation} \label{eqn:second-estimate}
\left| a_{n,v}(x_{d_{i-1}}^+)- a_{n,v}(x_{d_i}^-) -\Theta_{f}(x_{d_{i-1}}^+, x_{d_i}^- )\right|<\delta/2N\ .
\end{equation}
Here we have extended the definition of $\Theta_{f}(x_{d_{i-1}}^+, x_{d_i}^- )$ from that of $\Theta_{f}(p_{i-1}, p_{i})$ by using the tangent planes to the plaques containing $x_{d_{i-1}}^+$ and $x_{d_i}^-$ to determine the dihedral angles.
To simplify the notation, set 
$$
\Delta_i:= a_{n,v}(x_{d_{i-1}}^+)- a_{n,v}(x_{d_i}^-) \quad ,\quad 
\Theta_i:=
\Theta_{f}(x_{d_{i-1}}^+, x_{d_i}^- )\ ,
$$
and let $\DBbb_i$ denote the totally geodesic plane containing the plaque $P_i$. 
Eq.\ \eqref{eqn:second-estimate} follows from simple estimates
in $\HBbb^3$. The idea is that either \emph{both} $\Delta_i$ and
$\Theta_i$
 are close to $0$, close to $\pi$, or the points $x_{d_{i-1}}^+$ and $x_{d_i}^-$ are close to the intersection   $\gamma_T:= \DBbb_{i-1}\cap \DBbb_i$. If the latter holds, then parallel translation of the vector $v$ to the crease of the tent formed by $\DBbb_{i-1}$ and $\DBbb_i$ only changes $v$ by a small amount, and so the difference $\Delta_i$ of angles of the parallel translates is nearly the dihedral angle of the tent.

\medskip
\par\noindent{\bf Step 1.}
Let $y_{i-1}$ be the endpoint of the geodesic $A$ from $x_{d_{i}}^-$ to the plane $\DBbb_{i-1}$. 
Define $a_{n,v}(y_{i-1})$ to be the angle from the projection of $v$ to the normal to $\DBbb_{i-1}$, where the projection is onto the parallel translation of the normal plane to the leaf $g_{i-1}^+$ from $x_{d_{i-1}}^+$ to $y_{i-1}$. By the hypothesis that $g_{i-1}^+$ and $g_i^-$ are $C^1_\varepsilon$ close at their basepoints $x_{d_{i-1}}^+$ and $x_{d_{i}}^-$, we see that $x_{d_{i-1}}^+$ and $y_{i-1}$ are at most a distance $2\varepsilon$ apart; then since $v$ is continuous, we have that $a_{n,v}(y_{i-1})$ and $a_{n,v}(x_{d_{i-1}}^+)$ are equal up to an error comparable to $\epsilon$. 
Let $|A|$ denote the length of $A$, and note that the normal to $\DBbb_{i-1}$ is tangent to the segment $A$ at the point $y_{i-1}$. Let us denote the normals to the planes $\DBbb_{i-1}$ at $y_{i-1}$ and $\DBbb_{i}$ at $x_{d_{i}}^-$ by $n_{i-1}$ and $n_i$, respectively.

\medskip
\par\noindent{\bf Step 2.}  Let $z_{i-1}\in \DBbb_{i-1}$ be the endpoint of the geodesic segment from  $\DBbb_i$ and $\DBbb_{i-1}$, in the case where the planes do not intersect, and when they do intersect $z_{i-1}\in \gamma_T$ is the endpoint of the geodesic from $y_{i-1}$ to $\gamma_T$. In either case, let $B$  the geodesic from $y_{i-1}$ to $z_{i-1}$. 
 The points $\{y_{i-1}, x_{d_{i}}^-, z_{i-1}\}$ give a geodesic triangle in $\HBbb^3$ with sides $A$, $B$, and a third geodesic $C$ from $x_{d_{i}}^-$ to $z_{i-1}$ with length $|C|$. Let $\alpha,\beta,\gamma=\pi/2$ be the corresponding angles of this right-angled geodesic triangle.

\medskip
\par\noindent{\bf Step 3.}  Suppose that
$|\Theta_i|\geq\delta/4N$, and
 $|\Theta_i-\pi|\geq \delta/4N$. 
 By Definition \ref{def:bending-map}, this means in particular that $\DBbb_{i}$ and $\DBbb_{i-1}$ intersect along a geodesic $\gamma_T$. 
 
Now $A$ is orthogonal to $\DBbb_{i-1}$, so its parallel translate along $B$ is orthogonal to $\gamma_T$ at $B \cap \gamma_T$.  As $B$ is also orthogonal there to $\gamma_T$, we see that $\gamma_T$ meets the triangle $ABC$ orthogonally, and hence
$C$ also meets $\gamma_T$ orthogonally. Thus, $\alpha=\Theta_i-\pi$, so the assumption implies $|\alpha|\geq\delta/4N$. 
By the hyperbolic law of sines,
$$
\sinh |B|=\sinh |A|\cdot\frac{\sin\beta}{\sin \alpha} \ ,
$$
which implies that $|B|$ and $|C|$ are of the order of $|A|=O(\varepsilon)$. Thus, $\gamma_T$ is within $O(\varepsilon)$ of the points $y_{i-1}$ and $x^-_{d_i}$.  
Moreover, since the dihedral angle is bounded away from $0$ and $\pi$, $\gamma_T$ must be $C^1_\varepsilon$-close to the leaves $g^+_{i-1}$ and $g^-_i$. In particular, the normal planes to all three are close.
The quantity $\Delta_i$ can then be computed by parallel translation along $B$ and $C$. By Lemma \ref{lem:v-estimate}, it follows that $\Delta_i$ and $\Theta_i$ are close; in particular, less that $\delta/2N$ for $\varepsilon$ sufficiently small.

\medskip
\par\noindent{\bf Step 4.} Suppose that $|\Delta_i|\geq\delta/4N$, and $|\Delta_i-\pi|\geq\delta/4N$. Then $\beta$ is bounded away from $\pi/2$. For a general  right-angled  hyperbolic triangle one has the relation
$$
 \tanh |B|=\sinh|A|\cdot  \tan\beta\ . 
$$ 
 Since $\cos\beta$ is bounded away from zero,   
one arrives at an estimate of the form: $\tanh |B|\leq c_0 \sinh |A|$ for some constant $c_0$ depending on this bound. It again follows that 
$\gamma_T$ is close to the points $x^+_{d_{i-1}}$ to $x^-_{d_i}$, and therefore arguing as in the previous step, $\Delta_i$ and $\Theta_i$ are close.

\medskip
\par\noindent{\bf Step 5.} Suppose that neither of the assumptions of Step 3 or 4 hold. If $|\Delta_i-\pi|< \delta/4N$, for example, then since $v$ is continuous it follows that the orientations of $\DBbb_{i-1}$ and $\DBbb_{i}$ are compatible. Since the assumption of Step 3 fails, this forces $|\Theta_i-\pi|<\delta/4N$. A similar argument holds if $|\Delta_i|< \delta/4N$, and so in either case
$|\Delta_i-\Theta_i|< \delta/2N$. 

\end{proof}

\subsection{Cocycles and Prym differentials} \label{subsec: transverse cocycles and Prym differentials}

In this section we relate the notion of a bending cocycle to the spectral data parametrization of Higgs bundles discussed in \textsection{\ref{sec:spectralcurve}}.  Let $q\in \QD^\ast(X)$, and choose any maximal geodesic lamination $\Lambda$ containing $\Lambda_q^h$ as a sublamination. Thus, if the horizontal foliation of $q$ has no saddle connections, then $\Lambda=\Lambda_q^h$.
  Let $\widehat X_q\to X$ be the spectral curve associated to $q$. Recall that the Prym variety $\Prym(\widehat X_q, X)$ contains $J_2(X)$ as a subgroup. 
The goal is to prove the following

\begin{thm} \label{thm:prym-cocycle}
There is a group isomorphism,  
$$
 {\mathcal H}^o(\Lambda, S^1) \simeq
\Prym(\widehat X_q,X)/J_2(X)\ .
$$
\end{thm}
This result is essentially contained in \cite[\textsection{3.2}]{PenHar}  and \cite[p.\ 13]{Bonahon:shearing}. The idea is to view a bending cocycle in terms of periods of a Prym differential. The choice of sign is fixed by a choice of orientation of the lift of the lamination on the spectral curve.  Since this correspondence is so central to the present paper, we present the details below.

As mentioned in the Introduction, there is a nice
 interpretation of Theorem \ref{thm:prym-cocycle} which goes as follows: the space $\Prym(\widehat X_q,X)/J_2(X)$ is the fiber over $q$ of the Hitchin fibration  for (a component of) the moduli space
of $\PSL(2,\CBbb)$-Higgs bundles, whereas
  ${\mathcal H}^o(\Lambda, S^1)$ is a torus fiber over $\Ccal(\Lambda)$ in Bonahon's parametrization of the character variety $R(\Sigma)$.
 Via the nonabelian Hodge correspondence, $\QD(X)\simeq \Ccal(\Lambda)$, and the moduli space of Higgs bundles is homeomorphic to $R(\Sigma)$.

\subsubsection{Homology of branched covers} \label{sec:branched_homology}
Here we digress to make precise the construction of a homology basis for the spectral curve. 
Consider the general case of a closed, oriented surface
  $\Sigma$. Suppose $p: \widehat \Sigma\to \Sigma$ is a connected branched double cover of $\Sigma$ 
with branching set $B$ and involution $\sigma$, and let $p_{\ast}$ denote the induced map 
$H_1(\widehat \Sigma )\to H_1(\Sigma)$ on homology. Let $g, \widehat g$ be the genera of $\Sigma, \widehat \Sigma$. 
Recall by the Hurwitz formula that $2\widehat g=2g+(2g+\# B-2)$, where we have split the sum to indicate the dimensions of the even and odd homology of $\widehat \Sigma$ under the involution $\sigma$.

\begin{prop} \label{prop:homology}
There is an exact sequence
\begin{equation} \label{eqn:exact-sequence}
0\lra \ZBbb\lra  H_1(\Sigma,B) \stackrel{\phi}{\lra} H_1(\widehat \Sigma) \stackrel{p_{\ast}}{\lra} H_1(\Sigma)\lra 0\ ,
\end{equation}
where
the map $\phi$ is surjective onto the odd homology of $\widehat \Sigma$.
\end{prop}

\begin{proof}
A topological model for the branched cover is given by decomposing $B$ into pairs and introducing branch cuts. In this setting, 
generators of the homology of $\widehat \Sigma$ are given
as follows. First, choose generators $c_1, \ldots, c_{2g}$ for $H_1(\Sigma)$. Let $\hat c_1, \ldots, \hat c_{2g}$ be lifts in $H_1(\widehat \Sigma)$, \emph{i.e.}\ $p_{\ast}(\hat c_k)=c_k$. We set $\phi(c_k)=\hat c_k-\sigma(\hat c_k)$.
Set $N=\# B/2$. Now choose generators $\{a_i, b_j\}$, $i=1,\ldots, N$, $j=1, \ldots, N-1$,  of $H_1(\Sigma, B)$ as in the diagram below.
Define closed curves $\hat a_i$, $\hat b_i$ on $\widehat \Sigma$ as follows:
choose lifts $\hat\alpha_i$, $\hat\beta_i$ of $a_i$, $b_i$, and set
$\hat a_i=\hat\alpha_i-\sigma(\hat \alpha_i)$, $\hat b_i=\hat\beta_i-\sigma(\hat \beta_i)$. Then $\hat a_i=\phi(a_i)$, $\hat b_j=\phi(b_j)$.
With  the orientation indicated, there is a  single relation: $\sum_{i=1}^{ N}\hat a_i=0$. 
The collection $\{\hat a_i, \hat b_j, \hat c_k, \sigma(\hat c_k)\}$ generate $H_1(\widehat \Sigma)$. See Figure \ref{fig:branched}.

\begin{figure} 
\begin{tikzpicture}
\draw  [pattern=north west lines, pattern color=yellow!50] 
(-4,0.75) -- (4,0.75) to
 (4,0.75) -- (4, 1.5) to [out=180, in=180] (4,2.5)
 to (4,2.5) -- (4,3.25) to (4,3.25) -- (-4,3.25) to
 (-4,3.25) -- (-4, 2.5) to (-4,2.5) to [out=0, in=0] (-4,1.5)
 to (-4,1.5) -- (-4,0.75);
 \draw [fill=white, thick] (-1.5,2) ellipse (.75cm and .5cm);
 \draw [fill=white, thick] (1.5,2) ellipse (.75cm and .5cm);
 \draw [thick] (-3.7,2) to [out=35, in=145] (-2.25,2);
 \draw [thick] (2.25,2) to [out=35, in=145] (3.7,2);
 \draw [thick] (-0.75,2) to [out=35, in=145] (0.75,2);
 \draw [thick, dotted] (-3.7,2) to [out=-35, in=-145] (-2.25,2);
 \draw [thick, dotted] (2.25,2) to [out=-35, in=-145] (3.7,2);
 \draw [thick, dotted] (-0.75,2) to [out=-35, in=-145] (0.75,2);
 \draw [pattern=north west lines, pattern color=gray!20] (0,-2) ellipse (4.5cm and 1.2cm);
 \draw [thick, ->] (0,.5) -- (0,-.5);
 \node at (0.5,0) {$\pi$};
 \draw [thick] (-3.7,-2) -- (3.7,-2);
 \node at (-3.7,-2) {$\bullet$};
 \node at (-2.25,-2) {$\bullet$};
 \node at (-0.75,-2) {$\bullet$};
 \node at (0.75,-2) {$\bullet$};
 \node at (3.7,-2) {$\bullet$};
 \node at (2.25,-2) {$\bullet$};
 \node at (-2.9,-2.3) {$a_1$};
 \node at (-1.4,-2.3) {$b_1$};
 \node at ( 0,-2.3) {$a_2$};
 \node at (1.4,-2.3) {$b_2$};
 \node at (2.9,-2.3) {$a_3$};
  \node at (-2.9,1.5) {$\widehat a_1$};
 \node at (-1.5,2.8) {$\widehat b_1$};
 \node at ( 0,1.5) {$\widehat a_2$};
 \node at (1.5,2.8) {$\widehat b_2$};
 \node at (3,1.5) {$\widehat a_3$};
 \end{tikzpicture}
\caption{Branched surface.}
\label{fig:branched}
\end{figure}
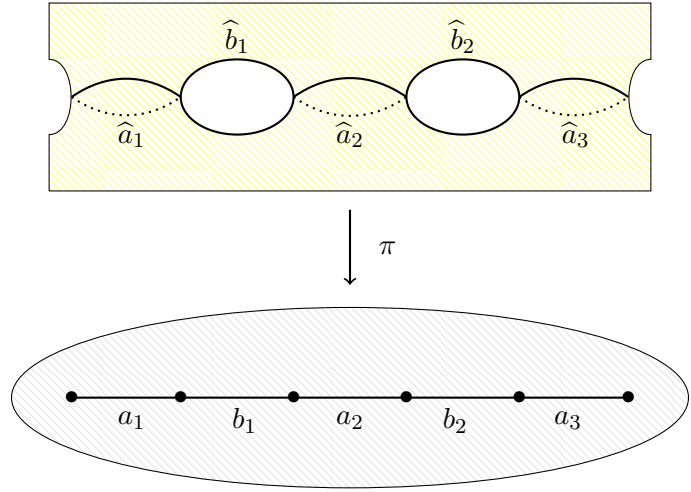

Note that with the appropriate choice of orientations we have the following intersection numbers:
\begin{align*}
\hat a_i \cdot \hat a_j &=\hat b_i\cdot \hat b_j=0 \\
\hat a_i \cdot \hat b_i = +1 \ , \ \hat a_{i+1} \cdot \hat b_i &= -1\ ,\ \hat a_i \cdot \hat b_j = 0\ \text{otherwise}\ .
\end{align*}
These are compatible with the relation.
The intersection numbers of the $\{\hat c_k\}$ and $\{ \sigma(\hat c_k)\}$ are the same as those of $\{c_k\}$ on $\Sigma$, with the additional relations 
\begin{equation} \label{eqn:intersection-numbers}
\hat a_i\cdot \hat c_k=\hat b_j\cdot \hat c_k =\hat c_k\cdot \sigma(\hat c_{\ell})=0\ ,
\end{equation}
for all $i, j, k, \ell$.

By construction, $p_{\ast}\circ\, \phi=0$. We show that $\phi$ is surjective onto the odd homology, which will prove exactness of the second part of the sequence. Indeed, for $\hat\gamma\in H_1(\widehat \Sigma)$, write
$$
\hat\gamma = \sum_{i=1}^N r_i \hat a_i+\sum_{j=1}^{N-1} s_j \hat b_j+\sum_{k=1}^{2g} m_k \hat c_k
+\sum_{k=1}^{2g} n_k \sigma(\hat c_k)\ .
$$
If $\hat\gamma$ is odd then since $\hat a_i$, $\hat b_j$ are also odd we have 
$$
-\hat\gamma=\sigma(\hat \gamma) 
=-\sum_{i=1}^N r_i \hat a_i-\sum_{j=1}^{N-1} s_j \hat b_j+\sum_{k=1}^{2g} m_k \sigma(\hat c_k)
+\sum_{k=1}^{2g} n_k \hat c_k\ .
$$
So $n_k=-m_k$, and 
\begin{align}
\hat\gamma &= \sum_{i=1}^N r_i \hat a_i+\sum_{j=1}^{N-1} s_j \hat b_j+\sum_{k=1}^{2g} m_k(\hat c_k-\sigma(\hat c_k))\label{eqn:gamma-hat} \\ 
&=\phi\left(\sum_{i=1}^N r_i a_i+\sum_{j=1}^{N-1} s_j b_j+\sum_{k=1}^{2g} m_k  c_k\ .
\right)\notag
\end{align}
Let us verify that  $\ker\phi\simeq \ZBbb$. Suppose  $\gamma\in \ker\phi$. We can write
\begin{align*}
\gamma& =\sum_{i=1}^{N} r_i  a_i + \sum_{j=1}^{N-1}s_j  b_j+\sum_{k=1}^{2g} m_k c_k\ ,\\ 
0=\phi(\gamma)& =\sum_{i=1}^{N} r_i \hat a_i + \sum_{j=1}^{N-1}s_j \hat b_j+\sum_{k=1}^{2g} m_k(\hat c_k-\sigma(\hat c_k)) \ .
\end{align*}
Taking intersections with appropriate elements $\hat c_k$ and $\sigma(\hat c_k)$, and using \eqref{eqn:intersection-numbers}, it is easy to see that $m_k=0$ for all $k$. 
Now
$$
0=\phi(\gamma) \cdot \hat a_i =\sum_{j=1}^{N-1}s_j \hat b_j\cdot\hat a_i \ \Rightarrow\ 0=s_{i-1}-s_i\ ,
$$
which implies $s_j=0$ for all $j$. Similarly, 
$$
0=\phi(\gamma)\cdot\hat b_j = \sum_{i=1}^{N} r_i \hat a_i \cdot b_j= r_j-r_{j+1}\ ,
$$
which implies $r_j$ is a fixed constant for all $j$. Hence, $\gamma$ is a multiple of the class $\displaystyle \sum_{i=1}^N a_i$.
This completes the proof of \eqref{eqn:exact-sequence}.
\end{proof}

\begin{remlabel} \label{rem:choice} \upshape
The map $\phi$ is not canonically determined but depends rather on the choices of lifts of the cycles $a_i, b_i, c_i$. 
\end{remlabel}

\subsubsection{Periods of Prym differentials} \label{sec:periods}
 Let us ignore the hyperbolic structure and regard $\Lambda\subset \Sigma$ (recall the discussion at the end of \textsection{\ref{sec:measured-foliations}}). Because $\Lambda$ is maximal, if we choose a collection $B$ of points, one in each component of $\Sigma\setminus\Lambda$, we may define a double cover $\hat\Sigma\to\Sigma$ branched at $B$. The preimage $\hat\Lambda$ is now orientable and we fix such once and for all. Recall the result of the  previous section. In this case, the map $\phi$ in \eqref{eqn:exact-sequence} is actually determined uniquely (see Remark \ref{rem:choice}). Indeed, 
we may assume representatives for the cycles $a_i, b_i, c_i$ are transverse to $\Lambda$. Then choose the lifts to $\widehat \Sigma$ to be positively oriented with respect to the orientation of $\hat\Lambda$. This determines a choice of lifts: the only possible ambiguity would be the existence of cycles not meeting $\Lambda$, but this is ruled out by maximality.

 Notice that we have an identification of the $\pi_1$-orbits of plaques with the set $B$;  let us denote this $\pi_1$-invariant map $p:\Pcal(\Lambda)\to B$.
 Let $\tilde \gamma$ be a transverse $C^1$ curve from  plaques $ P$ and $Q$ in $\widetilde \Sigma$, and let $\gamma$ be the projection from corresponding points $p$ to $q$ in $B$.
 As in the previous paragraph, there is a unique lift $\widehat \gamma$ that is positively oriented transverse to $\widehat \Lambda$, \emph{i.e.}\ the lamination is oriented to the left at a point of intersection of $\widehat \gamma\cap \widehat\Lambda$.
Now given a closed  $1$-form $\hat\alpha$, $[\hat\alpha]\in H^1_{\odd}(\widehat \Sigma, \RBbb)$, define
 \begin{equation} \label{eqn:def-cocycle}
 \sigmabold_{\hat\alpha}( P,  Q) := 2 \int_{\hat\gamma}\hat\alpha\ .
 \end{equation}
 The factor of $2$ is added here for convenience (see Remark \ref{rem:period} below).

 We first 
note that $\sigmabold_{\hat\alpha}$ is well defined. First, it is independent of the choice of $\widetilde \gamma$; for a relative homotopy of $\widetilde\gamma$ induces one on $\gamma$, and therefore $\widehat\gamma$, and this does not affect the integral of the closed form $\widehat\alpha$. Second, it is independent of choice of representative $\widehat \alpha$. Any other choice can be written as $\widehat \alpha+ df$, for an odd $\RBbb$-valued function $f$, and since the endpoints of $\widehat \gamma$ lie on the fixed point set of $\sigma$, this contributes nothing to the integral. 

With this understood, we prove the following:

 \begin{prop} \label{prop:cocycle}
 The function $\sigmabold_{\hat\alpha}$ in \eqref{eqn:def-cocycle} defines a transverse cocycle depending only on the class of $\hat\alpha$.
 \end{prop}

 \begin{proof}
 Equivariance is clear, since the path $\tilde \gamma$ and $g\tilde\gamma$ define the same curve $\gamma$ on $X$. If the plaque $ R$ separates $ P$ and $ Q$, let  $\tilde \gamma: [a,b]\to {\mathbb H}$, $\tilde\gamma(a)= P$, $\tilde\gamma(b)= Q$.  Then there exist  $a<t_1<t_2<b$ such that 
 $\tilde \gamma(t_1), \tilde\gamma(t_2)\in \partial R$, and 
 $\tilde\gamma[a,t_1)\cap  R=\tilde\gamma(t_2,b]\cap  R=\emptyset$.  After a homotopy, we may assume $\tilde\gamma(t_1,t_2)\subset R$, and after a further homotopy we may assume there is $t_1<c<t_2$ such that $\tilde\gamma(c)$ is the point in $B$ associated to $ R$. It follows that $\tilde\gamma$
  can be written as a sum of  quasitransverse paths from $ P$ to $ R$, and $ R$ to $ Q$. 
 The additivity then follows from the additivity of the integral in \eqref{eqn:def-cocycle}.  
 It remains to prove symmetry. Let $\overline \gamma$ denote the curve $\gamma$ with the reverse orientation.
 Then we  note the following.
 \begin{equation} \label{eqn:orientation}
 \widehat{\overline\gamma}=\overline{\sigma(\hat\gamma)}
 \end{equation}
 Indeed, 
 $\sigma(\hat\gamma)$ is negatively oriented with respect to $\widehat\Lambda$, and so
  both sides of \eqref{eqn:orientation} are positively oriented lifts of $\overline\gamma$.
 Using \eqref{eqn:orientation}, we have
 \begin{equation} \label{eqn:reverse}
 \frac{1}{2}\sigmabold_{\hat\alpha}( Q,  P)=\int_{\widehat{\overline\gamma}}\hat\alpha =\int_{\overline{\sigma(\hat\gamma)}}\hat\alpha 
 =-\int_{\sigma(\hat\gamma)}\hat\alpha=-\int_{\hat\gamma}\sigma^{\ast}\hat\alpha 
 =\int_{\hat\gamma}\hat\alpha=\frac{1}{2}\sigmabold_{\hat\alpha}( P,  Q)\ .
\end{equation}
 This completes the proof.
 \end{proof}
 \begin{remlabel} \label{rem:period} \upshape
 Note that $\sigmabold_{\hat\alpha}( P,  Q)$ is equal to a period of the differential $\hat\alpha$.  Indeed, from the discussion above, the curve $\hat\gamma_{ P Q}=\hat\gamma\cup \widehat{\overline\gamma}$ is a closed oriented curve on $\widehat X_q$, and by \eqref{eqn:reverse},
 $$
 2\sigmabold_{\hat\alpha}( P,  Q)=\sigmabold_{\hat\alpha}( P,  Q)+\sigmabold_{\hat\alpha}( Q,  P)
 =2\int_{\hat\gamma}\hat\alpha+2\int_{\widehat{\overline\gamma}}\hat\alpha=2\int_{\hat\gamma_{ P Q}}\hat\alpha\ .
 $$
  Conversely, by Proposition \ref{prop:homology}, every element of $H^{\odd}_1(\widehat \Sigma, {\mathbb Z})$ is represented by a linear combination of
   oriented curves of the form $\hat\gamma_{ P Q}$, for some lifts $ P$, $ Q$ of some points $p,q\in B$.  It follows that the periods of $\hat\alpha$, and hence $[\hat\alpha]$ itself,  are determined by $\sigmabold_{\hat\alpha}$.
 \end{remlabel}

 We now return to the case where $\Sigma$ has a Riemann surface structure $X$ and the lamination comes from a holomorphic quadratic differential.
 \begin{exlabel} \label{ex:SW-cocycle} \upshape
Let $q\in \QD^\ast(X)$, and let $\Lambda$ be a maximalization of $\Lambda_q^h$ in the sense of Lemma \ref{lem:maximalization}. 
The Seiberg-Witten differential $\lambda_{\SW}$ from \eqref{eqn:SW-differential} is a holomorphic Prym differential on $\widehat X_q$.  We can orient the lift $\widehat\Lambda_q^h$ by the condition $\real\lambda_{\SW}>0$. 
The harmonic Prym differential $\real\lambda_{\SW}$ defines a canonical transverse cocycle $\sigmabold^{can}_{q}$. 
 By the previous remark, $\sigmabold^{can}_{q}$ is determined by the real parts of the periods of \eqref{eqn:central_charge}.
 \end{exlabel}
 
By Proposition \ref{prop:cocycle}, we have a map 
\begin{equation} \label{eqn:beta1}
T: H^1_{\odd}(\widehat X_q,{\mathbb R})\longrightarrow {\mathcal H}(\Lambda, {\mathbb R}) \ ,\
[\hat\alpha]\mapsto \sigmabold_{\hat\alpha}\ .
\end{equation}
We can do a similar construction for bending cocycles. If $[\hat\eta]\in H^1_{\odd}(\widehat X_q, i\RBbb)$, set 
\begin{equation} \label{eqn:def-bending-cocycle}
\betabold_{\hat\eta}( P,  Q) := -2i\int_{\hat\gamma}\hat\eta\quad\mod 2\pi\ .
\end{equation}
By Remark \ref{rem:period}, $\betabold_{\hat\eta}$ only depends on the class of $[\hat\eta]$ modulo the lattice
\begin{equation}\label{eqn:Lambda}
L=L(\widehat X_q):=\Hom\left(H_1^{\odd}(\widehat X_q,\ZBbb), 2\pi i\ZBbb\right)\ .
\end{equation}
Hence, we have a map
\begin{equation} \label{eqn:beta2}
B: H^1_{\odd}(\widehat X_q,i{\mathbb R})/L\longrightarrow {\mathcal H}^o(\Lambda, \RBbb/2\pi\ZBbb) \ ,\
[\hat\eta]\mapsto \betabold_{\hat\eta}\ .
\end{equation}

Clearly, $T$ is linear. By Remark \ref{rem:period}, it is also injective. For if $\sigmabold_{\hat \eta}\equiv 0$, then the periods of $\hat\eta$ must all vanish; hence, $[\hat\eta]=0$.   By \cite[Prop.\ 1]{Bonahon:shearing}, the dimensions of the two sides of \eqref{eqn:beta1} agree. 
In the case of the map $B$, notice that the lattices on either side are isomorphic under the map $T$. This proves the following result.

\begin{cor} \label{cor:prym}
The maps $T$ and $B$ in \eqref{eqn:beta1}, \eqref{eqn:beta2} are  isomorphisms.
\end{cor}

 We observe the following: 
\begin{lem} \label{lem:lattices}
The inclusion induces an exact sequence:
$$
0\lra H^1_{\odd}(\widehat X_q, 2\pi i\ZBbb)\lra L\lra J_2(X)\lra 0\ .
$$
\end{lem}

\begin{proof}
This is most easily seen in terms of the explicit generators in \textsection{\ref{sec:branched_homology}}. Let $\hat\alpha_j$, $\hat\beta_j$, and $\hat\delta_j$ be Poincar\'e duals of $\hat a_j$, $\hat b_j$, and $\hat c_j$, respectively. Then $\delta_j:=\pi i(\hat\delta_j-\sigma^\ast\hat\delta_j)\in \Lambda$, $j=1,\ldots, 2g$, is not $2\pi i$-integral, but  $2\delta_j\in H^1_{\odd}(\widehat X_q, 2\pi i\ZBbb)$. It is easily seen that $L$ is generated by  $H^1_{\odd}(\widehat X_q, 2\pi i\ZBbb)$ and all such elements $\delta_j$.   This proves the result.
\end{proof}

\begin{proof}[Proof of Theorem \ref{thm:prym-cocycle}]
Immediate from Corollary \ref{cor:prym},  Lemma \ref{lem:lattices}, and eq.\ \eqref{eqn:prym-prym}.
\end{proof}

\begin{definition} \label{def:gamma-SW}
The complex cocycle $\Gamma_{\SW}\in \Hcal^o(\Lambda, \RBbb+i\RBbb/2\pi\ZBbb)$ is the one determined as in Corollary \ref{cor:prym} by the periods of the real and imaginary parts of $\lambda_{\SW}$.
\end{definition}

\subsection{Approximation by pleated surfaces}

  In this section, we define what it means for a family of harmonic maps and a family of pleated surfaces to be asymptotic, and we relate the corresponding notions of bending.
  The intuition behind the definition below may be summarized as follows: the image of a pleated surface $f:\widetilde S\to \HBbb^3$, for a representation outside a large compact set, consists of a configuration of plaques sheared far apart from one another and related by long leaves of the lamination. At the same time, the image of a harmonic map in a neighborhood of the zeroes of a quadratic differential is nearly planar, whereas leaves of the horizontal foliation are nearly geodesic. The approximation requires these planes and approximate geodesics to be close to the plaques and leaves of the lamination of the pleated surface.

We furthermore assume that we have chosen maximal laminations 
$\Lambda_n$ (\emph{resp}.\ $\Lambda$) containing $\Lambda_{q_n}^h$
 (\emph{resp}.\ $\Lambda_q^h$), and that $\Lambda_n\to \Lambda$ 
in the Hausdorff sense. 
\medskip\\
If there exists a pleated surface 
$\sfP_n = (S_{n},f_{n},\Lambda_{n},\rho_{n})$, then 
by the discussion in \textsection{\ref{sec:laminations}} there is a bijective correspondence between the zeroes $\widetilde Z(q)$ and plaques of $\widetilde \Lambda_n$,
 and each bi-infinite leaf in $\Fcal_{q_n}^h$ determines a
 leaf in $\Lambda_{q_n}^h$. The choice of maximalization
 $\Lambda_n$ is determined by a finite choice
 of ``additional'' leaves (see \textsection{\ref{sec:maximalization}}). We also recall from \textsection \ref{subsect:highenergyharmonicnearzeroes}  the definition of the hexagonal sets $\Qcal_n$ and $\widetilde \Qcal_n$. With this understood, we make the following

\begin{definition} \label{def:asymptotic-surfaces}\upshape
A sequence of pleated surfaces 
$\sfP_n = (S_{n},f_{n},\Lambda_{n},\rho_{n})$
is {\bf asymptotic to $u_n$} if for any $\varepsilon>0$ there is $N$ so that if $n\geq N$ the following holds.  
\begin{compactenum}[(i)]
\item The image by $u_n$ of the horizontal leaves in $\widetilde X\setminus \widetilde \Qcal_n$ are $C^1_\varepsilon$-close to the corresponding leaves in $\widetilde \Lambda_n$;
\item if $\widetilde p\in \widetilde Z(q_n)$, then
$u_n(\widetilde p)$ is $\varepsilon$-close to the image by $f_{n}$ 
of the corresponding plaque $P$ in $\widetilde S_n$. Moreover,  the
parallel translation of the tangent  plane to the image of $u_n$
at 
 $ u_n(\widetilde p)$ along the geodesic to $f_n(P)$
 makes an angle less than $\varepsilon$ with the totally
geodesic subspace containing  $f_n(P)$.
\end{compactenum}
\end{definition}

With this definition we are  in a position to compare the notion of bending for sequences of harmonic maps and of pleated surfaces that are asymptotic to each other. 

\begin{prop}
 \label{prop: bending for asymptotic maps and pleated surfaces} 
Let $\sfP_n = (S_{n},f_{n},\Lambda_{n},\rho_{n})$ be a sequence of 
pleated surfaces that is asymptotic to the sequence 
$u_{n}:\widetilde X\to \HBbb^3$ of $\rho_n$-equivariant harmonic maps 
in the sense of Definition \ref{def:asymptotic-surfaces}. Denote 
by $\betabold_n\in \Hcal^o(\Lambda_n,S^1)$ the bending 
cocycles of $\sfP_n$. 
Fix $\delta>0$.
Then for any  $P,Q\in\Pcal(\Lambda)$, there are 
plaques $\{P_i\}_{i=0}^N$ between $P$ and $Q$, $P_0=P$, $P_N=Q$, 
with centers $\widetilde p_i$, such that
$$
\lim_{n\to \infty} \Bigl(\betabold_n(r^\Lambda_{\Lambda_n}(P),
r^\Lambda_{\Lambda_n}(Q))-\sum_{i=1}^N\Theta_{u_n}
(\widetilde p_{i-1},\widetilde p_i)\Bigr)\leq \delta\ .
$$

\begin{proof}
We shall use the set up of Lemma 
\ref{lem:finite-approximation-bending}. Note that by 
the convergence $\Lambda_n\to \Lambda$, the approximation of the 
bending cocycle by sums over finitely many plaques is uniform. By 
a further subdivision, we may assume that between any 
two centers $\widetilde p_{i-1}^{(n)}\to \widetilde p_{i-1}$ and 
$\widetilde p_i^{(n)}\to \widetilde p_i$ there are quasitransverse 
 arcs (or modified saddle connections) 
$k_i^{(n)}$ with small vertical ends that meet the zeroes of the 
Hopf differentials $q_n$ only at $\widetilde p_{i-1}^{(n)}$ 
and $\widetilde p_i^{(n)}$. Then the images by $u_n$ of the horizontal 
parts of $k_i^{(n)}$ are
 $C^1_\varepsilon$-close, 
and therefore by the asymptotic 
assumption the same is true for the 
leaves of $\Lambda_n$ along $k_i^{(n)}$.
   Thus, the hypotheses of Lemma
 \ref{lem:finite-approximation-bending} are satisfied 
for sufficiently large $n$, and we have 
\begin{equation} \label{eqn:pleated-surface-estimate}
\Bigl|\betabold_n(r^\Lambda_{\Lambda_n}(P),r^\Lambda_{\Lambda_n}(Q))-\sum_{i=1}^N\Theta_{f_n}(\widetilde p_{i-1}^{(n)},\widetilde p_i^{(n)})\Bigr|<\delta/2\ ,
\end{equation}
for large enough $n$. 
On the other hand, 
an argument analogous to the one used in the proof of that lemma shows that
\begin{equation} \label{eqn:bending-estimate}
\left| \Theta_{f_n}(\widetilde p_{i-1}^{(n)}, \widetilde p_i^{(n)})-
\Theta_{u_n}(\widetilde p_{i-1}^{(n)}, \widetilde p_i^{(n)}) \right| \leq \delta/2N\ ,
\end{equation}
for large $n$.
Indeed,  
suppose that not both 
$\Theta_{f_n}(\widetilde p_{i-1}^{(n)}, \widetilde p_i^{(n)})$
and 
$\Theta_{u_n}(\widetilde p_{i-1}^{(n)},
\widetilde p_i^{(n)})$
 are within $\delta/4N$ of $\pi$, and neither
are  they both within $\delta/4N$ of $0$.
Then the image by  $u_n$  of the horizontal parts of $k_i^{(n)}$
is arbitrarily close to the crease of the tent formed by the
totally geodesic planes associated to the plaques $P_{i-1}$ and
$P_i$. By Definition
 \ref{def:asymptotic-surfaces} (ii), these are also
close to the planes tangent to the image of $u_n$ at  $p_{i-1}$
and $p_i$. 
The angle
$\Theta_{f_n}(\widetilde p_{i-1}^{(n)}, \widetilde p_i^{(n)})$
  can be computed by parallel
translation of the normal vectors to the plaques, as discussed
after Definition \ref{def:tent}. These normal vectors are close
to the normal vectors to the planes defined by $u_n$. By Lemma
\ref{lem:v-estimate}, the parallel translations along
$u_n(k_i^{(n)})$  are also close to the parallel translations
along the crease.
Note that in the statement of that lemma, the term
$L\varepsilon$ is small, since $\varepsilon$ is exponentially
small compared to the  length of $u_n(k_i^{(n)})$ 
by Proposition \ref{prop:pullbackmetric}.

Combining \eqref{eqn:bending-estimate} 
with \eqref{eqn:pleated-surface-estimate}, 
$$
\Bigl|\betabold_n(r^\Lambda_{\Lambda_n}(P),r^\Lambda_{\Lambda_n}(Q))-\sum_{i=1}^N\Theta_{u_n}(\widetilde p_{i-1}^{(n)},\widetilde p_i^{(n)})\Bigr|<\delta\ .
$$
Since this holds for fixed $N$ and $\delta$,  and any sufficiently
large $n$, this  completes the proof.
\end{proof}
\end{prop}

\subsubsection{Bending cocycles and periods}\label{sec: two bendings}

We now combine the considerations above with the results of \textsection{\ref{sec:bending}}. 

\begin{thm} \label{thm:bending-periods}
 Let $\sfP_n = (S_{n},f_{n},\Lambda_{n},\rho_{n})$ be a sequence of pleated surfaces with bending cocycles $\betabold_n$. 
 We assume the following two conditions:
 \begin{compactenum}[(i)]
 	\item The sequence $\sfP_n$
 is asymptotic to the sequence $u_{n}:\widetilde X\to \HBbb^3$ of $\rho_n$-equivariant harmonic maps in the sense of Definition \ref{def:asymptotic-surfaces};
\item the sequence $(A_n,\Psi_n)$ of Higgs pairs for $\rho_n$
converges to a limiting configuration with associated Prym differential $\widehat \eta$.
\end{compactenum}
Let $\betabold_{\widehat \eta}$ be defined as in \eqref{eqn:def-bending-cocycle}. Then
in $\Hcal^o(\Lambda, S^1)$: 
$\displaystyle
\lim_{n\to \infty}\betabold_n=\betabold_{\widehat \eta}
$ . 
\end{thm}

\begin{proof}
Let $\betabold$ be any subsequential limit of $\betabold_n$. By Proposition \ref{prop: bending for asymptotic maps and pleated surfaces} it suffices to estimate the geometric bending $\Theta_{u_n}(\widetilde p_{i-1}, \widetilde p_i)$. 
Using the assumption in the proof of that result, we have a quasitransverse path $k_i$ from $p_{i-1}$ to $p_i$ that intersects the zeroes of $q$ only at the endpoints. 
By Theorem \ref{thm:asymptotic-bending}, it follows that 
$\betabold_n(P,Q)$ is approximated by the sum of $\Theta_{k_i}(A_n, \Psi_n)$. Since the latter is additive, $\betabold_n(P,Q)$ is approximated by $\Theta_k(A_n,\Psi_n)$, where $k$ is the image of a path from $\widetilde p$ to $\widetilde q$. 
By Proposition \ref{prop:LimitBendingOfPairs}, this converges as $n\to \infty$ to the period of $\widehat \eta$. 
\end{proof}

\section{Realization of pleated surfaces} \label{sec:realization}
The goal of this section is to prove the following result, which is part (i) of the Main Theorem.

\begin{thm} \label{thm:bonahon-image}
Let $[\rho_n]\in R^o(\Sigma)$ be a divergent sequence, $u_n:\widetilde X\to \HBbb^3$ the $\rho_n$-equivariant harmonic maps, and $t_n^2q_n$ the Hopf differentials of $u_n$, where $q_n\in \SQD^\ast(X)$, $t_n\to +\infty$. We assume $q_n\to q\in \SQD^\ast(X)$, and in some $($hence any$)$ realization of the associated geodesic laminations,   choose maximal laminations $\Lambda_n$ $($\emph{resp}.\ $\Lambda$$)$ containing $\Lambda_{q_n}^h$ $($\emph{resp}.\ $\Lambda_q^h$$)$, with $\Lambda_n\to \Lambda$. 
Then there is $N$ such that for all $n\geq N$, the class $[\rho_n]$ is in the image of the map $B_{\Lambda_n}$ in \eqref{eqn:bonahon}, \emph{i.e.}\ there is a pleated surface
$\sfP_n = (S_{n},f_{n},\Lambda_{n},\rho_{n})$. 
\end{thm}

\subsection{Realizing laminations}\label{subsect:realizinglaminations}

\begin{definition}[{\emph{cf}.\ \cite[Def.\ I.5.3.4]{CEG}}]\upshape
Suppose $\Lambda\subset S$ is a geodesic lamination and  $\rho:\pi_1(S)\to \PSL(2,\CBbb)$. Then $\Lambda$ is {\bf realizable} if there exists a continuous $\rho$-equivariant map $ \varphi:\widetilde S\to \HBbb^3$ that takes the leaves of $\widetilde \Lambda$ homeomorphically onto geodesics in $\HBbb^3$. 
\end{definition}

The goal of this subsection is to prove the following. 

\begin{prop} \label{prop:realization}
Assume the hypotheses of Theorem \ref{thm:bonahon-image}. Then for $n$ sufficiently large, there is a marked hyperbolic surface $\widehat S_n$ such that the geodesic lamination  $\Lambda_n\subset \widehat S_n$ is realizable for $\rho_n$. 
\end{prop}

The proof of Proposition \ref{prop:realization} will proceed by
using the result of  Minsky  on high energy harmonic maps into  $\H^3$ that we summarized in \textsection \ref{subsect:Minskysresults}.
 This will lead to  a suitable collection of train tracks carrying the laminations $\Lambda_n$.

\subsubsection{The companion surface} \label{sec:companion}
By \cite[Thm.\ 3.1]{Wolf:thesis} and \cite[Thm.\ 11.2]{Hitchin:87}
there is a marked hyperbolic surface $\widehat S_n$ such that the harmonic diffeomorphism $v_n:X\to \widehat S_n$ has Hopf differential $t_n^2q_n$. Moreover, the class $[\widehat S_n]\in T(\Sigma)$ is uniquely determined by $t_n^2q_n$. We let $\widetilde v_n : \widetilde X\to \HBbb^2$ denote the lift to the universal cover. Via $v_n$, the laminations $\Lambda_n$ are realized as geodesic laminations in $\widehat S_n$. As previously, we continue to use the notation $\Lambda_n\subset \widehat S_n$ to simplify the notation. Also, denote the lift by $\widetilde \Lambda_n\subset\HBbb^2$.

\begin{prop}[{\cite[Thm.\ 7.1]{Minsky:92a}}] \label{prop:realizing-2d}
For $A, c_0, C_0$ as in Proposition \ref{prop:realizing-3d} and $n\geq N$ and $s_n\leq c_0$, there is a $\pi_1$-equivariant map $\Pi_\ast$  from the leaves of $\widetilde\Fcal_{q_n}^h$ in the complement of $\Qcal_n$  to the leaves of $\widetilde\Lambda_n^{h}\subset\HBbb^2$ which factors through $v_n$. Moreover, for any $p\in \widetilde X\setminus\widetilde\Qcal_n$, 
$$
d_{\HBbb^3}(v_n(p), \Pi_\ast(p))\leq A\exp(-t_nC_0)\ ,
$$
and the derivative along the horizontal leaf through $p$ $($in
 the $|q_n|$ metric$)$ is
$$
\left| |d\Pi_\ast|-2\right| \leq A\exp(-t_nC_0)\ .
$$
\end{prop}
We note that in the case of a maximalization $\Lambda_n$ of $\Lambda_n^h$, $\Pi_\ast$ can be extended to the additional leaves as  remarked in \textsection{\ref{subsect:Minskysresults}}.

The train track used in the proof of Proposition \ref{prop:realizing-3d} may be chosen so that for  $n$ sufficiently large the following holds.
\begin{compactenum}[(i)]
\item Let  $\widetilde\tau_{n,\ast}=\widetilde v_n(\widetilde \tau_n)\subset\HBbb^2$. Then the branches of $\widetilde\tau_{n,\ast}$ have length comparable to $t_n$ and geodesic curvature $O(\varepsilon_n)$ and meet tangentially.
\item The collection $\widetilde \Lambda_n^h$ is $C^1_{\varepsilon_n}$-carried by $\widetilde \tau_{n,\ast}$.
\end{compactenum}

Let $\widehat\sigmabold_n$ denote the shearing cocycle of $\widehat S_n$ with respect to the lamination $\Lambda_n$. 
We will need the following result from \cite{Wolf:thesis}. (Stronger estimates are implicit in \cite{DumasWolf15}.)
\begin{lem} \label{lem:wolf}
For any $\delta>0$ there is $N$ such that for all $n\geq N$,
$$
\Vert\widehat\sigmabold_n-t_n\sigmabold^{can}_{q_n}\Vert<\delta t_n\ .
$$
\end{lem}

\subsubsection{Proof of Proposition \ref{prop:realization}} \label{sec:train-track-conclusion}

We first choose a constant $\delta>0$ so that there are disjoint  arcs $c_i$, one for each branch $b_i$ of $\tau_n$, so that $c_i$ intersects only $b_i$, and this only once. The endpoints of $c_i$ lie in exactly two components of $X\setminus \tau_n$ (see Figure \ref{fig:realization}). Viewed on $\widehat S_n$, we may assume that the endpoints of $c_i$ are in $\widehat S_n\setminus N_{2\delta}$, where $N_r$ is the $r$ neighborhood of $\tau_{n,\ast}$. We furthermore assume $\Lambda_n\subset N_\delta$. These assumptions are made possible by Proposition \ref{prop:realizing-2d}.

Let $g$ be a leaf of $\widetilde\Lambda_n\subset \HBbb^2$. Then $u_n\circ \widetilde{v_n}^{-1}$ produces a well-defined geodesic $g^\ast\subset\HBbb^3$. Namely, if $\ell\subset\widetilde \Fcal_{q_n}^h$ follows a train path such that the straightening of $\widetilde{v_n}(\ell)$ is $g$, then $g^\ast$ is the straightening of $u_n(\ell)$.
For every intersection point $p$ in $ g\cap \widetilde c_i$, we map $p$ to the point $\varphi(p)=p^\ast$ given by the nearest point projection of $u_n\circ \widetilde{v_n}^{-1}(p)$ onto $g^\ast$. Let $p_1$ and $p_2$ be consecutive points on $ g$, in the sense that there is no   other point in $ g\cap \widetilde c_i$ in the geodesic segment $ g_{p_1p_2}$.   Extend the map along the segment
$$\varphi: g_{p_1p_2}\isorightarrow  g^\ast_{p_1^\ast p_2^\ast}
$$
as a homothety. Continuing in this way we obtain a continuous map $\varphi: \widetilde\Lambda_n\to \HBbb^3$ mapping leaves homeomorphically to geodesics. Moreover, it is clearly equivariant. 
Since $\widetilde \Lambda_n$ is a closed subset, by the Tietze extension theorem we can extend $\varphi$ to a continuous map $\widetilde{\overline N}_\delta$ with the same Lipschitz constant. Now we use a geodesic homotopy to join $u_n\circ \widetilde{v_n}^{-1}$ on the complement of $\widetilde N_{2\delta}$ to this extension. This defines the map $\varphi$, and it is equivariant. 

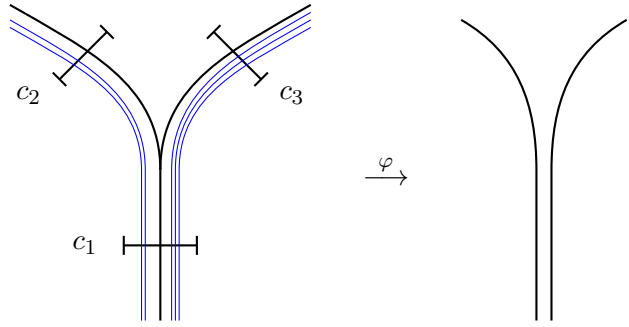
\begin{figure} 
\begin{tikzpicture}
\draw [thick] (0,0) -- (0,2);
\draw [thick] (0,2) to [out=90, in=-30] (-2,4.2);
\draw [thick] (0,2) to [out=90,in=210] (2,4.2);
\draw [thick, |-|] (-0.5,1) -- (0.5,1);
\draw [thick, |-|] (-1.35,3.2) -- (-.7,3.85);
\draw [thick, |-|] (1.35,3.2) -- (.7,3.85);
\draw [color=blue] (-.2,0) -- (-.2,2) to [out=90, in=-30] (-2,4);
\draw [color=blue] (-.25,0) -- (-.25,2) to [out=90, in=-30] (-2,3.9);
\draw [color=blue] (.15,0) -- (.15,2) to [out=90, in=210] (2,4.1);
\draw [color=blue] (.2,0) -- (.2,2) to [out=90, in=210] (2,4);
\draw [color=blue] (.25,0) -- (.25,2) to [out=90, in=210] (2,3.9);
\node at (-1,1) {$c_1$};
\node at (-1.75,3) {$c_2$};
\node at (1.75,3) {$c_3$};
\node at (3,2) {$\stackrel{\varphi}{\lra}$};
\draw [thick] (5,0) -- (5,2) to [out=90, in=-30] (4,4);
\draw [thick] (5.2,0) -- (5.2,2) to [out=90, in=210] (6.2,4);
\end{tikzpicture}
\caption{Realization of $\Lambda_n$.}
\label{fig:realization}
\end{figure}

\subsection{Perturbing the companion surface}
For closed $3$-manifolds the existence of a realization of a lamination leads to a pleated surface (see \cite[Thm.\ I.5.3.9]{CEG}). The goal of this section is to prove the same in the equivariant case that we consider.  The rough idea is that the companion surfaces $\widehat S_n$ obey the same asymptotics as the image of the equivariant harmonic maps $u_n$, so the hyperbolic structure on the putative pleated surface should be obtained from a small perturbation of that on $\widehat S_n$. For a similar construction, see \cite{Bonahon:shearing}, proof of Lemma 30. 

\subsubsection{The shearing cocycle from the realization} \label{sec:shearing_realization}
We first describe a shearing cocycle associated to the realization of $\Lambda_n$ obtained in Proposition \ref{prop:realization}. 
In order to do this, recall the notation of \textsection{\ref{subsec:ShearingCocycles}}.

Let $ \varphi_n: \widetilde{\widehat S_n}\to \HBbb$ be a realization of $\Lambda_n$. 
Let $k$ be a transverse path to $\Lambda_n$ from plaque $P$ to plaque $Q$, and let $d$ be a component of $k\setminus \Lambda_n$.
Then $d$ corresponds to a plaque $R\in \Pcal(\Lambda_n)$, and therefore an ideal triangle (also denoted $R$) in $\HBbb^2$.
(Here $R$ depends on $n$, but in this passage, the index $n$ will not vary, so we suppress the notational dependence.)
Recall the lamination $\Lambda_n^\ast$ constructed in Proposition~\ref{prop:realizing-3d}.
 We first observe that under the correspondence between leaves of $\Lambda_n$ and $\Lambda_n^\ast$, the geodesics in $\Lambda_n^\ast$ associated to the edges of $R$ form an ideal triangle $R^\ast\subset\HBbb^3$.
This is because first,  two geodesics in $\Lambda_n^\ast$ corresponding to a pair of edges of $R$ must be asymptotic on one end, since the map $u_n$ is Lipschitz. Second, if the edges of $R$ do not form a triangle in $\Lambda_n^\ast$,
  then two such geodesics would collapse, and this is ruled out (for sufficiently large $n$) by Proposition \ref{prop: harmonic map C1 localizes near zeroes}. 

With this understood,
set $x_d^{\pm,\ast}=\varphi(x_d^\pm)$. 
For each $d$, 
let $h_n: g_d^{\pm}\to\RBbb$ denote the signed distance to the foot of the geodesic, as described in \textsection{\ref{subsec:ShearingCocycles}. 
Similarly, define  $h_n^\ast: g_d^{\pm,\ast}\to\RBbb$ for the corresponding ideal triangles in $\HBbb^3$. We then define:

\begin{equation} \label{eqn:realization-transverse-cocycle}
\sigmabold_n(P,Q):=
h_n^\ast(x_{d^-}^{+,\ast})-h_n^\ast(x_{d^+}^{-,\ast})
+\sum_{d\neq d^+, d^-} (h_n^\ast(x_{d}^{+,\ast})-h_n^\ast(x_{d}^{-,\ast}))
\end{equation}

\begin{lem}
Eq.\ \eqref{eqn:realization-transverse-cocycle} defines a transverse cocycle $\sigmabold_n\in \Hcal(\Lambda_n, \RBbb)$. 
\end{lem}

\begin{proof}
For each $d$,
the quantity $| h_n^\ast(x_{d}^{+,\ast})-h_n^\ast(x_{d}^{-,\ast})|$ may be bounded by the distance from 
$x_{d}^{+,\ast}$ to $x_{d}^{-,\ast}$
(see \cite[proof of Lemma 8]{Bonahon:shearing}).
The map $\varphi$ is Lipschitz with constant $M_n$, say, so 
\begin{equation} \label{eqn:hast}
\left| h_n^\ast(x_{d}^{+,\ast})-h_n^\ast(x_{d}^{-,\ast})\right|\leq M_n\ell(d)\ ,
\end{equation}
where $\ell(d)$ is the hyperbolic length of $d$.
By the estimate in \cite[Lemma 5]{Bonahon:shearing}, the sum 
 in \eqref{eqn:realization-transverse-cocycle} converges, and 
 $\sigmabold_n$ is therefore well defined. 
 The symmetry and  additivity conditions of \textsection{\ref{subsec: transverse cocycles}} are clear. 
\end{proof}

We shall require a more precise relationship between $\sigmabold_n$ and $\widehat\sigmabold_n$. 

\begin{lem} \label{lem:norm}
Fix a finite set $\Pcal'\subset\Pcal(\Lambda)$ and $\delta>0$. Then
there is $N$ such that for all $n\geq N$ and 
 all $P,Q\in \Pcal'$, 
 $$
\left| \sigmabold_n(P,Q)-\widehat\sigmabold_n(P,Q)\right|\leq \delta t_n\ .
 $$
\end{lem}

\begin{proof}
Let $A_n$ be the constant defined in \cite[Lemma 3]{Bonahon:shearing} for the hyperbolic structure $\widehat S_n$ (this depends on $P$ and $Q$), which gives a lower bound on the length of a leaf in $\Lambda_n$ that intersects a transverse arc from $P$ to $Q$ multiple times. 
Since the laminations $\Lambda_n$ converge, and since the $\widehat S_n$-length of a leaf of $\Lambda_n$ is
 stretched by a factor of $t_n$ (see Proposition \ref{prop:pullbackmetric}),  it follows that there is a constant $A_0>0$ such that $A_n\geq A_0 t_n$ for all $n$ and all choices of pairs in $\Pcal'$. 

Appealing again to \cite[Lemma 5]{Bonahon:shearing}, we have
\begin{equation}\label{eqn:ell-estimate}
\ell(d)\leq B\exp(-t_nA_0r(d))\ ,
\end{equation}
 where $r(d)$ is the divergence radius of $d$ and $B$ is independent of $n$. Using \eqref{eqn:hast} and \eqref{eqn:ell-estimate}, then 
  as in the proof of Lemma \ref{lem:finite-approximation-bending}, there is a sequence of plaques separating $P$ and $Q$, $P=P_0, P_1, \ldots, P_N=Q$, such that 
$$
 \sigmabold_n(P,Q)=\sum_{i=1}^N(h_n^\ast(x_{d_{i-1}}^{+,\ast})-h_n^\ast(x_{d_i}^{-,\ast})) + E^\ast_n\ ,
$$
where $E^\ast_n=O(\exp(-Ct_n))$ for some $C>0$.  Similarly, we have
$$
 \widehat\sigmabold_n(P,Q)=\sum_{i=1}^N(h_n(x_{d_{i-1}}^{+})-h_n(x_{d_i}^{-})) + E_n\ ,
$$
where $E_n=O(\exp(-Ct_n))$.

Each component $d_i$ is associated with  a zero $p_i\in Z(q_n)$. 
By Proposition \ref{prop: harmonic map C1 localizes near zeroes} applied to the map $v_n$ near the zero $ p_i$, there are curves $k_n^\pm$ in $\Fcal_{q_n}^h$ such that the intersection $i(\Fcal_{q_n}^v, k_n)$ is uniformly bounded, and 
$$
\lim_{n\to \infty}\left\{
t_n^{-1}h_n(x^\pm_{d_i})\pm i(\Fcal_{q_n}^v, k_n^\pm)\right\}=0\ .
$$
From the construction of the map $\varphi$ in \textsection{\ref{sec:train-track-conclusion}}, it follows that for the same curves, 
$$
\lim_{n\to \infty}\left\{
t_n^{-1}h_n^\ast(x^{\pm,\ast}_{d_i})\pm i(\Fcal_{q_n}^v, k_n^\pm)\right\}=0\ .
$$
 Since the sums in the expressions for $\sigmabold_n$
 and $\widehat\sigmabold_n$ are finite (independent of $n$), and $P$ and $Q$ range over a finite set, we can satisfy the desired inequality  for any $\delta>0$
if $n$ is sufficiently large.
\end{proof}


Recall the open convex polyhedral cone $\Ccal(\Lambda_n)$ from the end of \textsection{\ref{subsec:ShearingCocycles}}.

\begin{prop} \label{prop:pleated-surface}
For $n$ sufficiently large, $\sigmabold_n\in \Ccal(\Lambda_n)$. Hence, there is a marked hyperbolic surface $S_n$ with shearing cocycle $\sigmabold_n$. 
\end{prop}

\begin{proof}
The constant $C$ defined in \cite[Lemma 6]{Bonahon:shearing} only depends on the combinatorics of the train track supporting the laminations (\emph{cf}.\ \cite[p.\ 26]{Bonahon:shearing}), and therefore may be taken independent of $n$. By Proposition \ref{prop:cocycle-properties} (i), one can choose a sufficiently large finite set $\Pcal'\subset\Pcal(\Lambda_n)$ so that transverse  cocycles are determined by their values on $\Pcal'$. Then
 using Lemma \ref{lem:norm} with $\delta<A_0/2C$, we conclude that 
$$
\Vert \sigmabold_n-\widehat\sigmabold_n\Vert< A_n/2C\ .
$$
for $n$ sufficiently large. The result then follows from \cite{Bonahon:shearing}, Proposition 13 and the proof thereof.
\end{proof}

\subsubsection{Proof of Theorem \ref{thm:bonahon-image}}
It remains to prove the existence of a $\rho_n$-equivariant pleated surface map $\widetilde S_n\to \HBbb^3$. 
Here we copy a construction in \cite{Bonahon:shearing}.
Let $\Pcal'\subset\Pcal(\Lambda_n)$ be a finite collection of plaques. For each $P\in \Pcal'$, define $f_{n,\Pcal'}$ on $P\subset \widetilde S_n\setminus \Lambda_n$,  to be the oriented isometry with 
 the corresponding plaque $P^\ast\subset \HBbb^3$.  
 The complement $\widetilde S_n\setminus \cup_{P\in \Pcal'} P$ consists of a union of wedges. For each wedge $\Sigma$, the boundary consists of two geodesics $g$ and $h$ belonging to plaques in $\Pcal'$. Choose (if necessary) a diagonal $\gamma$ in $\Sigma$ joining opposite endpoints of $g$ and $h$, and map $\gamma$ to the corresponding geodesic in $\HBbb^3$. The diagonal $\gamma$ splits $\Sigma$ into two wedges, 
 and  there is a unique way to extend $f_{n,\Pcal'}$ across these to make a continuous and totally geodesic map $\widetilde S_n\to \HBbb^3$, albeit without any equivariance property. 
 Using Lemma \ref{lem:norm}, as the finite sets $\Pcal'$ exhaust $\Pcal(\Lambda_n)$,
 the $f_{n,\Pcal'}$ converge locally uniformly to a map $f_n$. The fact that $f_n$ is $\rho_n$-equivariant and has shearing cocycle $\sigma_n$ follows as in \cite[proof of Lemmas 14 and 16]{Bonahon:shearing}.

\section{Proofs} \label{sec:proofs}

\subsection{Limiting trees} \label{sec:limit_trees}
 We begin by 
  proving a general result on ``factorization'' of equivariant harmonic maps to $\HBbb^2$ and $\HBbb^3$.   Let $[\rho_n]\in R^o(\Sigma)$.
 Suppose we are given a sequence of pleated surfaces  $\Psf_n=(S_n, f_n, \Lambda_n, \rho_n)$, where $\Lambda_n$ carries a transverse measure.
 Let $v_n:\widetilde X\to \HBbb^2$ denote the lift of the degree one harmonic diffeomorphism $X\to S_n$, the hyperbolic surface underlying $\Psf_n$. We also set $w_n=f_n\circ v_n: \widetilde X\to \HBbb^3$. Note that $w_n$ is $\rho_n$-equivariant, and since $f_n$ is totally geodesic on the complement of $\Lambda_n$, which has measure zero,
  $w_n$ is an $L^2_1$-map with the same energy as $v_n$. 
Finally, as usual, we let $u_n: \widetilde X\to \HBbb^3$ denote the $\rho_n$-equivariant harmonic map.

Let $q_n=\Hopf(u_n)$, $\psi_n=\Hopf(v_n)$. We may assume (after passing to a subsequence) that
$$
\frac{q_n}{\Vert q_n\Vert_1}\lra q\ ,\ \frac{\psi_n}{\Vert\psi_n\Vert_1} \lra \psi\ .
$$
Let $\phi_{\HF}(\Lambda_n)$ denote the  Hubbard-Masur differential whose horizontal measured foliation is measure equivalent to the one corresponding to the lamination $\Lambda_n$ (see \textsection{\ref{sec:measured-foliations}).

\begin{prop} \label{prop:hopf}
Suppose the following hold:\smallskip

\begin{compactenum}[(i)]
\item $\psi\in \SQD^\ast(X)$;\\
\item
$\displaystyle \lim_{n\to \infty}
\frac{\Vert 4\psi_n-\phi_{\HF}(\Lambda_n)\Vert_1}{\Vert\psi_n\Vert_1}=0
$.
\end{compactenum}
\smallskip
Then 
\begin{equation*} \label{eqn:convergence-qd}
\lim_{n\to \infty}\frac{\Vert q_n-\psi_n\Vert_1}{\Vert \psi_n\Vert_1}=0\ .
\end{equation*}
\end{prop}

\begin{lem} \label{lem:energy-estimate}
Under assumption (ii) above, there is a constant $0<c\leq 1$ such that 
$$c\cdot E(v_n)\leq E(u_n)\leq E(v_n)\ .$$
\end{lem}
\begin{proof}
The inequality $E(u_n)\leq E(v_n)$ is automatic, since $u_n$ is an energy minimizer among $\rho_n$-equivariant $L^2_1$-maps, and $E(v_n)=E(w_n)$.
Choose any conformal metric on $X$, and induce a metric on $\widetilde X$. 
By the uniform Lipschitz property of harmonic maps to NPC targets (\emph{cf}.\ \cite[Thm.\ 2.2]{Sch}), 
there is a constant $B$ independent of $j$ such that for any points $p,q\in \widetilde X$,
$$
d_{\HBbb^3}(u_n(p), u_n(q))\leq B\, d_{\widetilde X}(p,q)\cdot E^{1/2}(u_n)\ . 
$$
In particular, for $\gamma\in \pi_1$,
\begin{equation} \label{eqn:lipschitz}
\tau_{\HBbb^3}(\rho_n(\gamma))\leq B\, \ell_X(\gamma)\cdot E^{1/2}(u_n)\ . 
\end{equation}
 By (ii), the laminations $\Lambda_n$ are close to  $\Lambda_{\psi_n}^h$. By Poincar\'e recurrence, we may choose a nondegenerate leaf of $\Fcal^h_{\psi_n}$ and form a closed loop $\gamma$ by adding small segments of the vertical foliation (see \cite[Cor.\ 5.3]{FLP}). The image by $f_n$ of the lift of this loop consists of nearly geodesic segments joined by tiny orthogonal segments, so the length approximates the translation length of the corresponding element in $\Iso(\HBbb^3)$. 
The high energy behavior of $v_n$ (\emph{cf}.\ \cite{Wolf:thesis} and Proposition \ref{prop:realizing-2d})  further
 implies that this length is approximated by the transverse measure to  $\Fcal_{\psi_n}^v$. From this we deduce  the existence of  a constant $c_0>0$ such that
$$
\tau_{\HBbb^3}(\rho_n(\gamma))\geq c_0\cdot i(\gamma,\Fcal^v_{\frac{\psi_n}{\|\psi_n\|}})E(v_n)^{1/2}\ .
$$
We then observe that since $\gamma$ may be chosen to be long and nearly along the leaves of $\Fcal^v_{\psi_n}$, the $X$- and $\frac{|\psi_n|}{\|\psi_n\|}$-lengths of $\gamma$ are comparable.
 The Lemma then  follows from \eqref{eqn:lipschitz}.
\end{proof}

Let 
$(C_n, d_{\HBbb^3})$ be the  closed convex hull of the image of $w_n$ in $\HBbb^3$. 
 We now consider the rescaled metric spaces $\HBbb^2_n:=(\HBbb^2, t_n^{-1}d_{\HBbb^2})$ and $W_n:=(C_n, t_n^{-1}d_{\HBbb^3})$, where $t_n^2=E(v_n)$.
By the uniform Lipschitz property used in the previous proof, $v_n:\widetilde X\to \HBbb^2_n$ has uniform modulus of continuity in the sense of Theorem \ref{thm:KS-compactness}.
 Since the map $f_n: \HBbb^2\to C_n$  is distance nonincreasing. it
  follows that $w_n:\widetilde X\to W_n$ has uniform modulus of continuity  as well.

\begin{lem} \label{lem:gh-convergence}
After possibly passing to a subsequence, we have the following properties:
\begin{compactenum}[(i)]
\item $\HBbb^2_n$ with the isometric action of $\pi_1$ converges in the Gromov-Hausdorff  sense to the $\RBbb$-tree $T_\psi$ dual to the quadratic differential $\psi$, up to scale, and the maps $v_n$ converge to a surjective $\pi_1$-equivariant harmonic map $v:\widetilde X\to T_\psi$;
\item $W_n$ with the isometric action of $\pi_1$ converges in the Gromov-Hausdorff  sense to an $\RBbb$-tree $T$ with $\pi_1$-action whose projective length function is equivalent to the Morgan-Shalen limit of $\{\rho_n\}$.  In particular, the action is minimal. The maps $w_n$ converge to an equivariant map $w:\widetilde X\to T$ of finite energy; 
\item The maps $f_n:\HBbb^2_n\to W_n$ converge to a morphism of trees $f: T_\psi\to T$. There is no folding of edges in $T_\psi$ 
corresponding to  adjacent critical leaves
of the horizontal foliation of $\psi$ meeting at a zero. Moreover,  $w:=f\circ v$.
\end{compactenum}
\end{lem}
Note that the embedding, into the surface, of the graph of critical leaves of the horizontal foliation of $\psi$, induces a natural notion of adjacency of critical leaves.
\begin{proof}
The convergence property in 
item (i) follows from \cite{Bestvina:88,PaulinGromovConv94}, and the harmonicity, surjectivity and convergence to the map is in \cite{WolfT}.
Convergence in  (ii) follows by the construction in Theorem \ref{thm:KS-compactness}.  Note that the result of Lemma \ref{lem:energy-estimate} guarantees that the length function has a well defined nonzero limit and that the resulting map $v$ is nonconstant. 
The fact that the limiting tree is the Morgan-Shalen (minimal) tree follows from Theorem \ref{thm:DDW}. 
Item (iii) can be seen by taking the images of leaves of the horizontal foliation of $\psi$ and using the fact that the pleated maps $f_n$ take leaves of the lamination to geodesics. The last assertion in (iii) is obvious.
\end{proof}

We shall need some further properties of the map $w:\widetilde X\to T$.
\begin{lem} \label{lem:v-map}
Fix $z_0\in \widetilde X$ and $Q\in T$. Then on a sufficiently small disk about $z_0$, the function $z\mapsto d_T(w(z), Q)$ is subharmonic. Moreover, the Hopf differential of $w$ is
well defined and equal to $\psi$. 
\end{lem}

\begin{proof}
We may assume $z_0$ is a point such that the map $p$ folds at $v(z_0)$, since otherwise $f$ is a local isometry, and the result follows since $v$ is harmonic with Hopf differential $\psi$. Choose the disk $U$ such that the image $v(U)$ consists of geodesic segments $e_1, e_2, e_3$ in $T_\psi$ meeting at a vertex. Because $f$ is a folding, we may assume that $f$ maps each $e_i$ isometrically onto corresponding geodesic segments $\bar e_i$ in $T$. Note that since folding cannot occur on edges, hence not on adjacent edges incident to a vertex, and the zero in $U$ is trivalent, we see that the $\bar e_i$ are distinct.

 Either the geodesic segment  $\overline{w(z_0)Q}$ intersects
each $\bar e_i$ in the point $w(z_0)$ (which we call Case 1); or $\overline{w(z_0)Q}$ intersects
  some (and hence only one) $\bar e_i$ in a nondegenerate segment (which we call Case 2). 

Suppose we have Case 1. Then we claim that for lifts $\widetilde Q$ of $Q$ to $T_\psi$,
$$
\overline{v(z_0)\widetilde Q}\cap e_i=\{v(z_0)\}\ .
$$
If this were not the case, let $R$ be a nondegenerate segment in one of the intersections above, and set $\overline R=f(R)\subset T$. Set $R^-=\{v(z_0)\}$, and let $R^+$ denote the other endpoint of $R$. The geodesic $\gamma$ in $T_\psi$ from $R^+$ to $\widetilde Q$  is disjoint from the interior of $R$. Hence, its image $\bar\gamma$ in $T$ is a path from $\overline R^+$ to $Q$ that is disjoint from the interior of $\overline R$. On the other hand, there is another path from $w(z_0)$ to $Q$ that is disjoint from the interior of $\overline R$. This contradicts the fact that $T$ is a tree. 

Since the image of $v$ on $U$ is $e_1\cup e_2\cup e_3$, it follows that
\begin{align*}
d_{T_\psi}(v(z), \widetilde Q)&= d_{T_\psi}(v(z), v(z_0)) + d_{T_\psi}(v(z_0), \widetilde Q)\\
&=  d_{T}(w(z), w(z_0)) + d_{T_\psi}(v(z_0), \widetilde Q)\\
&=  d_{T}(w(z), Q)-d_T(w(z_0), Q) + d_{T_\psi}(v(z_0), \widetilde Q)
\end{align*}
Since $d_{T_\psi}(v(z), \widetilde Q)$ is subharmonic on $U$, so is $d_{T}(w(z), Q)$.

In Case 2, suppose without loss of generality that $\overline{w(z_0)Q}\cap \bar e_1=\overline R$ is a nondegenerate segment. Then for small enough $U$, 
$$
d_T(w(z), Q)=d_T(w(z), \overline R^+)+d_T(\overline R^+, Q)\ .
$$
Now because adjacent edges do not fold, 
$$
d_T(w(z), \overline R^+)= d_{T_\psi}(v(z), R^+)\ ,
$$
and the result follows as above. This proves the first part of the Lemma. 
Taking $Q=w(z_0)$, we see that $d_T(w(z), w(z_0))=d_{T_\psi}(v(z), v(z_0))$. Then the energy densities and  Hopf differentials of $w$ and $v$ must coincide (see \cite[\textsection{1.2} and \textsection{2.3}]{KorSch1}).
 The Lemma is proved. 
\end{proof}

 Again appealing to Theorem \ref{thm:DDW},   the rescaled  convex hulls of the images of the $u_n$ converge to give an equivariant harmonic map $u:\widetilde X\to T$ to the Morgan-Shalen limit. Let us emphasize that since the limiting length function is not abelian, there is a unique $\RBbb$-tree associated to the Morgan-Shalen limit with the given limiting length function. 

\begin{proof}[Proof of Proposition \ref{prop:hopf}]
Applying assumption (ii) provides for the estimates in Lemma \ref{lem:energy-estimate}, and hence the existence of the nontrivial map $w$ in Lemma \ref{lem:gh-convergence}. Assumption (i) is used in proof of Lemma \ref{lem:v-map}. With this understood, 
 consider $D(z)=d_T(u(z), w(z))$, for $z\in \widetilde X$. First, since $u$ and $w$ are equivariant for the same action, $D(z)$ descends to a function on $X$. Next, observe that $D(z)$ is continuous, since both $u$ and $w$ are Lipschitz. Let 
$$\Scal=\{ z\in X\mid D(z)=\max_X D\}\ .$$
 Then $\Scal$ is closed and nonempty. Let
$z_0\in \Scal$. We may assume $D(z_0)>0$, since otherwise $u=w$ and  there is nothing to prove. 
Let $Q\in T$ be the midpoint of  the geodesic  $\overline{u(z_0)w(z_0)}$, 
and choose an open disk $U$ about $z_0$ sufficiently small so that $u(U)\cap w(U)=\emptyset$. Since $T\setminus\{Q\}$ is disconnected and hence $u(U)$, $w(U)$ lie in different components, any path from one image to the other must pass through $Q$. In particular,
$$
D(z)=d_T(u(z), Q)+ d_T(w(z), Q)\ ,\ \text{for all $z\in U$.}
$$
Since distance to a point is a convex function, and harmonic maps pull back convex functions to subharmonic functions, the first term on the right hand side is subharmonic. For $U$ sufficiently small, the second term is also subharmonic by Lemma \ref{lem:v-map}.  Hence, using the strong maximum principle this implies that $D$ is constant on $U$, and so $\Scal$ is open. 
Hence, $D(z)$ is a constant function. 

We claim that $u$ and $w$ have the same Hopf differentials.
Suppose $D=D(z)>0$. Then the claim will follow (as above using the definition in \cite{KorSch1}) by showing that for any $z_0$ there is a small enough neighborhood $U$ about $z_0$ such that
\begin{equation}\label{eqn:equality}
d_T(u(z), u(z_0)) = d_T(w(z), w(z_0))\ ,\ \text{for all $z\in U$.}
 \end{equation}
Let $R\subset T$ denote the edge from $u(z_0)$ to $w(z_0)$. Let 
$$
D^+_u=\{ z\in U \mid u(z)\not\in R\} \quad ,\quad
D^-_u=\{ z\in U \mid u(z)\in R\}\ .
$$
Similarly, we define $D^\pm_w$. Notice that since $D(z)=D$ is constant, $D^+_w=D^-_u$ and $D^-_w=D^+_u$. Hence, for $z\in D^+_u$,
\begin{align*}
D=d_T(u(z), w(z))&= d_T(u(z), u(z_0))+ d(u(z_0), w(z)) \\
&= d(u(z), u(z_0))-d_T(w(z), w(z_0))+D\ ,
\end{align*}
whereas for $z\in D^-_u$,
\begin{align*}
D=d_T(u(z), w(z))&= d_T(w(z), w(z_0))+ d(w(z_0), u(z)) \\
&= d(w(z), w(z_0))-d_T(u(z), u(z_0))+D\ .
\end{align*}
In both cases, the equality \eqref{eqn:equality} holds. This proves that $\psi=q$.

In fact, since the energy densities of $u$ and $w$ agree,
$w$ is also energy minimizing. But
 equivariant harmonic maps to nontrivial $\RBbb$-trees are unique 
(see \cite{Mese}), so that in fact $w=u$.
 Choose $p\neq p'\in \HBbb^2$ to lie on a portion of a leaf
 $\ell\subset\widetilde\Fcal^h_q$ away from the zeroes.
 We assume that the map $f:T_q\to T$ maps the image of 
$\ell$ isometrically onto a geodesic segment in $T$. By the definition of a folding and the dual tree, such an $\ell$ can always be found: one simply takes a preimage in a leaf of an arc in a tree and restricts to a small enough subleaf that is not folded. By Theorems \ref{thm:KS} and \ref{thm:KS-compactness}, we have 
\begin{align}
\begin{split}\label{eqn:harmonic-map-convergence}
E(w_n)^{-1/2}d_{\HBbb^3}(w_n(p), w_n(p'))
=&E(v_n)^{-1/2}d_{\HBbb^3}(v_n(p), v_n(p'))\lra i(\ell,\Fcal_\psi^v)\\
E(u_n)^{-1/2}d_{\HBbb^3}(u_n(p), u_n(p'))&\lra i(\ell,\Fcal_q^v)
\end{split}
\end{align}
and since $q=\psi$, the right hand sides are equal. On the other hand, from Lemma \ref{lem:gh-convergence} and the fact that $w=u$, 
\begin{equation} \label{eqn:maps-converge}
\lim_{n\to\infty} E(v_n)^{-1/2}\left(
d_{\HBbb^3}(w_n(p), w_n(p'))-d_{\HBbb^3}(u_n(p), u_n(p'))
\right) =0
\end{equation}
(recall that $t_n^2=E(v_n)$). Eqs.\ \ref{eqn:harmonic-map-convergence} and \ref{eqn:maps-converge} force
$$
\lim_{n\to\infty} \frac{E(u_n)}{E(v_n)}=1\ ,
$$
which implies 
\begin{equation} \label{eqn:q-limit}
\lim_{n\to\infty} \frac{\Vert q_n\Vert_1}{\Vert \psi_n\Vert_1}=1\ .
\end{equation}
Indeed, this follows because the $u_n$ and $v_n$ are harmonic,
and the energy converges \cite[Thm.\ 3.9]{KorSch2}.
Alternatively, for $v_n$ we have the inequality 
$$
\Vert \psi_n\Vert_1-2\pi(g-1)\leq E(v_n)\leq
\Vert\psi_n\Vert_1+2\pi(g-1) 
$$
(see \cite{EellsWood:76}), and so 
$$
\lim_{n\to\infty} \frac{E(v_n)}{\Vert \psi_n\Vert_1}=1\ .
$$
Similarly, using Theorem \ref{thm:DonaldsonIsomorphism} (iii)
and (iv), eq.\ \ref{eqn:pullback-Phi}, and the asymptotics in
Theorem \ref{thm:ConvergenceOfHiggsPairs}, we also have
$$
\lim_{n\to\infty} \frac{E(u_n)}{\Vert q_n\Vert_1}=1\ .
$$
Hence, \ref{eqn:q-limit}. 
The  Proposition now follows from the fact that $\psi=q$, \eqref{eqn:q-limit}, and the algebraic inequality
$$
\frac{\Vert q_n-\psi_n\Vert_1}{\Vert \psi_n\Vert_1}\leq 
\left| 1-\frac{\Vert q_n\Vert_1}{\Vert\psi_n\Vert_1}\right|
+ \left\Vert \frac{q_n}{\Vert q_n\Vert_1} -\frac{\psi_n}{\Vert \psi_n\Vert_1}\right\Vert_1
$$

\end{proof}

\subsection{Limiting configurations and limits of representations}

\subsubsection{Proof of the Main Theorem}
Part (i) of the Main Theorem is the content of Theorem \ref{thm:bonahon-image}. The harmonic map estimates in Propositions \ref{prop:realizing-3d} and \ref{prop: harmonic map C1 localizes near zeroes}   show that the pleated surfaces $f_n:\widetilde S_n\to \HBbb^3$ are asymptotic to the images of the harmonic maps $u_n:\widetilde X\to \HBbb^3$ in the sense of Definition \ref{def:asymptotic-surfaces}. This is part (ii) of the Theorem. Part (iii) then follows from Lemmas \ref{lem:norm} and \ref{lem:wolf}. Finally, part (iv) is a consequence of the approximation in Definition \ref{def:asymptotic-surfaces} and Theorem \ref{thm:bending-periods}. This completes the proof.

\subsubsection{Proof of Corollary \ref{cor:invariance-of-basepoint}}

Let us first recast Theorem \ref{thm:prym-cocycle} in terms of train tracks.  Let $q\in \QD^\ast(X)$. Let $\tau$ be a  complete train-track (\emph{cf}.\ \cite[p.\ 27, 175]{PenHar}) carrying the horizontal lamination $\Lambda_q^h$. Let $\Hcal^o(\tau,S^1)$ be the connected component of the identity of the space of $S^1$-valued cocycles on $\tau$. Then we have the following 
\begin{thm}
There is a group isomorphism
$$
\Hcal^o(\tau,S^1)\simeq \Prym(\widehat X_q, X)/J_2(X)\ .
$$
\end{thm}

\begin{proof}
Choose a maximalization of $\Fcal_q^h$ in the sense of Definition \ref{def:maximalization}. By Lemma \ref{lem:maximalization}, this gives a maximalization $\Lambda$ of $\Lambda_q^h$. Now $\Lambda$ is carried by the splitting $\tau'$ of $\tau$ corresponding to the maximalization. Since there is a natural isomorphism
$$
\Hcal^o(\tau,S^1)\simeq
\Hcal^o(\tau',S^1)\simeq \Hcal^o(\Lambda, S^1)\ ,
$$
the result follows from Theorem \ref{thm:prym-cocycle}. Note that from the construction leading to Corollary \ref{cor:prym}, the isomorphism is independent of the choice of maximalization.
\end{proof}

Fix $X_0$ and $q_0\in \SQD^\ast(X_0)$.
The train track $\tau$ may be chosen so that for any 
$X\in U_0$, $\tau$ carries $\Lambda_q^h(X)$, where $q\in \SQD^\ast(X)$ is the Hubbard-Masur differential  for the measured foliation $\Fcal_{q_0}^v$. A  maximalization $\Lambda(X)$ of $\Lambda_q^h(X)$ is carried by a splitting of $\tau$.

We now continue with the proof of the Corollary. Let $T$ be a Morgan-Shalen limit of $[\rho_n]$. As we have noted before, by Theorem \ref{thm:DDW} there is an equivariant harmonic map $u: \widetilde X\to T$ that factorizes through $T_{q_0}$. Note that since $q_0$ has simple zeroes, the action on $T$ is  not abelian, and so the tree $T$ is uniquely (up to scale) associated to the projective length function of the Morgan-Shalen limit. 

Consider $X\in U_0$. Then as above, we have an equivariant harmonic map: 
$\widetilde X\to T_{q_X'}\stackrel{v}{\lra} T$.
 We claim that up to an overall scale, $T_{q_X'}$ is equivariantly isometric to $T_{q_0}$. From this, it follows that $q_X'=q_X$. To prove the claim, let $v:\widetilde X_0\to T_{q_X'}$ be the equivariant harmonic map, and set $w=f\circ v : \widetilde X_0\to T$. Then using exactly the same argument as in Lemma \ref{lem:v-map} and the proof of Proposition \ref{prop:hopf}, we conclude that the Hopf differential of $v$ is also $q_0$. Since the action on $T_{q_X'}$ is ``small'', it follows that the folding $T_{q_0}\to T_{q_X'}$ induced by $v$ from Theorem \ref{thm:KS} (iii) is actually an isometry (see \cite[Prop.\ 3.1]{Skora}).  This proves the first statement of the Corollary.

To prove the statement about bending cocyles, 
let $\widehat S_n(X_0)$ be the companion surfaces as in \textsection{\ref{sec:companion}}. Recall the construction of a pleated surface
$\Psf_n(X_0)=(S_n(X_0), \widetilde f_{n,X_0}, \Lambda_n(X_0),\rho_n)$, where the shearing cocycle of the hyperbolic surface $S_n(X_0)$ is obtained as a perturbation of the one for $\widehat S_n(X_0)$.
We may choose the train track $\tau_{n,\ast}$ used in that proof to carry both laminations $\Lambda_n^h(X)$ and $\Lambda_n^h(X_0)$.  By a straightforward energy estimate the scaling factors of the quadratic differentials on $X$ and $X_0$ are comparable: \emph{i.e.}, there is a constant $C$ depending only on $U_0$ so that
$$
C^{-1} t_n(X_0)\leq t_n(X)\leq Ct_n(X_0)\ .
$$
Perturbing the shearing cocycle of $\widehat S_n(X_0)$ as in \eqref{eqn:realization-transverse-cocycle}, but now  with respect to the lamination $\Lambda_n(X)$ instead of $\Lambda_n(X_0)$,
the argument in \textsection{\ref{sec:shearing_realization}} carries over to show that there is a  $\rho_n$-equivariant pleated surface  with pleating locus $\Lambda_n(X)$. By \cite[Lemma 29]{Bonahon:shearing}, this must agree (up to isotopy) with the pleated surface $\Psf_n(X)=(S_n(X), \widetilde f_{n,X}, \Lambda_n(X),\rho_n)$ constructed in Theorem \ref{thm:bonahon-image}  with the base point $X$. 

Now the complementary regions of the lift  $\widetilde\tau_{n,\ast}$
of the train track to $\HBbb^2$ give the identification of the plaques for $\Psf_n(X)$ and $\Psf_n(X_0)$. Each plaque $P$ is realized in two ways (say $P_X$ and $P_{X_0}$) as an ideal triangle in $\HBbb^3$, and where by the asymptotic  estimates  on the harmonic maps $u_n$ (\emph{cf.} Proposition \ref{prop: harmonic map C1 localizes near zeroes}), the triples of leaves in $\Lambda_n(X)$ and in $\Lambda_n(X_0)$ bounding $P$ are close over a large hexagonal region of $P$.  For a pair of plaques $P,Q$, fix points $\widetilde p\in P_X$, $\widetilde p_0\in P_{X_0}$, $\widetilde q\in Q_X$, $\widetilde q_0\in Q_{X_0}$. 
Then as in the proof of Lemma \ref{lem:finite-approximation-bending}, the proximity of the ideal triangles for $\Psf_n(X)$ and $\Psf_n(X_0)$ when $n$ is large gives an estimate on $|\Theta_{\widetilde f_{n,X}}(\widetilde p,\widetilde q)-\Theta_{\widetilde f_{n,X_0}}(\widetilde p_0,\widetilde q_0)|$.
From \eqref{eqn:geometric-estimate}, we see that the bending cocycles $\betabold_n(X_0)$ and $\betabold_n(X)$ give the same limit as a cocycle in $\Hcal^0(\tau,S^1)$. By the Main Theorem, this common limit determines the periods of $\eta_{X_0}$ and $\eta_X$, and so therefore identifies their cohomology classes under the Gauss-Manin connection. This concludes the proof of Corollary \ref{cor:invariance-of-basepoint}.

\subsubsection{Proof of Corollary \ref{cor:morgan-shalen}}
Consider the situation of the Main Theorem. Suppose that the limiting quadratic differential $q$ has a vertical saddle connection between $p,p'\in Z(q)$. 
Let $p_n$, $p_n'$ be zeroes of $q_n$ so that
 $p_n\to p$, $p_n'\to p'$.
If there is a folding in the Morgan-Shalen limit, then the following must happen: there is some $\delta_0>0$ such that for all $0<\delta\leq\delta_0$ there are points $z_n$, $z_n'$ with
\begin{align}
\begin{split} \label{eqn:folding-limits}
t_n^{-1} d_{\HBbb^3}(p_n^\ast, (p_n')^\ast)\lra 0\ ,\quad
t_n^{-1} d_{\HBbb^3}(z_n^\ast, (z_n')^\ast)\lra 0\ ;\\
t_n^{-1} d_{\HBbb^3}(z_n^\ast, p_n^\ast)\lra \delta\ ,\quad
t_n^{-1} d_{\HBbb^3}((z_n')^\ast, (p_n')^\ast)\lra \delta\ .
\end{split}
\end{align}
See Figure \ref{fig:folding}. The notation here means that $p_n^\ast=u_n(p_n)$, etc. See Theorem \ref{thm:ms} and the definition of folding in \textsection{\ref{sec:trees}}. As in the proof of Lemma \ref{lem:finite-approximation-bending}, the planes $\DBbb_{p_n^\ast}$ and $\DBbb_{(p_n')^\ast}$ intersect, and the assumption is that  the  dihedral angle is bounded away from $\pi$. Let $A_n$ denote the geodesic segment between $z_n^\ast$ and $p_n^\ast$,  $B_n$   the geodesic segment between $(z_n')^\ast$ and $p_n^\ast$, and $C_n$ between $z_n^\ast$ and $(z_n')^\ast$, and let $\alpha_n$, $\beta_n$, and $\gamma_n$ be the corresponding angles of the geodesic triangle thus formed.  Then the assumption implies $\gamma_n\geq \varepsilon>0$ for some fixed $\varepsilon$ and $n$ sufficiently large. From \eqref{eqn:folding-limits}, we may assume $|B|\geq \delta t_n/2$. But then
$$
\sinh|C|\geq (\sin\varepsilon) \sinh(\delta t_n/2)\ ,
$$
which contradicts the assumption that $t_n^{-1}|C|\to 0$.
This completes the proof.

\begin{figure}
\begin{tikzpicture}
\draw [thick, dotted] (-3,0) -- (-1,0);
\draw [thick] (-3,-2) -- (-3,2);
\draw [thick] (-1,-2) -- (0-1,2);
\draw [thick] (-5,0) -- (-3,0);
\draw [thick] (-1,0) -- (1,0);
\node at (-3,0) {$\bullet$};
\node at (-1,0) {$\bullet$};
\node at (-4,0) {$|$};
\node at (0,0) {$|$};
\node at (-2.7,.3) {$p_n$};
\node at (-4.3,.3) {$z_n$};
\node at (-1.3,.3) {$p_n'$};
\node at (.3,.3) {$z_n'$};
\node at (3,.2) {$\stackrel{\rm folding}{\xrightarrow{\hspace{2cm}}}$};
\draw [thick] (6,-2) -- (6,2);
\draw [thick] (6,0) -- (7,0);
\draw [thick] (7,0) -- (8,1);
\draw [thick] (7,0) -- (8,-1);
\node at (6,0) {$\bullet$};
\node at (7,0) {$|$};
 \end{tikzpicture}
\caption{Folding}
\label{fig:folding}
\end{figure}

\subsection{Complex projective structures} \label{sect:operfamily}

The goal of this section is to prove Corollary \ref{cor:thurston's-pleated-surface}.  In order to do so, insofar as the pleated surface is already given by Thurston,  it is necessary to in some sense reverse the argument used in \textsection{\ref{sec:realization}}. For this, we use the result of \textsection{\ref{sec:limit_trees}}, which gives
criteria to identify the limiting quadratic differential for equivariant harmonic maps in terms of the lamination of the associated pleated surfaces.
  In \textsection{\ref{sec:dumas}}, we review Dumas' estimates, which show that the criteria just mentioned hold for Thurston's pleated surfaces associated to projective structures. In \textsection{\ref{sec:oper_spectral}}, we use facts about opers to derive the limiting spectral data. Finally, Corollary \ref{cor:thurston's-pleated-surface} is proven in the last section.

\subsubsection{Dumas' estimates} \label{sec:dumas}
  We recall the estimates of Dumas in \cite{Dum2} (see particularly Theorems~1.1 and 14.2) which relate complex projective structures and Hopf differentials.
As mentioned in the Introduction, given $q\in \QD(X)$ the projective connection $\Op(q)$
produces  a pleated surface $\Psf(q)=(S(q), f_q, \Lambda(q), \Pscr(q))$. Moreover, the bending lamination $\Lambda(q)$ carries a transverse measure. Strictly speaking, $\Lambda(q)$ may not be maximal; it will turn out that the  choice of maximalization of $\Lambda(q)$ will be immaterial, and so we supress it from the notation. 

The first result compares $q$ with the Hubbard-Masur differential defined by the lamination\footnote{Because Dumas uses the Schwarzian, the quadratic differential he uses to  parametrize $S(q)$ differs from the one in \eqref{eqn:diff-eq}  by a factor of $-2$.}. 

\begin{thm}[{\cite[Thm.\ 1.1]{Dum2}}] \label{thm: Dumas Projective Structure Estimate}
	There is a constant $C=C(X)$ that only depends on the Riemann surface $X$ so that
\begin{equation}\label{eqn: Dumas estimate}
\Vert
4q - \phi_{\HF}(\Lambda(q))\Vert_1 \leq C(X)\left(1+ \Vert q\Vert_1^{1/2}\right) \ .	
\end{equation}
\end{thm}

The second important result is a comparison of the quadratic differentials parametrizing projective structures and those in the harmonic maps parametrization of Teichm\"uller space. 
 More precisely, Dumas proves the following
({\cite[Thm.\ 14.2 and proof]{Dum2}; see also \cite{Dumas06,Dumas06Err}}).

\begin{thm}
\label{thm: Dumas Grafting and hopf differentials agree}
	Fix $q\in \SQD(X)$, and 
	let $\psi$ denote the Hopf differential of the harmonic diffeomorphism $X\to S(q)$. Then
	\begin{equation} \label{eqn: Dumas estimate on Hopf diffs and prunings}
	\Vert
4\psi(q) - \phi_{\HF}(\Lambda(q))\Vert_1 \leq C(X)\left(1+ \Vert q\Vert_1^{1/2}\right) \ .	
	\end{equation}
\end{thm}

Recall that $\Pscr(q)$ denotes the monodromy of the projective connection $Q(q)$. 
Combining Theorems \ref{thm: Dumas Projective Structure Estimate} and \ref{thm: Dumas Grafting and hopf differentials agree} with Proposition \ref{prop:hopf}, we have
\begin{cor} \label{cor:hopf}
Let $q\in \SQD^\ast(X)$, and let $u_t:\widetilde X\to \HBbb^3$ be the $\Pscr(t^2q)$-equivariant harmonic map with Hopf differential $q_t$. Then
$$
\lim_{t\to+\infty}\Vert t^{-2}q_t -q\Vert_1=0  \ .
$$
\end{cor}

\subsubsection{Spectral data for opers} \label{sec:oper_spectral}
Here we determine the possible limiting bending cocycles  of the family  $\Op(t^2q)$. 
The argument we give is based on the identification of limiting configurations with limiting spectral data, and the classical fact that
the underlying holomorphic bundle $\Vcal$ of a (lift of a) complex projective structure is the unique nonsplit extension
$$
0\lra K_X^{1/2}\lra \Vcal \lra K_X^{-1/2}\lra 0
$$
(\emph{cf.} \cite[p.\ 201]{Gunning:66}).
In terms of Higgs pairs, this means that the $\dbar$-operator 
$\dbar_A+\Phi^\ast$ must induce the holomorphic structure on $\Vcal$. Moreover, since $\Vcal$ has a flat connection,
the holomorphic structure on $\Vcal$ can be uniquely characterized 
by the fact that it contains $K_X^{1/2}$ as a subsheaf, and this is the criterion we shall use. 

Before proceeding, it may clarify to recall once again that by Theorem \ref{thm:DonaldsonIsomorphism} and the definition of the Hitchin map \eqref{eqn:Hitchin-map}, if $u$ is the equivariant harmonic map associated to a solution $(A,\Psi)$ of the self-duality equations, then  $\Hopf(u)=4\Hscr([A,\Psi])$. 
With this understood, we have the following.

\begin{prop} \label{prop:oper-limiting-bending}
Let $q$, $q_t$ be as in the statement of Corollary \ref{cor:hopf}. 
Let $[(A_t, \Psi_t)]=\Op(t^2q)$, and let $\eta_t$ be the term appearing in the approximation in
Definition \ref{def:ConvergenceOfHiggsPairs} and $\widehat\eta_t$ the Prym differential corresponding to $\eta_t$ in Proposition \ref{prop:prym-limiting}. Then
$$
[\widehat\eta_t-it\imag\lambda_{\SW}]\lra 0\quad {\rm  in }\quad \Prym(\widehat X_q,X)/J_2(X)\ .
$$
\end{prop}

\begin{proof}
Let $\widehat q_t=t^{-2}q_t$. 
Consider the spectral curves $\pi_t: \widehat X_{\widehat q_t}\to X$. 
Denote the Seiberg-Witten differential (resp.\ tautological
section) on $\widehat X_{\widehat q_t}$
by $\lambda_{\SW}(t)$ (resp.\ $\lambda_t$). 
Since by Corollary \ref{cor:hopf}, $\widehat q_t\to q$ at $t\to +\infty$, it follows that
 $\lambda_{\SW}(t)\to \lambda_{\SW}$ on $\widehat X_q$, where the convergence is taken with respect to the Gauss-Manin connection on Prym differentials.
 
Let $f_t: \widehat X_{\widehat q_t}\to \widehat X_{q_t/4}$ be as in 
the proof of Proposition \ref{prop:prym-limiting}. Then
$f_t^\ast \lambda_{\SW}=(t/2)\lambda_{\SW}(t)$. 

Now, from the discussion at the beginning of this subsection, 
there is an injective homomorphism of smooth bundles, 
$T_t: \pi_t^\ast K_X^{1/2}\to \pi_t^\ast E$, 
such that the image is preserved by the pullback 
 $\dbar$-operator $\pi_t^\ast(\dbar_{A_t}+\Phi_t^\ast)$, and the induced $\dbar$-operator is 
isomorphic to the canonical one on $\pi_t^\ast K_X^{1/2}$, up to possibly twisting by a $2$-torsion line bundle. As 
a smooth bundle, we have the splitting
$\pi_t^\ast E\simeq \pi_t^\ast K_X^{-1/2}\oplus \pi_t^\ast
K_X^{1/2}$ (see \eqref{eq:decompvbE}). Let $\sigma$ be a local trivialization of $\pi_t^\ast K_X^{1/2}$, and write
$
T_t(\sigma)=(\sigma_t^{(1)}, \sigma_t^{(2)})
$,
where $\sigma_t^{(i)}=\xi_t^{(i)}\sigma$
for a local smooth section $\xi_t^{(1)}\in \Gamma(\pi_t^\ast K_X^{-1})$ and
smooth function $\xi_t^{(2)}$.
By a straightforward calculation, one finds the component entries of
$\pi_t^\ast(\dbar_{A_t}+\Phi^\ast_t)T_t$ to be (after an overall rescaling):
\begin{align}
\begin{split} \label{eqn:approx-limit}
\dbar(\lambda_t\xi_t^{(1)}\cdot \Vert\lambda_t\Vert^{-1/2})
+(\widehat\eta_t''+\tfrac{t}{2}\overline\lambda_{\SW}(t))\xi_t^{(2)}\Vert\lambda_t\Vert^{1/2}=R_1(t) \\
 \dbar(\xi_t^{(2)}\cdot
\Vert\lambda_t\Vert^{1/2})+(\widehat\eta_t''+
\tfrac{t}{2}\overline\lambda_{\SW}(t))\lambda_t\xi_t^{(1)}\Vert\lambda_t\Vert^{-1/2}
= R_2(t)
\end{split}
\end{align}
Here, the remainder terms $R_i(t)$ are linear combinations 
 of  $\lambda_t\xi_t^{(1)}\cdot \Vert\lambda_t\Vert^{-1/2}$ and
$\xi_t^{(2)}\Vert\lambda_t\Vert^{1/2}$ with coefficients
that are exponentially small as $t\to +\infty$. 
To obtain \eqref{eqn:approx-limit}, we use the expression for
the limiting connection  in
\eqref{eq:stdlimconn} to calculate:
$$
\pi_t^\ast(\dbar_{A_\infty^0(q_t)}+\eta_t'')(T_t(\sigma))-T_t(\dbar\sigma)
=\left(\begin{matrix}
\dbar\xi^{(1)}_t-(\dbar\log\Vert\lambda_t\Vert^{1/2})\xi_t^{(1)}
+\widehat\eta_t''\lambda_t^{-1}\Vert\lambda_t\Vert \xi^{(2)}_t
\\
\dbar\xi^{(2)}_t+(\dbar\log\Vert\lambda_t\Vert^{1/2})\xi_t^{(2)}
+\widehat\eta_t''\lambda_t\Vert\lambda_t\Vert^{-1} \xi^{(1)}_t
\end{matrix} \right)\sigma
$$
Similarly, 
$$
\pi_t^\ast(\Phi^\ast(q_t)(T_t(\sigma))=\frac{t}{2}\overline\lambda_{\SW}(t)\left(\begin{matrix}
\lambda_t^{-1}\Vert\lambda_t\Vert \xi^{(2)}_t
\\
\lambda_t\Vert\lambda_t\Vert^{-1} \xi^{(1)}_t
\end{matrix}\right)\sigma
$$
Multiplying the first entries by $\lambda_t\Vert\lambda_t\Vert^{-1/2}$, and the second entries by
$\Vert\lambda_t\Vert^{1/2}$, we obtain the left hand side of \eqref{eqn:approx-limit}.
The error terms come from applying the result in Theorem
\ref{thm:ConvergenceOfHiggsPairs}.

Now suppose to the contrary that there is a sequence $t_n\to+\infty$ and $\beta_n$, odd harmonic $(0,1)$ forms
 with periods in $2\pi i\ZBbb$, 
 such that
$$
\lim_{n\to\infty}\left\{(\eta_{t_n}''+\tfrac{t_n}{2}\overline\lambda_{\SW}(t_n))-\beta_n\right\}= \alpha 
\ ,
$$
where the class of $\alpha$ in the Prym variety is nonzero (see
\textsection{\ref{subsect:prymdifferentials}}). 
Choose an arbitrary base point $z_0\in \widehat X_q$, and redefine
$$
\widetilde \xi_n^{(i)}(z)=\exp(-\int_{z_0}^z\beta_n)\xi_{t_n}^{(i)}(z)\ .
$$
Then \eqref{eqn:approx-limit} becomes
\begin{align}
\begin{split} \label{eqn:approx-limit2}
\dbar(\lambda_{t_n}\widetilde\xi_n^{(1)}\cdot
\Vert\lambda_{t_n}\Vert^{-1/2})+(\widehat\eta_{t_n}''+\tfrac{t_n}{2}\overline\lambda_{\SW}(t_n)-\beta_n)
\widetilde\xi_n^{(2)}\Vert\lambda_{t_n}\Vert^{1/2}=\widetilde
R_1(t_n)\ , \\
\dbar(\widetilde \xi_n^{(2)}\cdot
\Vert\lambda_{t_n}\Vert^{1/2})+(\widehat\eta_{t_n}''+\tfrac{t_n}{2}\overline\lambda_{\SW}(t_n)-\beta_n)\lambda_{t_n}\widetilde
\xi_n^{(1)}\Vert\lambda_{t_n}\Vert^{-1/2} = \widetilde R_2(t_n)\ ,
\end{split}
\end{align}
and where the remainder terms $\widetilde R_i(t_n)$ are exponentially small as  
$t_n\to +\infty$ and of the
order of  $\lambda_{t_n}\widetilde \xi_n^{(1)}
\Vert\lambda_{t_n}\Vert^{-1/2}$ and
$\widetilde \xi_n^{(2)}\Vert\lambda_{t_n}\Vert^{1/2}$. 

Fix a conformal metric on $\widehat X_q$ with area form $\rm dv$.
Let us now normalize the sequence of homomorphisms $T_{t_n}$ so that 
$$
\int_{\widehat X_q}
\Vert \lambda_{t_n}\Vert \left(\Vert \widetilde \xi_n^{(1)}\Vert^2+ |\widetilde
\xi_n^{(2)}|^2\right)\, {\rm dv } =1\ .
$$
By applying elliptic regularity to  \eqref{eqn:approx-limit2}, we may assume
that $\lambda_{t_n}\widetilde\xi_n^{(1)}\Vert\lambda_{t_n}\Vert^{-1/2}\to f_1$ and
$\widetilde \xi_n^{(2)}\Vert\lambda_{t_n}\Vert^{1/2}\to f_2$, for functions $f_i$
satisfying 
$$
\int_{\widehat X_q}
(|f_1|^2+|f_2|^2)\, {\rm dv } =1\ ,
$$
and 
$
\dbar f_1+\alpha f_2=0$, $\dbar f_2+\alpha f_1= 0
$.
This, of course,  implies $(\dbar+\alpha)(f_1+f_2)=0$. If $f_1+f_2\neq
0$, then the holomorphic line bundle $\Lcal$ defined by $\alpha$ has a nonzero holomorphic section
and is therefore trivial. If $f_1+f_2=0$, then $f_1$ is nonzero, and
 $(\dbar -\alpha)f_1=0$; so
$\Lcal^\ast$ is trivial. In either case, this contradicts the assumption. The Proposition is
proved. 
\end{proof}

\subsubsection{Proof of Corollary \ref{cor:thurston's-pleated-surface}}
Fix $q\in \SQD^\ast(X)$. Let  $\widetilde S(t^2q)\to \HBbb^3$ be Thurston's pleated surface associated to the projective connection $Q(t^2q)$ with monodromy $\Pscr(t^2q)$, and with bending lamination $\Lambda(t^2q)$. 
Part (i) of the Corollary follows from Corollary \ref{cor:hopf} and Lemma \ref{lem:wolf}. Part (ii)  is the content of Proposition \ref{prop:oper-limiting-bending}. We now move on to prove part (iii) of the Corollary.
Let $S_t$ denote the companion surface defined by requiring the Hopf differential for the harmonic diffeomorphism $X\to S_t$ to be $t^2q$. 
  
By Theorem~\ref{thm: Dumas Grafting and hopf differentials agree}, $\Lambda(t^2q)$ converges to $\Lambda_q^h$.  In particular, for $t$ sufficiently large, $\Lambda(t^2q)$ is carried by the train track $\tau_{t,\ast}\subset S_t$ constructed in \textsection{\ref{sec:companion}}. By (i) and arguing as in the proof of Corollary \ref{cor:invariance-of-basepoint}, a perturbation of the shearing cocycle of $S_t$ as described in  \textsection{\ref{sec:shearing_realization}}  results in  a pleated surface for $\Pscr(t^2q)$ with bending lamination $\Lambda(t^2q)$; as before by the uniqueness of pleated surfaces for a fixed lamination, the pleated surface constructed in this way must coincide with the pleated surface $S(t^2q)$. 
Now, by Theorem~\ref{thm: Dumas Projective Structure Estimate} and Corollary~\ref{cor:hopf}, the bending laminations for the harmonic map $u_t$ and the bending lamination $\Lambda(t^2q)$ are close and hence carried by the same track $\tau_{t,\ast}\subset S_t$. Thus the plaques for the associated pleated surfaces are also in proximity, in the sense of the last portion of the proof of Corollary~\ref{cor:invariance-of-basepoint}; it then follows that their bending cocycles are also close.
Thus, by the Main Theorem, 
 the  bending cocycle of either is approximated by the one given by the Prym differential associated to the limiting configuration of $\Op(t^2q)$. The result now follows from Lemma \ref{lem:wolf} and Proposition \ref{prop:oper-limiting-bending}.
Notice that since $\imag\lambda_{\SW}$ has zero periods on saddle connections of the horizontal foliation, the choice of a possible maximalization of $\Lambda_q^h$ is irrelevant.

\subsubsection{Refined Estimate}

In section, we refine the estimate Corollary~\ref{cor:hopf} of the previous section and so prove an improvement of Corollary~\ref{cor:thurston's-pleated-surface}. We show the proposition.

\begin{prop} \label{prop: refined error}
	Let $q\in \SQD^\ast(X)$, and let $u_t:\widetilde X\to \HBbb^3$ be the $\Pscr(t^2q)$-equivariant harmonic map with Hopf differential $q_t = \Hopf(u_t)$. Then
$$
\lim_{t\to+\infty}\Vert t^{-2}q_t -q\Vert_1=O(t^{-1})  \ .
$$
\end{prop}

 \begin{proof}

     In outline, the argument begins as previously by using results (Theorems~\ref{thm: Dumas Projective Structure Estimate} and \ref{thm: Dumas Grafting and hopf differentials agree}) of Dumas in \cite{Dumas06} , \cite{Dumas06Err}, and \cite{Dum2} to show that $t^2q$, $\Hopf(v_t)$ and $\phi_{\HF}(\Lambda_t)$ are both of order $O(t^2)$ and within an order of $O(t)$ of one another.  Then, for a properly chosen element $[\gamma] \in \pi_1(\Sigma)$, we estimate hyperbolic translation lengths $\tau_{\HBbb^3}(\rho_t(\gamma))$ in two ways: first as a nearly geodesic path on the bent hyperbolic surface $S(t^2q)$, and next as a nearly geodesic path on the image in $\H^3$ of $X$ by the harmonic map $u_t$.  In both cases, for the arcs we will consider, the images of the arcs are controlled well by the intersection numbers of the arcs with the vertical measured foliations of the Hopf differentials -- this is the content of Propositions~\ref{prop:pullbackmetric} and \ref{prop: harmonic map C1 localizes near zeroes} -- and some hyperbolic geometry then asserts that the common translation length must then be predicted by the intersection numbers.  This forces that those intersection numbers are close for a large family of curve classes, which are enough to in turn imply that the Hopf differentials are within a controlled error of each other.

     The first step is to recall the results of Dumas on the $L^1$ norms of the differences of some quadratic differentials.  
      
     Theorems~~\ref{thm: Dumas Projective Structure Estimate} and \ref{thm: Dumas Grafting and hopf differentials agree},     together show that $$\|\Hopf(X, S(t^2q)) - t^2q\| \leq O(t).$$ In particular, we may begin with the relatively weak estimate $\|\Hopf(X, S(t^2q))\| \asymp t^2$.     On the other hand, we may use Propositions~\ref{prop:pullbackmetric} and \ref{prop: harmonic map C1 localizes near zeroes} to get good control on the images of a robust set of arcs $\gamma \subset \Sigma$. For example, represent $[\gamma] \in \pi_1(\Sigma)$ by a curve $\gamma$ which is quasitransverse to the vertical foliation of  $\Hopf(X, S(t^2q))$, as well as (i) piecewise vertical and horizontal with respect to that differential and also (ii) vertical near the zeroes of that differential:  it is routine that this can be accomplished by simply modifying the geodesic representative of $[\gamma]$ in the flat singular metric defined by $|\Hopf(X, S(t^2q))|$.  
     Then Propositions~\ref{prop:pullbackmetric} and \ref{prop: harmonic map C1 localizes near zeroes} assert that the image of a vertical arc through a zero is nearly a geodesic arc of some fixed positive (and finite) length, and moreover, that the image of $\gamma$ is an arc on $S(t^2q)$ comprising those geodesics of uniformly bounded length meeting at images of zeroes of $v_t$ and connected by arcs which have geodesic curvature at most $O(e^{-ct^2})$ and have length given by
     $$\ell_{S(t^2q)}(v_t(\gamma_{hor})) = i(\gamma, \SF_{\Hopf(X, S(t^2q))}^v) + O(e^{-ct^2}).$$
     Thus, by the Morse Lemma in elementary hyperbolic geometry, because on the hyperbolic surface, the arc $v_t(\gamma)$ comprises long nearly geodesic arcs connected, at angles bounded away from zero, by nearly geodesic arcs of uniformly bounded length, the $S(t^2q)$-geodesic representative of $[\gamma]$ has length given by
     $$\ell_{S(t^2q)}([\gamma])=i(\gamma, \SF_{\Hopf(X, S(t^2q))}^v) + O(1),$$ and moreover, outside neighborhoods of uniform size of the $v_t$-images of the zeroes, lies exponentially (in $t^2$) close to the horizontal geodesic {\it lamination} defined by $\Hopf(X, S(t^2q))$.

     Of course, this is the length on a surface, so if we want to promote this estimate to an estimate of the $\rho_t$-translation length of $[\gamma]$, we need to consider the image $w_t(\gamma)$ after the isometry $f_t$. We will need to worry about curves $\gamma$ which are poorly positioned with respect to a fold, so we now restrict to curve classes $[\gamma]$ which may be represented by polygonal quasi-transverse arcs which also contain no vertical saddle connections; later on, we will see that these represent is a sufficiently diverse collection of elements of the fundamental group that suffice to determine the relevant Hopf differentials.

     Now, by \eqref{eqn: Dumas estimate on Hopf diffs and prunings}, we see that the difference $\|t^{-2}\Hopf(X, S(t^2q)) - t^{-2}\phi_{\HF}(\Lambda(t^2q))\|$ of normalized differentials is of order $O(t^{-1})$.  Thus the corresponding laminations make an increasingly shallow angle with one another, or expressed in a way that is better for our purposes, if $\tau$ is any sufficiently split train track that carries $\Lambda(t^2q))$, then both $v_t(\gamma)$ and the horizontal geodesic {\it lamination} defined by $\Hopf(X, S(t^2q))$ meet $\tau$ at angles comparable to $O(t^{-1})$.  But that train track has an image under $f_t$ that carries the pleating lamination $\Lambda(t^2q))$ as the bending lamination for the pleated surface $S(t^2q))$.

     Now, suppose we are focusing on a curve $\gamma$ and $\gamma$ contains a subarc  $k \subset \gamma$ that connects a pair of zeroes of $\Hopf(X, S(t^2q))$: it is possible that that arc $k$ has bending with respect to the bending cocycle that is not bounded away from $\pi$. If such an arc is purely vertical, then the translation length between the $w_t$-image of its endpoints could be arbitrarily small, and the geodesic in $\H^3$ representing the $\rho_t$-image of $[\gamma]$ might be far away from the $f_t$-image of $v_t(\gamma)$.  On the other hand, if all of the subarcs connecting zeroes of $\Hopf(X, S(t^2q))$ have horizontal segments, then the $f_t$-image ofthose subarcs, since they have been sheared by an amount comparable to $t^2$ along a lamination nearly parallel to $\Lambda_t$, relative to their endpoint, will then be mapped to arcs in $\H^3$ that make only a shallow angle with the bending lamination $\Lambda_t$.

 Thus, in that case, the $v_t$-image $v_t(\gamma)$ of $\gamma$ will, after composition with $f_t$, may be seen to comprise some nearly geodesic arcs of uniformly bounded length arising from the vertical arcs of $\gamma$ near the zeroes of $\Hopf(X, S(t^2q))$ together with some very long, nearly geodesic arcs of length $i(\gamma, \SF_{\Hopf(X, S(t^2q))}^v) + O(e^{-ct^2})$, with only some shallow breaks of angle $O(t^{-1})$ at points far removed from their endpoints where they cross the
bending lamination $\Lambda_t$. 

     Thus, again by hyperbolic geometry and using that the geodesic arcs make only a shallow angle with the lamination, the $\rho_t$-geodesic representative of such an arc $[\gamma]$ has length
     \begin{equation} \label{eqn: wn gamma length}
         \ell_{\H^3}([\gamma])= i(\gamma, \SF_{\Hopf(X, S(t^2q))}^v) + O(1).
     \end{equation}
   (This $\H^3$-geodesic representative of $[\gamma]$ also closely shadows the $w_t$-images of the $\Hopf(X, S(t^2q))$-horizontal portions of the arc $\gamma$, but we will not need that in the sequel.)

     We now turn to the map $u_t$. Again, we can find, for a large collection of curve classes $[\gamma]$, a representative of $[\gamma]$ that is well-positioned with respect to $\Hopf(u_t)$, i.e. it is vertical near the zeroes of $\Hopf(u_t)$, always quasi-transverse to the vertical foliation of $\Hopf(u_t)$, comprising arcs that are alternately horizontal and vertical, and containing no vertical saddle connections. The image of the lift of this curve is, again by Propositions~\ref{prop:pullbackmetric} and \ref{prop: harmonic map C1 localizes near zeroes}, an arc in $\H^3$ comprising images of vertical arcs that are of uniformly bounded length meeting orthogonally images of horizontal arcs that have exponentially small geodesic curvature and length given by
     $$\ell_{\H^3}(u_t(\gamma_{hor})) = i(\gamma, \SF_{\Hopf(u_t)}^v) + O(e^{-ct^2}).$$
     Then, again elementary hyperbolic geometry provides that the geodesic representative of $u_t([\gamma])$ lies close the images of the horizontal segments and has length
     \begin{equation}\label{eqn: un gamma length}
 \ell_{\H^3}([\gamma])=i(\gamma, \SF_{\Hopf(u_t)}^v) + O(1).
 \end{equation}

 We next compare equations \eqref{eqn: wn gamma length} and \eqref{eqn: un gamma length} for curves $\gamma$ that meet our conditions for both $\Hopf{v_t}$ and $\Hopf(u_t)$. We find that the Hopf differentials for $v_t$ and $u_t$ have intersection numbers, with the curve classes $\gamma$ that we have considered that meet the conditions for both holomoprhic differentials, that agree up to $O(1)$.

 We next point out that the collection of these curve classes that meet the conditions for both $\Hopf{v_t}$ and $\Hopf(u_t)$are sufficient for determining the vertical foliations of $\Hopf(v_t)$ and $\Hopf(u_t)$.  Indeed, in the case of $\Hopf(v_t)$, we began the construction of $\gamma$ by considering geodesics in the metric $|\Hopf(v_t)|$, and then adjusting the paths between zeroes.  A subcollection of these initially chosen geodesics provides enough paths between zeroes to provide a triangulation of the surface $\Sigma$, from which the intersection numbers with the arcs suffice to determine the vertical measured foliation of, say, $\Hopf(v_t)$.  When we exclude some curves that contain vertical saddle connections, we will inevitably lose the immediate means to find that those arcs have zero intersection number with the vertical foliation.  To recover such information, we begin with such an arc and follow a horizontal leaf on the surface until it returns to the vertical arc, near its initial point. By either doing surgery on the original curve by adding this long horizontal segment to vertical arcs that emanate from the original pair of zeroes or, alternatively, just joining the endpoints of this horizontal arc, we find two simple curves whose intersection numbers, together with the intersection numbers obtained from the other curves in our distinguished class, determine the vertical measured foliation of $\Hopf(v_t)$. We undertake a similar process for choices of curve classes for $\Hopf(u_t)$.

 It remains to compare equations \eqref{eqn: wn gamma length} and \eqref{eqn: un gamma length}. There is an obvious issue to address as these equations apparently refer to collections of curves that are defined independently. On the other hand, it is possible to find curve classes as in the previous paragraph that are simultaneously in the distinguished classes for both $\Hopf(v_t)$ and $\Hopf(u_t)$.  The easiest cases in which to see this are when $\Hopf(v_t)$ and $\Hopf(u_t)$ are projectively equal, in which case the assertion just follows from the construction in the previous paragraph, and when they are transverse. In that latter case, we note that it is possible to realize both vertical foliations as horizontal and vertical foliations of a quadratic differential on the same surface. Then on that surface, we again triangulate the surface using arcs that are saddle connections for neither foliation, replacing any original choice of arc with a curve as in the previous paragraph, this time chosen to be at some angle with respect to both foliations. In particular, for a saddle connection, we remove a small subarc of the saddle connection and replace it with a long arc from the other foliation: the resulting arc may be replaced by an arc transverse to both foliations. The result of all these constructions is a collection of arcs which can be assembled into a collection of curves whose intersection numbers determine both $\Hopf(v_t)$ and $\Hopf(u_t)$. In the case where the vertical foliations agree on some subsurface and are transverse on another, we apply the two cases just discussed on each subsurface.  In that case, we also replace any vertical subsurface boundary leaf with pairs of curves with no vertical saddle connections as described earlier.

 Comparing equations \eqref{eqn: wn gamma length} and \eqref{eqn: un gamma length} for the common collection of curve classes shows that
 $$
 i(\gamma, \SF_{\Hopf(u_t)}^v)= i(\gamma, \SF_{\Hopf(X, S(t^2q))}^v)+ O(1).
 $$

 Thus, since intersection numbers for a quadratic differential are computed as integrals involving $\sqrt{q}$, and our relation above holds for a class of curves whose intersection numbers may determine the quadratic differential, we see that the normalized Hopf differentials $\Hopf(v_t)$ and $\Hopf(u_t)$ agree up to O(1).
 Thus, applying \eqref{eqn: Dumas estimate} and \eqref{eqn: Dumas estimate on Hopf diffs and prunings}, we see that $\|\Hopf(u_t) - t^2q\|\leq O(t)$.
 \end{proof}

\bibliographystyle{amsalpha}
\bibliography{biblioLimitConfigPleated.bib}

\end{document}